\documentclass{amsart}  

\usepackage{pstricks, pst-tree, graphicx, amssymb, amsfonts, amsmath}
\usepackage[all]{xy}

\newtheorem{theorem}{Theorem}[section]
\newtheorem{lemma}[theorem]{Lemma}
\newtheorem{proposition}[theorem]{Proposition}
\newtheorem{corollary}[theorem]{Corollary}

\theoremstyle{definition}
\newtheorem{definition}[theorem]{Definition}
\newtheorem{example}[theorem]{Example}

\theoremstyle{remark}
\newtheorem{remark}[theorem]{Remark}

\newtheorem*{acknowledgments}{Acknowledgments}
\numberwithin{equation}{section}

\newcommand{\CA}{C_*\hat{\mathcal A}}
\newcommand{\QA}{Q_*\hat{\mathcal A}}
\newcommand{\Ca}[1]{C_{#1}\hat{\mathcal A}}
\newcommand{\Qa}[1]{Q_{#1}\hat{\mathcal A}}
\newcommand{\CAxy}{C_*\hat{\mathcal A}^{\mathbf{x}}_y}
\newcommand{\QAxy}{Q_*\hat{\mathcal A}^{\mathbf{x}}_y}
\newcommand{\Caxy}[1]{C_{#1}\hat{\mathcal A}^{\mathbf{x}}_y}

\newcommand{\q}{q^{\mathbf{x}}_y}
\newcommand{\twig}[1] {\begin{pspicture}(0,#1)(0.4,0.2) \rput(0.2,0){$\leadsto$} \end{pspicture}}

\newcommand{\Pos}{\oplus}
\newcommand{\Neg}{\ominus}
\newcommand{\ost}{\omega^{\text{std}}}

\newcommand{\Z}{\mathbb Z }
\newcommand{\A}{T}
\newcommand{\sgn}{\text{sgn}}
\newcommand{\fcan}{f_\circlearrowright}
\newcommand{\M}{M}
\newcommand{\I}{I}
\newcommand{\T}{^{\text{th}}}

\newcommand{\iia}[2]{\resizebox{!}{.6cm}{ \begin{pspicture}(0,0.7)(2,2.2) \pscircle*(1,1){.15}  \psline[linewidth=3pt](0,1)(2,1) \psline(.6,1)(.4,1.6) \psline(.2,1)(0,1.6) \rput(.3,.7){$#1$} \rput(.8,.7){$#2$} \end{pspicture} } }
\newcommand{\iib}[2]{\resizebox{!}{.6cm}{ \begin{pspicture}(0,0.7)(2,2.2) \pscircle*(1,1){.15}  \psline[linewidth=3pt](0,1)(2,1) \psline(.5,1)(.5,1.4) \psline(.5,1.4)(0.1,1.8)\psline(.5,1.4)(0.9,1.8) \rput(.2,1.2){$#1$} \rput(.7,.7){$#2$} \end{pspicture} } }
\newcommand{\iic}[2]{\resizebox{!}{.6cm}{  \begin{pspicture}(0,0.7)(2,2.2)  \pscircle*(1,1){.15}  \psline[linewidth=3pt](0,1)(2,1) \psline(1.5,1)(1.5,1.4) \psline(1.5,1.4)(1.1,1.8)\psline(1.5,1.4)(1.9,1.8) \rput(1.8,1.2){$#1$} \rput(1.3,.7){$#2$} \end{pspicture} } }
\newcommand{\iid}[2]{\resizebox{!}{.6cm}{  \begin{pspicture}(0,0.7)(2,2.2) \pscircle*(1,1){.15}  \psline[linewidth=3pt](0,1)(2,1) \psline(1.4,1)(1.6,1.6) \psline(1.8,1)(2,1.6) \rput(1.3,.7){$#1$} \rput(1.7,.7){$#2$} \end{pspicture} } }
\newcommand{\iie}[2]{\resizebox{!}{.6cm}{  \begin{pspicture}(0,0.7)(2,2.2) \pscircle*(1,1){.15}  \psline[linewidth=3pt](0,1)(2,1) \psline(.5,1)(.3,1.6) \psline(1.5,1)(1.7,1.6) \rput(.75,.7){$#1$} \rput(1.25,.7){$#2$} \end{pspicture} } }
\newcommand{\iif}[1]{\resizebox{!}{.6cm}{ \begin{pspicture}(0,0.7)(2,2.2)  \pscircle*(1,1){.15}  \psline[linewidth=3pt](0,1)(2,1) \psline(.5,1)(.8,1.6) \psline(.5,1)(.2,1.6) \rput(.75,.7){$#1$}  \end{pspicture} } }
\newcommand{\iig}[1]{\resizebox{!}{.6cm}{ \begin{pspicture}(0,0.7)(2,2.2) \pscircle*(1,1){.15}  \psline[linewidth=3pt](0,1)(2,1) \psline(1,1)(1,1.5) \psline(1,1.5)(0.6,1.9)\psline(1,1.5)(1.4,1.9) \rput(.7,1.3){$#1$}  \end{pspicture} } }
\newcommand{\iih}[1]{\resizebox{!}{.6cm}{ \begin{pspicture}(0,0.7)(2,2.2) \pscircle*(1,1){.15}  \psline[linewidth=3pt](0,1)(2,1) \psline(1.5,1)(1.8,1.6) \psline(1.5,1)(1.2,1.6) \rput(1.25,.7){$#1$}  \end{pspicture} } }
\newcommand{\iii}[1]{\resizebox{!}{.6cm}{ \begin{pspicture}(0,0.7)(2,2.2) \pscircle*(1,1){.15} \psline[linewidth=3pt](0,1)(2,1) \psline(1,1)(1,1.6) \psline(1.5,1)(1.7,1.6) \rput(1.3,.7){$#1$}  \end{pspicture} } }
\newcommand{\iij}[1]{\resizebox{!}{.6cm}{ \begin{pspicture}(0,0.7)(2,2.2) \pscircle*(1,1){.15}  \psline[linewidth=3pt](0,1)(2,1) \psline(1,1)(1,1.6) \psline(.5,1)(.3,1.6) \rput(.7,.7){$#1$}  \end{pspicture} } }
\newcommand{\iik}{\resizebox{!}{.6cm}{  \begin{pspicture}(0,0.7)(2,2.2) \pscircle*(1,1){.15}  \psline[linewidth=3pt](0,1)(2,1) \psline(1,1)(1.3,1.6) \psline(1,1)(.7,1.6)  \end{pspicture} } }

\begin{document}

\title{Tensor Products of $A_\infty$-algebras with Homotopy Inner Products}
\author[Thomas~Tradler]{Thomas~Tradler$^{1}$}
  \address{Thomas Tradler,
  Department of Mathematics, College of Technology, City University of New York, 300 Jay Street, Brooklyn, NY 11201, USA}
  \email{ttradler@citytech.cuny.edu}
 \thanks{$^{1}$ This research funded in part by the PSC-CUNY grant PSCREG-41-316.} 
  
\author[Ronald~Umble]{Ronald~Umble$^{2}$}
\address{Department of Mathematics\\
Millersville University of Pennsylvania\\
Millersville, PA. 17551}
\email{ron.umble@millersville.edu}
\thanks{$^{2}$ This research funded in part by a Millersville University faculty
research grant.}  
 
\begin{abstract}
We show that the tensor product of two cyclic $A_\infty$-algebras is, in general, not a cyclic $A_\infty$-algebra, but an $A_\infty$-algebra with homotopy inner product. More precisely, following Markl and Shnider in \cite{MS}, we construct an explicit combinatorial diagonal on the pairahedra, which are contractible polytopes controlling the combinatorial structure of an $A_\infty$-algebra with homotopy inner products, and use it to define a categorically closed tensor product. A cyclic $A_\infty$-algebra can be thought of as an $A_\infty$-algebra with homotopy inner products whose higher inner products are trivial.  However, the higher inner products on the tensor product of cyclic $A_\infty$-algebras are not necessarily trivial. 
\end{abstract}
\date{submitted August 25, 2011; revised February 9, 2012.}
 
\maketitle
\allowdisplaybreaks

\section{Introduction}

Let $R$ be a commutative ring with unity and let $C_{\ast}\left(  K\right)  $
denote the cellular chains of associahedra $K=\sqcup_{n\geq2}K_{n}$, see \cite{S}.
Identify $C_{\ast}\left(  K\right)  $ with the $A_{\infty}$-operad
$\mathcal{A}_{\infty}$ and consider the Saneblidze-Umble (S-U) diagonal
$\Delta_{K}:C_{\ast}\left(  K\right)  \rightarrow  C_{\ast}\left(  K\right)
\otimes C_{\ast}\left(  K\right)  $ \cite{SU}. Given $A_{\infty}$-algebras
$(A,\mu_{n})_{n\geq1}$ and $(B,\nu_{n})_{n\geq1}$ over $R$, represent $A$ and
$B$ as algebras over $\mathcal{A}_{\infty}$ via operadic maps $\left\{
\zeta_{A}:C_{\ast}\left(  K_{n}\right)  \rightarrow Hom\left(  A^{\otimes
n},A\right)  \right\}  _{n\geq1}$ and $\left\{  \zeta_{B}:C_{\ast}\left(
K_{n}\right)  \rightarrow Hom\left(  B^{\otimes n},B\right)  \right\}
_{n\geq1}$. Define $\varphi_{1}=\mu_{1}\otimes\mathbf{1+1\otimes}\nu_{1}$ and
\[
\varphi_{n}=\left[  \left(  \zeta_{A}\otimes\zeta_{B}\right)  \Delta
_{K}\left(  e^{n-2}\right)  \right]  \sigma_{n},
\]
where $\sigma_{n}:\left(  A\otimes B\right)  ^{\otimes n}\rightarrow
A^{\otimes n}\otimes B^{\otimes n}$ is the canonical permutation of tensor
factors and $e^{n-2}$ denotes the top dimensional cell of $K_{n}$. Then
$\left(  A\otimes B,\varphi_{n}\right)  _{n\geq1}$ is an $A_{\infty}$-algebra,
and for example,%
\[
\varphi_{2}=\left(  \mu_{2}\otimes\nu_{2}\right)  \sigma_{2}\text{ \ and
\ }\varphi_{3}=\left[  \mu_{2}(\mu_{2}\otimes\mathbf{1})\otimes\nu_{3}+\mu
_{3}\otimes\nu_{2}(\mathbf{1}\otimes\nu_{2})\right]  \sigma_{3}.
\]

An $A_{\infty}$-algebra $(V,\rho_{n})$ is \emph{cyclic} if $V$ is equipped with a
cyclically invariant inner product $\left\langle -,-\right\rangle _{V},$ \emph{i.e.},
$\left\langle \rho_{n}\left(  a_{1},\ldots,a_{n}\right)  ,a_{n+1}\right\rangle
_{V}=(-1)^\epsilon\left\langle \rho_{n}\left(  a_{2},\ldots,a_{n+1}\right)  ,a_{1}%
\right\rangle _{V}$. Thus if $(A,\mu_{n})$ and $(B,\nu_{n})$ are cyclic, it is
natural to ask whether $\left\langle -,-\right\rangle _{A}$ and $\left\langle
-,-\right\rangle _{B}$ induce a cyclically invariant inner product on
$\left(  A\otimes B,\varphi_{n}\right)  $. As a first approximation, consider
the inner product%
\[
\left\langle a_{1}|b_{1},a_{2}|b_{2}\right\rangle _{A\otimes B}=(-1)^{|a_2||b_1|}\left\langle
a_{1},a_{2}\right\rangle _{A}\left\langle b_{1},b_{2}\right\rangle _{B},
\]
which respects $\varphi_{1}$ and $\varphi_{2}$ but not $\varphi_{3}$ since%
\[
\left\langle \varphi_{3}(a_{1}|b_{1},a_{2}|b_{2},a_{3}|b_{3}),a_{4}%
|b_{4}\right\rangle -(-1)^\epsilon\left\langle \varphi_{3}(a_{2}|b_{2},a_{3}|b_{3}%
,a_{4}|b_{4}),a_{1}|b_{1}\right\rangle\neq 0. 
\]
However, there is a chain homotopy $\varrho_{2,0}:(A\otimes B)^{\otimes 4}\to R$ such that
\begin{multline*}
(\varrho_{2,0}\circ d)(a_{1}%
|b_{1},a_{2}|b_{2},a_{3}|b_{3},a_{4}|b_{4})= \\ 
\quad\quad\quad\quad\quad\left\langle \varphi_{3}(a_{1}|b_{1},a_{2}|b_{2},a_{3}|b_{3}),a_{4}%
|b_{4}\right\rangle -(-1)^\epsilon\left\langle \varphi_{3}(a_{2}|b_{2},a_{3}|b_{3}%
,a_{4}|b_{4}),a_{1}|b_{1}\right\rangle ,
\end{multline*}
where $d$ is the linear extension of $\varphi_{1}$ and 
\[
\varrho_{2,0}(x_{1}|y_{1},x_{2}|y_{2},x_{3}|y_{3},x_{4}|y_{4})=\pm\left\langle \mu
_{3}(x_{1},x_{2},x_{3}),x_{4}\right\rangle _{A}\left\langle y_{1},\nu
_{3}(y_{2},y_{3},y_{4})\right\rangle _{B}.
\]
Thus $\left\langle -,-\right\rangle
_{A\otimes B}$ respects $\varphi_3$ up to a homotopy, and there is hope.

In \cite{T}, the first author defined the notion of an $A_{\infty}%
$-\emph{algebra with} \emph{homotopy inner product,} which is an $A_{\infty}%
$-algebra\emph{\ }$\left(  A,\mu_{n}\right)  $ together with
\textquotedblleft compatible\textquotedblright\ families of module maps
\[
\left\{  \lambda_{j,k}:A^{\otimes j}\otimes A\otimes A^{\otimes k}\rightarrow
A\right\}  _{j,k\geq0}%
\]
and higher inner products
\[
\left\{  \varrho_{j,k}:A\otimes A^{\otimes j}\otimes A\otimes A^{\otimes k}\rightarrow R\right\}  _{j,k\geq0}.
\]
Homotopy inner products appear in
actions on moduli spaces (\emph{e.g.} \cite{T2, TZ}), in the deformation
theory of inner products \cite{TT}, and in symplectic structures on formal
non-commutative supermanifolds \cite{C}.

Following the construction of Markl 
and Shnider in \cite{MS}, we extend the inner product $\left\langle
-,-\right\rangle _{A\otimes B}$ defined above to a homotopy inner product on $\left(  A\otimes
B,\varphi_{n}\right)$ as follows:  
\medskip
\begin{enumerate}
\item We identify the cellular chains of the associahedra and pairahedra
with a three-colored operad $\CA$ (the pairahedra,
constructed by the first author in \cite{T}, encode the relations among $%
A_{\infty }$-operations, module maps, and homotopy inner products).
\medskip
\item We adapt Boardman and Vogt's $W$-construction \cite{BV} 
to obtain the three-colored operad $\QA$, which is a cubical decomposition of $\CA$.  
\medskip
\item We define a quasi-invertible subdivision map $q:\CA \rightarrow \QA$ 
whose quasi-inverse $p:\CA \rightarrow \QA$ is
defined in terms of an appropriate partial ordering on binary planar diagrams.
\medskip
\item The Serre
diagonal on cellular chains of the $n$-cube induces a coassociative diagonal 
$\Delta _{Q}:\QA \rightarrow \QA \otimes \QA$, which in turn induces a
non-coassociative diagonal
\[
\Delta_C :\CA \overset{q}{\longrightarrow }\QA \overset{\Delta _{Q}}{\longrightarrow }\QA
\otimes \QA \overset{p\otimes p}{\longrightarrow}\CA \otimes \CA.
\]

\item We represent the homotopy inner products on $A$ and $B$ as operadic maps $\psi
_{A}:\CA\rightarrow \mathcal{E}nd_{A}$ and $\psi
_{B}:\CA\rightarrow \mathcal{E}nd_{B}$, and define
\[
\hspace{.2in}\psi _{A\otimes B}:\CA\overset{\Delta_C }{
\longrightarrow }\CA\otimes \CA\overset{\psi _{A}\otimes \psi _{B}}{\longrightarrow }\mathcal{E}
nd_{A}\otimes \mathcal{E}nd_{B}\overset{\approx }{\longrightarrow }
\mathcal{E}nd_{A\otimes B}.
\]
\end{enumerate}

The paper is organized as follows. Section \ref{SEC:CA-QA} defines the three-colored operads $\CA$ and $\QA$. The map $q:\CA\to \QA$ is defined in Section \ref{q-sec} and a quasi-inverse $p$ is defined in Section \ref{p-sec}. Section \ref{SEC:diag} introduces the diagonal $\Delta_C$, and we conclude the paper with some computations in Section \ref{SEC:computations}. To maximize accessibility, longer proofs and other technical considerations are collected in the appendices.  Signs are discussed in Appendix \ref{APP:signs}; the contractibility of pairahedra is proved in Appendix \ref{I_k,l k,l>0}; the fact that relation ``$\leq$'' in the definition of the morphism $p$ is a partial ordering is established in Appendix \ref{APP:partial-order}; and the fact that $p$ is a chain map is proved in Appendix \ref{p-chain-proof}.

\section{The three-colored operads $\CA$ and $\QA$}\label{SEC:CA-QA}

In this section we construct the three-colored operads $\CA$ and $\QA$. Algebras over $\CA$ are $A_{\infty}$-algebras with homotopy inner products and
$\QA$ is a cubical subdivision of $\CA$. Both operads are defined in terms of three types of planar diagrams:\ (1) \emph{planar trees}, which encode the homotopy associativity structure, (2) \emph{module trees}, which encode the
homotopy bimodule structure, and (3) \emph{inner product diagrams}, which encode the
homotopy inner product structure. We shall refer to a diagram in any of these
categories as a \emph{planar diagram}. The term \emph{leaf} of a planar
diagram will always refer to an external inward directed edge; the term
\emph{root} will always refer to the single external outward directed edge of a tree; 
and the term \emph{edge} will always refer to an internal edge, which is
neither a root nor a leaf.

Planar diagrams are generated by three families of corollas: $\left\{
T_{n}\right\}  _{n\geq2},$\linebreak $\left\{  M_{j,k}\right\}  _{j+k\geq0},$ and
$\left\{  I_{j,k}\right\}  _{j+k\geq0}.$  The root and
leaves of a corolla have one of three colors: \emph{thick, thin}, or \emph{empty}. 
The corolla $T_{n}$ has $n$ thin
leaves and a thin root; $M_{j,k}$ has a thick vertical root, a thick vertical
leaf, $j$ thin leaves in the left half-plane, and $k$ thin leaves in the right
half-plane; and $I_{j,k}$ has an empty root, which is graphically represented by a thick vertex, two thick horizontal leaves, $j$
thin leaves in the upper half-plane, and $k$ thin leaves in the lower
half-plane, all meeting at the thick vertex. An example of each type appears in Figure \ref{F:3-corollas}.
\medskip
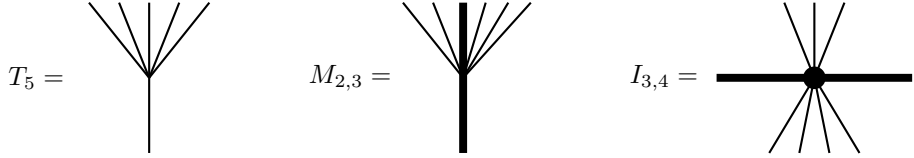
\begin{figure}[h]
\begin{pspicture}(0,1)(3,3)
 \rput(.5,2){$\A_5=$}
 \psline(2,2)(1.2,3)
 \psline(2,2)(1.6,3)
 \psline(2,1)(2,3)
 \psline(2,2)(2.4,3)
 \psline(2,2)(2.8,3)
\end{pspicture}
\quad\quad\quad
\begin{pspicture}(0,1)(3,3)
 \rput(.5,2){$\M_{2,3}=$}
 \psline(2,2)(1.2,3)
 \psline(2,2)(1.6,3)
 \psline(2,2)(2.3,3)
 \psline(2,2)(2.6,3)
 \psline(2,2)(2.9,3)
 \psline[linewidth=3pt](2,1)(2,3)
\end{pspicture}
\quad\quad\quad
\begin{pspicture}(-.5,1)(3.5,3)
 \rput(0,2){$\I_{3,4}=$}
 \psline(2,2)(1.6,3)
 \psline(2,2)(2,3)
 \psline(2,2)(2.4,3)
 \psline(2,2)(1.4,1)
 \psline(2,2)(1.8,1)
 \psline(2,2)(2.2,1)
 \psline(2,2)(2.6,1)
 \psline[linewidth=3pt](.7,2)(3.3,2)
 \pscircle*(2,2){.15}
\end{pspicture}
\caption{Three types of corollas}\label{F:3-corollas}
\end{figure}

A planar tree $T$ is composed
with a diagram $D$ by attaching the root of $T$ to a thin
leaf of $D$. Two module trees are composed by attaching the thick root of one
to the thick leaf of the other. A module tree $M$ is composed with an inner product 
diagram $I$ by attaching the thick root of $M$ to a thick leaf of $I.$ Two
inner product diagrams cannot be composed. A planar diagram resulting from each of
the various compositions appears in Figure \ref{F:3-planar}.
\medskip
\begin{figure}[h]
\resizebox{!}{3cm}{
\begin{pspicture}(0,1)(4,5)
 \psline(2,2)(1.2,3)
 \psline(2,2)(1.6,3)
 \psline(2,1)(2,3)
 \psline(2,2)(2.4,3)
 \psline(2,2)(2.8,3)
 \psline(1.2,3)(1.0,4)  \psline(1.2,3)(1.4,4)
 \psline(2.4,3)(2.2,4)  \psline(2.4,3)(2.6,4)  \psline(2.4,3)(3,4)  \psline(2.4,3)(1.8,4)
 \psline(2.2,4)(2.2,5)  \psline(2.2,4)(1.8,5)  \psline(2.2,4)(2.6,5)
\end{pspicture} }
\quad\quad\quad
\resizebox{!}{3cm}{
\begin{pspicture}(0,1)(4,5)
 \psline(2,2)(1.2,3)
 \psline(2,2)(0.8,3)
 \psline(2,2)(2.8,3)
 \psline(2,2)(3.2,3)
 \psline(2,2)(0.4,3)
 \psline[linewidth=3pt](2,1)(2,5)
 \psline(0.8,3)(0.6,4) \psline(0.8,3)(1.0,4)
 \psline(0.6,4)(0.5,5) \psline(0.6,4)(0.7,5) \psline(0.6,4)(0.9,5) \psline(0.6,4)(0.3,5)
 \psline(2,3)(1.7,4) \psline(2,3)(2.3,4)  \psline(2,3)(2.6,4)
 \psline(3.2,3)(3,4) \psline(3.2,3)(3.4,4)
 \psline(2.6,4)(2.4,5) \psline(2.6,4)(2.8,5)
 \psline(2,4)(1.7,5) \psline(2,4)(1.4,5) \psline(2,4)(1.1,5)
\end{pspicture} }
\quad\quad\quad
\resizebox{!}{3.75cm}{
\begin{pspicture}(0,0)(4,5)
 \psline(2,2)(1.6,3)
 \psline(2,2)(2,3)
 \psline(2,2)(2.4,3)
 \psline(2,2)(1.4,1)
 \psline(2,2)(1.8,1)
 \psline(2,2)(2.2,1)
 \psline(2,2)(2.6,1)
 \psline[linewidth=3pt](0,2)(4,2)
 \pscircle*(2,2){.15}
 \psline(2.5,2)(3.2,1.1) \psline(2.5,2)(3.2,1.4) \psline(2.5,2)(3.2,1.7) \psline(2.5,2)(3.2,2.3)
 \psline(2.5,2)(3.2,2.6) \psline(3.2,1.4)(4,1.3) \psline(3.2,1.4)(3.9,1.0)
 \psline(1.8,1)(1.8,0) \psline(1.8,1)(1.4,0) \psline(1.8,1)(2.2,0)
 \psline(1.5,2)(0.8,2.3) \psline(1.5,2)(0.8,2.6)  \psline(1.5,2)(0.8,2.9)
 \psline(0.7,2)(0,1.6) \psline(0.7,2)(0,1.2) \psline(0.7,2)(0,2.4)
 \psline(0.8,2.9)(0.2,3.4) \psline(0.8,2.9)(0.5,3.6)
 \psline(2.4,3)(2.2,4) \psline(2.4,3)(2.6,4)
 \psline(2.2,4)(2.0,5) \psline(2.2,4)(2.4,5) \psline(2.2,4)(2.8,5) \psline(2.2,4)(1.6,5)
\end{pspicture} }
\caption{Three types of planar diagrams (a tree diagram $T$, a module diagram $M$, an inner product diagram $I$)}\label{F:3-planar}
\end{figure}
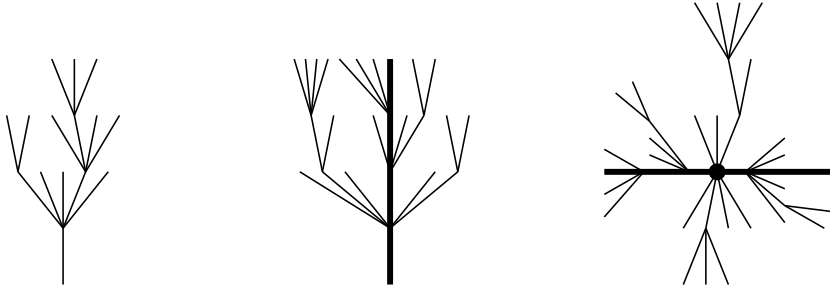

Given a planar diagram $D$, let $\mathcal L(D)$ denote the set of leaves of $D$. If $k=\#\mathcal L(D)$, we use a bijection $f:\{1,\dots, k\}\to \mathcal L(D)$ to label the elements of $\mathcal L(D)$.  In particular, $\fcan$ will denote the ``canonical'' labeling bijection, which numbers the leaves of a (planar or module) tree sequentially from left-to-right and the leaves of an inner product diagram clockwise starting with the left thick leaf (see Figure \ref{FIG:canon-leaf-order}). 
Let $\mathcal E(D)$ denote the set of edges of a planar diagram $D$. 
\begin{figure}[ht]
\[
\resizebox{!}{2.3cm}{
\begin{pspicture}(.5,.5)(4,3.5)
 \psline(2,2)(1.2,3) \psline(1.5,2.6)(1.6,3) \psline(2,1)(2,3)
 \psline(2,2)(3,3) \psline(2.4,2.4)(2.2,3) \psline(2.4,2.4)(2.6,3)
 \psline(2.5,2.7)(2.4,3)  \psline(2.5,2.7)(2.8,3)
 \rput(2.1,3.2){$1\,\,\, 2 \,\,\,\,\, 3 \, 4\, 5 6 7 8$}
\end{pspicture}}
\quad \quad \quad 
\resizebox{!}{2.3cm}{
\begin{pspicture}(-.5,.5)(5,3.5)
 \psline(1.3,2)(1,3) \psline(1.15,2.5)(1.3,3)  \psline(2,2)(1.7,3) 
 \psline(2,2)(2.3,3) \psline(2.15,2.5)(2,3) 
 \psline(2.4,2)(2.7,3) \psline(2.4,2)(3,3) \psline(1.3,2)(1,1)
 \psline(2,2)(2,1) \psline(2,1.5)(1.6,1) \psline(2,1.5)(2.4,1)
 \psline(2.4,2)(2.7,1) \psline(2.8,2)(3.2,1.4)
 \psline(3.2,1.4)(3.2,1)  \psline(3.2,1.4)(3.6,1)
 \psline[linewidth=3pt](.7,2)(3.6,2)
 \rput(.5,2){$1$} \rput(2,3.2){$2 \,\,\,3 \,\,\,4 \,\,\,5\,\,\, 6\,\,\, 7\,\,\, 8$}
 \rput(3.8,2){$9$} \rput(2.2,.7){$16 \quad 15\, 14 \,13 \,12\, 11 \,10$}
 \pscircle*(2,2){.15}
\end{pspicture}}
\]
\caption{Canonical numbering $\fcan$ of the leaves of a diagram}\label{FIG:canon-leaf-order}
\end{figure}
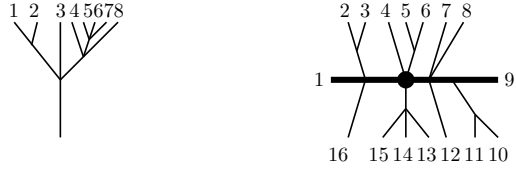

\begin{definition}
Define the \textbf{orientation} of a corolla to be $+1$ or $-1$; an \textbf{orientation} of a planar diagram $D$ with $\mathcal E(D)=\{e_1,\dots,e_k\}$
is a formal skew-commutative product $\omega = e_{i_1}\wedge \dots\wedge e_{i_k}$.
Thus $D$ admits exactly two orientations: $\omega$ and $-\omega$.   
\end{definition}

Following the construction of the colored operad governing homotopy inner
products in \cite{LT}, we define the three-colored operad $\CA$.
Let $\Z_3=\{0,1,2\}$, and denote the empty color by $0$, the thin color by $1$, and the thick color by $2$.

\begin{definition}\label{DEF:CA}
The \textbf{three-colored operad }
\[
\CA\hspace{.1in}=\bigoplus_{\substack{ \mathbf{x}\times y
\in \mathbb{Z}_{3}^{n+1} \\ n\geq 1}}\CAxy,
\]
is the graded $R$-module in which $\CAxy=0$ unless
\begin{enumerate}

\item[\textit{(i)}] 
$\mathbf{x}\times y =1\cdots1\times 1\in \mathbb{Z}_{3}^{n+1}$:\smallskip \\$\CAxy$ is generated by triples $\left(  T,f, \omega \right)  $ modulo $(T,f,-\omega)= -(T,f,\omega)$,  where $T$ is a planar tree with $n$ leaves and $\omega$ is an orientation on $T$.\smallskip
\item[\textit{(ii)}] 
$\mathbf{x}\times y=1\cdots 121\cdots 1\times 2\in \mathbb{Z}_{3}^{n+1}$
 with $2$ in the $i\T$ position of $\mathbf{x}$: \smallskip\\ 
$\CAxy$ is generated by triples $\left(  M,f,\omega \right)$ modulo $(M, f,-\omega)= -(M, f,\omega)$, where $M$ is a module tree with $n$ leaves of which $f(i)$ is thick, and $\omega$ is an orientation on $M$.\smallskip
\item[\textit{(iii)}] 
$\mathbf{x}\times y=1\cdots 121\cdots 121\cdots 1 \times 0 \in \mathbb{Z}_{3}^{n+1}$ with $2$ in the $i\T$ and $j\T$ positions:\smallskip \\
$\CAxy$ is generated by triples $\left(  I, f,\omega\right)$ modulo $(I, f,-\omega)= -(I, f,\omega)$, where $I$ is an inner
product diagram with $n$ leaves of which $f(i)$ and $f(j)$ are thick, and $\omega$ is an orientation on $I$.
\end{enumerate}
\noindent 
The \textbf{coloring} of a generator $\left(  D,f,\omega \right)  \in \CAxy$ is the pair $\mathbf{x}\times y$, and its \textbf{degree}   
\[
\left\vert(  D,f, \omega)  \right\vert :=\#\mathcal{L}\left(  D\right)
-\#\mathcal{E}\left(  D\right)  -2.
\]
When $\left\vert(  D,f, \omega)  \right\vert=n$, we write $(  D,f,\omega )  \in \Ca{n}$. 
Formally adjoin \textbf{units} to $\CA$ and define their degrees to be $0$. Given an edge $e\in\mathcal E(D),$ let $D/e$ denote the planar diagram obtained from $D$ by contracting $e$ to a point. Define
the \textbf{boundary} of $\left(  D,f,\omega\right)  $ by
\[
\partial_{C}(D,f,\omega):=\sum_{\substack{D^{\prime}/e^{\prime}=D \\e^{\prime}\in
\mathcal E(D^{\prime})}}\left(  D^{\prime},f,e'\wedge\omega \right)  ,
\]
summed over all diagrams $D^{\prime}$ and all edges $e^{\prime}\in\mathcal E(
D^{\prime})$ such that $D^{\prime}/e^{\prime}=D$. The relation $\partial_C^2=0$ follows from the fact that $e'\wedge e''\wedge \omega=-e''\wedge e'\wedge \omega$. If $n = \#\mathcal L(D)$ and $\sigma\in S_n$, define the \textbf{$S_n$-action} by
\[
\sigma\cdot(  D,f,\omega)  := \sgn(\sigma)\cdot (  D,f\circ \sigma,\omega).
\]
\noindent
Now, given generators
$\left(  D,f_{D},\omega_D\right)\in \Ca{m} $ and $\left(  E,f_{E},\omega_E\right)\in \Ca{n}$,  let
$k=\#\mathcal{L}\left(  D\right)  $ and $l=\#\mathcal{L}\left(  E\right)$. For $1\leq i\leq k$, define the $i\T$ \textbf{operadic composition } $\left(D,f_{D},\omega_D\right)  \circ_{i}\left(  E,f_{E},\omega_E\right)  $ to be zero unless the root of $E$ and the $i\T$ input of $D$ have the same color, in which case, 
\begin{multline}\label{EQ:circ-i}
(  D,f_{D}, \omega_D)  \circ_{i}(  E,f_{E},\omega_E)
:= (-1)^\epsilon( \, D\circ_{f_D(i)}E\, ,\, f_{D\circ E} \, ,\,  \omega_D\wedge\omega_E\wedge e \, ),
\end{multline}
where $\epsilon=i(l+1)+kn$, $D\circ_{f_D(i)}E$ is obtained by attaching the root of $E$ to leaf $f_{D}\left(  i\right)$ of $D$, ``$e$'' denotes the new edge, and
\[
f_{D\circ E}(j)  =\left\{
\begin{array}
[c]{ll}%
f_{D}(j)  , & 1\leq j<i, \\
f_{E}(j-i+1)  , & i\leq j<i+l, \\
f_{D}(j-l+1)  , & i+l\leq j\leq k+l-1.
\end{array}
\right.
\] 
For a justification of the sign $(-1)^\epsilon$, refer to Equations 
\eqref{EQ:A-infty-sign}, \eqref{EQ:A-module-sign}, \eqref{EQ:Inner-product-sign}, and Definition \ref{DEF:F:CA->End} in Appendix \ref{APP:A-infty-sign}. 
\end{definition}
\begin{remark}
We remark, that the $S_n$ action may be used to arrange the coloring $\mathbf{x}\times y$ in $(iii)$ above in the cyclic order starting with the color $2$ at the leftmost leaf. With this, an algebra over $\CA$ may be represented by inner product diagrams via maps $\rho_{j',j''}:M\otimes A^{\otimes j'}\otimes M\otimes A^{\otimes j''}\to R$. For more on algebras over $\CA$, see Appendix \ref{APP:A-infty-sign}.
\end{remark}

We need the following ``metric refinement'' of $\CA$ generated by diagrams whose edges are labeled either $m$ (metric) or $n$ (non-metric):

\begin{definition}
\label{operad-QA}The \textbf{three-colored operad} 
\[
\QA=\bigoplus_{\substack{ \mathbf{x}\times y
\in \mathbb{Z}_{3}^{n+1} \\ n\geq 1}} \QAxy,
\]
is the graded $R$-module generated by tuples $(D,f,g,\omega)$ modulo $(D,f,g, -\omega)= -(D,f,g,\omega)$, where $D$ and $f$ are as in Definition \ref{DEF:CA} and $g:\mathcal{E}\left(  D\right)
\rightarrow\{m,n\}$ labels the edges in $D$ (if $D$ is a corolla, $g$ is the empty map). The \textbf{metric edges} of $D$ form the set $\mathcal {M}(D)=g^{-1}(m)$; all other edges are \textbf{non-metric}. If $\mathcal {M}(D)=\{e_1,\dots,e_l\}$, an \textbf{orientation} of $D$ is a formal skew-commutative product $\omega = e_{i_1}\wedge \dots\wedge e_{i_l}$, and the \textbf{degree} $\left\vert (D,f,g,\omega)\right\vert := \#\mathcal {M}(D)$.
\textbf{Units} 
are inherited from $\CA$ and the \textbf{boundary} is given by
\begin{multline*}
\hspace{-.1in}\partial_{Q}(D,f,g,e_{i_1}\wedge\dots\wedge e_{i_l}):=\\
\hspace{.2in}\sum_{j=1}^l (-1)^{j} \Big[ (D/e_{i_j},f,g/e_{i_j},e_{i_1}\wedge\cdots \hat e_{i_j}\cdots\wedge e_{i_l})- (D_{e_{i_j}},f,g_{e_{i_j}}, e_{i_1}\wedge\cdots \hat e_{i_j}\cdots\wedge e_{i_l})\Big],
\end{multline*}
where $g/e_{i_j}(e)=g(e)$ if $e\neq e_{i_j}$, $D_{e_{i_j}}$ is obtained from $D$ by relabeling $e_{i_j}$ non-metric ($n$), and $g_{e_{i_j}}$ is the corresponding relabeling. It is straightforward to check that $\partial_Q^2=0$. This time the \textbf{$S_k$-action} is given by $\left.\sigma\cdot(D,f,g,\omega)=(D,f\circ \sigma,g,\omega)\right.$, and the \textbf{coloring} of $(D,f,g,\omega)\in \QAxy$ is the pair $\mathbf{x}\times y$.
Given generators $(D,f_{D},g_D,\omega_D)$ and $(E,f_{E},g_E,\omega_E)$, define the $i\T$ \textbf{operadic composition}\linebreak $(D,f_{D},g_D,\omega_D)\circ_i(E,f_{E},g_E,\omega_E)$ to be zero unless the root of $E$ and the $i\T$ leaf of $D$ have the same color, in which case
\[
(D,f_{D},g_D,\omega_D)\circ_{i}(E,f_{E},g_E,\omega_E):=\left(  D\circ_{f_D(i)}E, f_{D\circ E}, g_{D\circ E},\omega_D\wedge\omega_E\right),
\]
where $D\circ_{f_D(i)}E$ is defined as in Definition \ref{DEF:CA} and
\[
g_{D\circ E}\left( e\right) =\left\{ 
\begin{array}{ll}
n, & e\text{ is the new edge} \\ 
g_{D}\left( e\right) , & e\in \mathcal{E}\left( D\right)  \\ 
g_{E}\left( e\right) , & e\in \mathcal{E}\left( E\right). 
\end{array}%
\right. 
\]
\end{definition}
\noindent For comparison, note that the sign prefixes that appear in the formulas defining the $S_k$-action and $\circ_i$-composition in $\CA$, do not appear correspondingly in $\QA$.  Figure \ref{FIG-3} on page \pageref{FIG-3} displays a graphical representation of the combinatorial relationships among the generators in $\QA_{0}^{2112}$; this clearly illustrates the boundary $\partial_Q$. As we shall see in the next section, $\CA$ and $\QA$ are operadically quasi-isomorphic.

\section{The map $q:\CA\rightarrow \QA$}\label{q-sec}
In this section, we define an operadic map $q:\CA\to \QA$ and show that it is a quasi-isomorphism. Our proof depends on the fact that pairahedra are contractible. 
Let $\ost_B$ denote the standard orientation on a binary diagram $B\in\Ca{0}$ defined in Appendix \ref{APP:binary-standard-orientation}, and let $m$ denote the constant map $g(e)=m$.  We remark that a ``binary inner product diagram'' is an inner product diagram whose (empty) root has valence $2$ and whose other vertices have valence $3$.

\begin{definition}
Define $q:\CA \rightarrow \QA$ as follows:
\begin{enumerate}
\item[\textit{(i)}]
On units, define $q$ to be the identity.
\item[\textit{(ii)}] 
On a corolla $(c,\fcan,+1)\in\CAxy$, define 
\[ q(c,\fcan,+1)=\sum_{B \in \Caxy{0}} (B,\fcan,m,\ost_B). \]
\item[\textit{(iii)}]
Decompose a generator $(D,f,\omega)\in\CA$ as a $\circ_i$-composition of corollas, and define $q(D,f,\omega)$ by extending $S_k$-equivariantly and $\circ_i$-multiplicatively, \emph{i.e.},
\begin{eqnarray*}
q\Big(\sigma\cdot(D,f,\omega)\Big)&=&\sigma\cdot q(D,f,\omega);\\
\quad \quad q\Big((E,f,\omega)\circ_i (E',f',\omega')\Big)&=&q(E,f,\omega)  \circ_{i}q(E',f',\omega').
\end{eqnarray*}
This extension is well-defined since $\CA$ is freely generated by corollas $(c,\fcan,+1)$ (modulo the relation $(\cdots,-\omega)=-(\cdots,\omega)$).
\end{enumerate}
\end{definition}
\noindent By definition, $q$ respects units, the $S_k$-action, and $\circ_i$-composition. And furthermore:
\begin{proposition}
\label{q-bound}The map $q:\CA \rightarrow \QA$ is a chain map.
\end{proposition}
\begin{proof}
Since $\partial_{C}$ and $\partial_{Q}$ respect $S_k$-actions and act as derivations of $\circ_{i}$, it is sufficient to check the result on a corolla $(c,f_{\circlearrowright },+1)\in\CAxy$:\vspace{.2in}\newline
\noindent $q(\partial _{C}(c,f_{\circlearrowright
},+1))=\sum\limits_{D/e=c}q(D,f_{\circlearrowright },e)$%
\begin{eqnarray*}
&& \\
&=&\sum_{\substack{ _{\substack{ D/e=c, \\ (D,f_{\circlearrowright
},e)=(-1)^{\epsilon _{1}}\sigma \cdot \left[ (c^{\prime
},f_{\circlearrowright }^{\prime },1)\circ _{j}(c^{\prime \prime
},f_{\circlearrowright }^{\prime \prime },1)\right] }} \\ (c^{\prime
},f_{\circlearrowright }^{\prime },1)\in C_{\ast }\hat{\mathcal{A}}%
_{y^{\prime }}^{\mathbf{x^{\prime }}};\text{ }(c^{\prime \prime
},f_{\circlearrowright }^{\prime \prime },1)\in C_{\ast }\hat{\mathcal{A}}%
_{y^{\prime \prime }}^{\mathbf{x^{\prime \prime }}}}}\hspace*{-0.5in}%
(-1)^{\epsilon _{1}}\sigma \cdot \left[ q(c^{\prime },f_{\circlearrowright
}^{\prime },1)\circ _{j}q(c^{\prime \prime },f_{\circlearrowright }^{\prime
\prime },1)\right]  \\
&& \\
&=&\sum_{\substack{ B^{\prime }\in \Ca{0}
_{y^{\prime }}^{\mathbf{x^{\prime }}} \\ B^{\prime \prime }\in
\Ca{0}_{y^{\prime \prime }}^{\mathbf{x^{\prime \prime }}}
}}\hspace*{-0.2in}(-1)^{\epsilon _{1}}\sigma \cdot \left[ (B^{\prime
},f_{\circlearrowright }^{\prime },m,\omega _{B^{\prime }}^{\text{std}%
})\circ _{j}(B^{\prime \prime },f_{\circlearrowright }^{\prime \prime
},m,\omega _{B^{\prime \prime }}^{\text{std}})\right]  \\
&& \\
&=&\sum_{\substack{ (B,f_{\circlearrowright },g,\omega _{j})\in Q_{\ast }%
\hat{\mathcal{A}}_{y}^{\mathbf{x}} \\ \#g^{-1}(n)=1}}(-1)^{\epsilon
_{2}}(B,f_{\circlearrowright },g,\omega _{j}),
\end{eqnarray*}%
where $\epsilon_2$ is a function of $\epsilon_1$, and both signs are computed in Appendix \ref{APP:(3.1)-signs}. The last expression is summed over all binary diagrams $B$ with exactly one non-metric edge (the one coming from the composition $\circ_j$) and $\omega_j$ is the induced orientation. On the other hand,
\begin{multline}\label{EQ:d(q(c))}
\partial_{Q}(q(c,\fcan,+1))=\sum_{B\in\Caxy{0}}\partial_{Q}\left(B,\fcan,m,\ost_B=e_1\wedge\dots\wedge e_k\right) \\
 =\sum_{B}  \sum_{e_i}(-1)^i\left[(  B/e_i,\fcan,m/e_i,\omega_B^{\hat{e_i}}) -(  B_{e_i},\fcan, g_{e_i},\omega_B^{\hat {e_i}}) \right],
\end{multline}
where $\omega_B^{\hat {e_i}}=e_1\wedge\dots\wedge \hat{e_i}\wedge\dots\wedge e_k$. First, we claim that the ``$B/e_i$'' terms cancel. To see this, note that if $e_i$ is an edge of a binary diagram $B$, there is second binary diagram $B'$ and an edge $e'$ in $B'$ such that $B'/e'=B/e_i$. In fact, $B$ and $B'$ are related by one of the 6 local moves depicted in Definition \ref{DEF:order} (if $B < B'$, obtain $B/e_i$ by collapsing the edges within the two circles). Thus for purposes of orientation, we may symbolically identify $e'$ with $e_i$ and express their standard orientations as $\ost_B=e_1\wedge\dots\wedge e_i\wedge \dots\wedge e_k$ and $\ost_{B'}=- e_1 \wedge \dots\wedge e_i\wedge \dots\wedge e_k$. Consequently, each ``$B/e_i$'' term appears twice with opposite signs.
Formula \eqref{EQ:d(q(c))} now simplifies to \begin{equation*}
\partial_{Q}(q(c,\fcan,1))=\sum_{B}  \sum_{e_i}(-1)^{i+1}(B_{e_i},\fcan, g_{e_i},\omega_B^{\hat {e_i}}) =q(\partial_{C}(c,\fcan,1)),
\end{equation*}
summed over all binary diagrams $B$ with exactly one non-metric edge $e_i$; 
the fact that $(-1)^{i+1}\omega_B^{\hat {e_i}}=(-1)^{\epsilon_2}\omega_j$ is verified in Appendix \ref{APP:(3.1)-signs}. 
\end{proof}

The fact that $q$ is a quasi-isomorphism follows from the next proposition.

\begin{proposition}\label{contract}
The polytopes associated with $\A_{n}$, $M_{k,l}$, and
$I_{k,l}$ are contractible.
\end{proposition}

The proof of Proposition \ref{contract} is technical and appears in Appendix \ref{I_k,l k,l>0}.

\begin{corollary}
\label{q-quasi}The map $q:\CA\rightarrow \QA$ is a quasi-isomorphism of operads.
\end{corollary}

\begin{proof}
Since $q$ is a chain map respecting the $S_k$-action, $\circ_i$-composition, and units, it is operadic. Furthermore, $\q:\CAxy\rightarrow \QAxy$ is a quasi-isomorphism for each pair $\mathbf{x}\times y$ by Proposition \ref{contract} (\emph{c.f.} \cite{BV}).
\end{proof}

Geometrically, $\CA$ is represented by the cellular chains of the various polytopes mentioned above, $\QA$ is represented by the cellular chains of an appropriate subdivision, and the morphism $q$ associates each cell of $\CA$ with the sum of cells in its subdivision (\emph{c.f.} Figure \ref{FIG-3}). However, constructing a quasi-inverse of $q$ requires us to make a choice, which we shall do in the next section.

\section{The map $p:\QA\rightarrow \CA$}\label{p-sec}

In this section, we define a quasi-inverse $p$ of the morphism $q$ defined in the previous section. Our strategy is to extend the usual Tamari partial ordering on planar binary trees to the set $\mathcal{B}$ of all planar binary diagrams.  Given a planar diagram $D$, let $\mathcal{B}_{D}$ denote the
subposet of $\mathcal{B}$ given by applying all possible sequences of edge
insertions to $D.$ Then $\mathcal{B}_{D}$ has a unique minimal
element $D_{\min}$ and maximal element $D_{\max}$ (Lemma \ref{max-min}), and $p$ applied to a fully metric diagram $D$ is the sum of all diagrams $S$ such that $S_{\max}\leq D_{\min}$, and extending this in the general case $S_k$-equivariantly and $\circ_i$-multiplicatively (see Definition \ref{DEF:p-map}).

\begin{definition}\label{DEF:order}
A pair of binary diagrams
$( B_{1},B_{2})  \in\mathcal{B\times B}$ is an \textbf{edge-pair} if $B_1$ and $B_{2}$ differ only within a single neighborhood in one of the following six possible ways:
\medskip
\[
\resizebox{!}{1.5cm}{
\begin{pspicture}(3,0)(5.5,2)
 \rput(5.5,1){$<$} \rput(5.5,1.5){(1)}
\pscircle(4,1){.9}
 \psline(4,0)(4,.8) \psline(4,.8)(3.2,1.6) \psline(4,.8)(4.8,1.6) \psline(3.6,1.2)(4.4,2)
 \end{pspicture}
\begin{pspicture}(-.5,0)(2,2)
 \pscircle(1,1){.9} 
 \psline(1,0)(1,.8) \psline(1,.8)(.2,1.6) \psline(1,.8)(1.8,1.6) \psline(1.4,1.2)(.6,2)
 \end{pspicture}
  \quad\quad\quad\quad\quad\quad\quad\quad\quad
\begin{pspicture}(3,0)(5.5,2)
 \rput(5.5,1){$<$} \rput(5.5,1.5){(2)}
\pscircle(4,1){.9} \psline[linewidth=3pt](4,0)(4,2)
 \psline(4,.6)(3.1, 1.5) \psline(4,1.4)(4.5,1.9)
 \end{pspicture}
\begin{pspicture}(-.5,0)(2,2)
 \pscircle(1,1){.9}  \psline[linewidth=3pt](1,0)(1,2)  
 \psline(1,.6)(1.9, 1.5) \psline(1,1.4)(.5,1.9)
 \end{pspicture}}
\]
\[
\resizebox{!}{1.5cm}{
\begin{pspicture}(3,0)(5.5,2)
 \rput(5.5,1){$<$} \rput(5.5,1.5){(3)}
 \pscircle(4,1){.9}  \psline[linewidth=3pt](4,0)(4,2)
 \psline(4,.6)(3.1, 1.5) \psline(3.6,1)(3.4,1.8)
 \end{pspicture}
\begin{pspicture}(-.5,0)(2,2)
 \pscircle(1,1){.9}  \psline[linewidth=3pt](1,0)(1,2) 
 \psline(1,.6)(.1, 1.5) \psline(1,1.4)(.5,1.9)
 \end{pspicture}
 \quad\quad\quad\quad\quad\quad\quad\quad\quad
\begin{pspicture}(3,0)(5.5,2)
 \rput(5.5,1){$<$} \rput(5.5,1.5){(4)}
 \pscircle(4,1){.9}  \psline[linewidth=3pt](4,0)(4,2)
 \psline(4,.6)(4.9, 1.5) \psline(4,1.4)(4.5,1.9)
 \end{pspicture}
\begin{pspicture}(-.5,0)(2,2)
 \pscircle(1,1){.9}  \psline[linewidth=3pt](1,0)(1,2)  
 \psline(1,.6)(1.9, 1.5) \psline(1.4,1)(1.6,1.9)
 \end{pspicture} }
\]
\[
\resizebox{!}{1.5cm}{
\begin{pspicture}(3,0)(5.5,2)
 \rput(5.5,1){$<$} \rput(5.5,1.5){(5)}
\pscircle(4,1){.9} \pscircle*(4,1){.15}
 \psline[linewidth=3pt](3,1)(5,1)
 \psline(3.6,1)(3.2,1.6)
 \end{pspicture}
\begin{pspicture}(-.5,0)(2,2)
 \pscircle(1,1){.9} \pscircle*(1,1){.15}
 \psline[linewidth=3pt](0,1)(2,1) \psline(1.4,1)(1.8,1.6)
 \end{pspicture}
 \quad\quad\quad\quad\quad\quad\quad\quad\quad
\begin{pspicture}(3,0)(5.5,2)
 \rput(5.5,1){$<$} \rput(5.5,1.5){(6)}
 \pscircle(4,1){.9} \pscircle*(4,1){.15}
\psline[linewidth=3pt](3,1)(5,1) \psline(4.4,1)(4.8,.4)
 \end{pspicture} 
\begin{pspicture}(-.5,0)(2,2)
 \pscircle(1,1){.9} \pscircle*(1,1){.15} 
 \psline[linewidth=3pt](0,1)(2,1)  \psline(.6,1)(.2,.4) 
 \end{pspicture} }
\]
\end{definition}
\noindent The set of all edge-pairs $(B_1 ,B_2)$ generate a partial order \textquotedblleft$\leq$\textquotedblright on $\mathcal{B}$; for a proof see Appendix \ref{APP:partial-order}. 

\begin{lemma}\label{max-min}
The subposet $\mathcal{B}_{D}$ has a unique minimal
element $D_{\min}$ and maximal element $D_{\max}$.
\end{lemma}

\begin{proof}
Applying all possible sequences of edge insertions to a corolla $c$ generates
the subposet $\mathcal{B}_{c}$ of binary diagrams with unique 
minimal and maximal elements of the following types:
\medskip
\[
 \begin{pspicture}(0,.5)(4,2)
 \rput(1,1.25){$(T_n)_{\min}=$} \rput(3.6,1.2){,}
 \psline(3,.5)(3,1) \psline(3,1)(2,2) \psline(3,1)(4,2)
 \psline(2.8,1.2)(3.6,2)  \psline(2.6,1.4)(3.2,2)  \psline(2.4,1.6)(2.8,2)  \psline(2.2,1.8)(2.4,2)
 \end{pspicture} 
 \quad
 \begin{pspicture}(0,.5)(4,2)
 \rput(1,1.25){$(M_{k,l})_{\min}=$} \rput(3.6,1.2){,}
 \psline[linewidth=3pt](3,.5)(3,2) \psline(3,1)(2,2)
 \psline(3,1.2)(3.8,2) \psline(3,1.4)(3.6,2) \psline(3,1.6)(3.4,2) \psline(3,1.8)(3.2,2)
 \psline(2.8,1.2)(2.8,2) \psline(2.6,1.4)(2.6,2) \psline(2.4,1.6)(2.4,2) \psline(2.2,1.8)(2.2,2)
 \end{pspicture} 
 \quad
 \begin{pspicture}(0,.5)(3.8,2)
 \rput(1,1.25){$(I_{k,l})_{\min}=$}
 \psline[linewidth=3pt](2,1.25)(4,1.25) \pscircle*(3,1.25){.15}
 \psline(3.2,1.25)(3.4,.5) \psline(3.4,1.25)(3.6,.5) \psline(3.6,1.25)(3.8,.5) \psline(3.8,1.25)(4,.5)
 \psline(2.8,1.25)(2.6,2) \psline(2.6,1.25)(2.4,2) \psline(2.4,1.25)(2.2,2) \psline(2.2,1.25)(2,2)
 \end{pspicture} \quad
\]
\[
 \begin{pspicture}(0,.5)(4,2)
 \rput(1,1.25){$(T_n)_{\max}=$} \rput(3.6,1.2){,}
 \psline(3,.5)(3,1) \psline(3,1)(2,2) \psline(3,1)(4,2)
 \psline(3.8,1.8)(3.6,2)  \psline(3.6,1.6)(3.2,2)  \psline(3.4,1.4)(2.8,2)  \psline(3.2,1.2)(2.4,2)
 \end{pspicture} 
 \quad
 \begin{pspicture}(0,.5)(4,2)
 \rput(1,1.25){$(M_{k,l})_{\max}=$} \rput(3.6,1.2){,}
 \psline[linewidth=3pt](3,.5)(3,2) \psline(3,1)(4,2)
 \psline(3,1.2)(2.2,2) \psline(3,1.4)(2.4,2) \psline(3,1.6)(2.6,2) \psline(3,1.8)(2.8,2)
 \psline(3.2,1.2)(3.2,2) \psline(3.4,1.4)(3.4,2) \psline(3.6,1.6)(3.6,2) \psline(3.8,1.8)(3.8,2)
 \end{pspicture} 
 \quad
 \begin{pspicture}(0,.5)(3.8,2)
 \rput(1,1.25){$(I_{k,l})_{\max}=$}
 \psline[linewidth=3pt](2,1.25)(4,1.25) \pscircle*(3,1.25){.15}
 \psline(3.2,1.25)(3.4,2) \psline(3.4,1.25)(3.6,2) \psline(3.6,1.25)(3.8,2) \psline(3.8,1.25)(4,2)
 \psline(2.8,1.25)(2.6,.5) \psline(2.6,1.25)(2.4,.5) \psline(2.4,1.25)(2.2,.5) \psline(2.2,1.25)(2,.5)
 \end{pspicture} \quad
\]
\noindent Given a diagram $D \in \mathcal{B}$, there is always a sequence of inequalities in $\mathcal{B}_D$ from $D_{\min}$ to $D_{\max}$. For example,
\begin{multline*}
\resizebox{!}{1.5cm}{
\begin{pspicture}(0,0)(4.5,2) 
 \psline[linewidth=3pt](0,1)(3.6,1) \pscircle*(1.5,1){.15} \rput(4.1,1){$<$}
\psline(1.1,1)(.3,1.8) \psline(.8,1)(0,.2) \psline(.5,1)(0,1.5) \rput(4.1,1.5){(2)}
           \psline(.7,1.4)(1.1,1.8) \psline(.4,.6)(.8,.2)
 \psline(1.9,1)(2.7,1.8)  \psline(2.2,1)(2.7,.5)  \psline(2.8,1)(3.6,.2)
           \psline(3.2,.6)(2.8,.2) \psline(3.4,.4)(3.2,.2)
           \psline(1.9,1.4)(2.3,1.8) \psline(1.7,1.6)(1.9,1.8) \psline(2.1,1.2)(1.5,1.8)
 \end{pspicture} 
\begin{pspicture}(-.3,0)(5.2,2)  
 \psline[linewidth=3pt](-.3,1)(3.6,1) \pscircle*(1.5,1){.15} \rput(4.5,1){$<$}
 \rput(4.5,1.5){(5),(2),(6)}
\psline(1.1,1)(.3,1.8) \psline(.5,1)(-.3,.2) \psline(.8,1)(.3,1.5) 
           \psline(.7,1.4)(1.1,1.8) \psline(.1,.6)(.5,.2)
 \psline(2.8,1)(3.6,1.8)  \psline(1.9,1)(2.4,.5)  \psline(2.5,1)(3.3,.2)
           \psline(2.9,.6)(2.5,.2) \psline(3.1,.4)(2.9,.2)
           \psline(2.8,1.4)(3.2,1.8) \psline(2.6,1.6)(2.8,1.8) \psline(3,1.2)(2.4,1.8)
 \end{pspicture} 
\begin{pspicture}(-.8,0)(4.7,2)  
 \psline[linewidth=3pt](-.8,1)(4.7,1) \pscircle*(1.5,1){.15} 
\psline(1.1,1)(.3,.2) \psline(0,1)(-.8,.2) \psline(.7,1)(.2,.5) 
           \psline(.7,.6)(1.1,.2) \psline(.9,.4)(.7,.2) \psline(-.4,.6)(.1,.2) 
\psline(1.9,1)(2.4,1.5)  \psline(2.5,1)(3.3,1.8)  \psline(3.9,1)(4.7,1.8)  
           \psline(2.9,1.4)(2.5,1.8) 
           \psline(3.9,1.4)(4.3,1.8) \psline(3.7,1.6)(3.9,1.8) \psline(4.1,1.2)(3.5,1.8)
 \end{pspicture} }
 \\
 \resizebox{!}{1.5cm}{
\begin{pspicture}(-1.4,0)(5.2,2) 
 \psline[linewidth=3pt](-.8,1)(4.7,1) \pscircle*(1.5,1){.15} \rput(-1.3,1){$<$} \rput(5.1,1){$<$}
\rput(-1.3,1.5){(1)} \rput(5.1,1.5){(3)}
\psline(1.1,1)(.3,.2) \psline(0,1)(-.8,.2) \psline(.7,1)(.2,.5) 
           \psline(.7,.6)(1.1,.2) \psline(.5,.4)(.7,.2) \psline(-.4,.6)(.1,.2) 
\psline(1.9,1)(2.4,1.5)  \psline(2.5,1)(3.3,1.8)  \psline(3.9,1)(4.7,1.8)  
           \psline(2.9,1.4)(2.5,1.8) 
           \psline(4.5,1.6)(4.3,1.8) \psline(4.3,1.4)(3.9,1.8) \psline(4.1,1.2)(3.5,1.8)
 \end{pspicture} 
\begin{pspicture}(-.3,0)(3.8,2) 
 \psline[linewidth=3pt](0,1)(3.8,1) \pscircle*(1.8,1){.15}
 \psline(.5,1)(0,.5)  \psline(.7,1)(0.2,.5)  \psline(.9,1)(0.4,.5)
   \psline(1.1,1)(.6,.5)  \psline(1.3,1)(0.8,.5)  \psline(1.5,1)(1,.5)
 \psline(2.1,1)(2.6,1.5) \psline(2.3,1)(2.8,1.5) \psline(2.5,1)(3,1.5) \psline(2.7,1)(3.2,1.5) 
   \psline(2.9,1)(3.4,1.5) \psline(3.1,1)(3.6,1.5) \psline(3.3,1)(3.8,1.5) 
\end{pspicture}}
\end{multline*}
Since every diagram $D$ decomposes as a $\circ_{i}$-composition of corollas,
the conclusion follows.
\end{proof}

\begin{remark}
The six local inequalities in Definition \ref{DEF:order} define one of $2^{6}$ such local systems, some of which fail to generate a partial order on $\mathcal{B}$. Of those that do, some fail to satisfy the conclusion of Lemma \ref{max-min}. However, any partial order generated by one of these local systems satisfying the conclusion of Lemma \ref{max-min} can be used to construct the desired quasi-inverse $p$.  
\end{remark}

We are ready to define the operadic quasi-inverse $p:\QA\to \CA$. We shall refer to a diagram with strictly metric edges as a \emph{fully metric} diagram.

\begin{definition}\label{DEF:p-map}
Define $p:\QA\to \CA$ on generators as follows: 
\begin{enumerate}
\item[\textit{(i)}]
On units, define $p$ to be the identity.
\item[\textit{(ii)}] 
On a fully metric diagram $(D,\fcan,m,\omega_D)  \in Q_k\hat{\mathcal A}^{\mathbf{x}}_y$, define
\[
p(D,\fcan,m,\omega_D)=\sum_{\substack{(S,\fcan,\omega(S,D))  \in C_k\hat{\mathcal A}^{\mathbf{x}}_y \\S_{\max}\leq D_{\min}}}(S,\fcan,\omega(S,D)),
\]
where $\omega(S,D)$ is the induced orientation defined in \ref{APP:binary-standard-orientation}.\medskip
\item[\textit{(iii)}]
Decompose a generator $(D,f,g,\omega)  \in Q_{\ast}\hat{\mathcal A}$ with non-metric edges as a $\circ_i$-composition of fully metric diagrams, and define $p(D,f,g,\omega)$ by extending $S_k$-equivariantly and $\circ_i$-multiplicatively,
 \emph{i.e.},
\begin{eqnarray*}
p\Big(\sigma\cdot(D,f,g,\omega)\Big)&=&\sigma\cdotp(D,f,g,\omega),\\
\quad \quad p\big((E,f,g,\omega)\circ_i (E',f',g',\omega')\big)&=&p(E,f,g,\omega)  \circ_{i}p(E',f',g',\omega').
\end{eqnarray*}
\end{enumerate}
\end{definition}

\begin{lemma}\label{p-values}
Let $(c,f,g,1)\in \QA$ be a corolla, let $(B,f,m,\ost_B)\in\QA$
be a fully metric binary diagram, and let $\omega(c_{\min},c)$ be the orientation of $c_{\min}$ given by Equation \eqref{EQ:omega(c)=xi} in \ref{APP:binary-standard-orientation}.  Then

\begin{enumerate}
\item[\textit{(i)}] $p(c,f,g,1)  =(c_{\min},f,\omega(c_{\min},c))$.
\item[\textit{(ii)}] $p(B,f,m,\ost_B)  =\left\{
\begin{array}[c]{ll}
(c,f,1), & \text{if } B=c_{\max}\\
0, & \text{otherwise.}
\end{array}
\right. $
\end{enumerate}
\end{lemma}
\begin{proof}
\emph{(i)} Let $(S,f,\omega)\in C_{\ast}\hat{\mathcal{A}}$ be a summand of $p(c,f,g,1).$ Then $|(S,f,\omega)|=|(c,f,g,1)|=0$, which implies $\#\mathcal{L}\left(  S\right)  =\#\mathcal{E}\left(S\right)+2$. Thus $S$ is a binary diagram and $S=S_{\max}.$ Since $c$ is a corolla, $S=S_{\max}\leq c_{\min}$ implies $S=c_{\min};$ hence $p(c,f,g,1)  =(c_{\min},f,\omega(c_{\min},c))$.

\emph{(ii)} Let $(B,f,m,\ost_B)\in Q_{k}\hat{\mathcal{A}}$ be a fully metric binary diagram and let $(S,f,\omega)\in C_{k}\hat{\mathcal{A}}$ be a summand of $p(B,f,m,\ost_B)$. Then $B=B_{\min}$ has $k+2$ leaves and $k$ strictly metric edges
so that $k=|(B,f,m,\ost_B)|=|(S,f,\omega)|=(k+2)  -\#\mathcal{E}\left(  S\right)  -2.$ But $\#\mathcal{E}(S)  =0\,$implies $S=c$ is a corolla. The condition $c_{\max}\leq B_{\min}=B$ shows that when $B$ is not a maximum, $p(B,f,m,\ost_B)=0$; otherwise $p(B,f,m,\ost_B)=(c,f,+1)$ (the fact that $\omega(c,c_{\max})=+1$ is verified in Equation \eqref{EQ:omega(max)=1} in \ref{APP:binary-standard-orientation}.
\end{proof}

\begin{proposition}
\label{p-chain}The map $p:\QA\to \CA$ is a chain map.
\end{proposition}
A rather lengthy proof of Proposition \ref{p-chain} appears in
Appendix \ref{p-chain-proof}.

\begin{proposition}
The composition $p q=\operatorname*{Id}_{\CA}$, and there is a chain homotopy from the identity $\operatorname*{Id}_{\QA}$ to the composition $q p$. Thus, $p$ and $q$ are operadic chain homotopy equivalences.
\end{proposition}

\begin{proof}
Since $p$ and $q$ respect $\circ_{i}$-compositions and $C_{\ast} \hat{\mathcal{A}}$ is generated by $\circ_{i}$-compositions of corollas, it is sufficient to check $p q$ on a corolla $(c,f,1)\in \CAxy$. With Lemma \ref{p-values} \emph{(ii)}, we calculate
\[ p(q(c,f,1))=\sum_{B\in \Caxy{0}} p(B,f,m,\ost_B)=p(c_{\max}, f, m,\ost_{c_{\max}})=(c,f,1). \] 

On the other hand, we construct a chain homotopy $S:{\QA}\to {\QA}$, such that $ \partial_Q S+S \partial_Q=q p -\text{Id}_{\QA}$. The chain homotopy $S:\Qa{n}\to \Qa{n+1}$ is defined by induction on the degree $n$. For $n=0$, and the corolla $(c,f,g,1)\in \Qa{0}$, Lemma \ref{p-values}\emph{(i)} shows that
\[
q(p(c,f,g,1))=q(c_{\min},f,\omega(c_{\min},c))=(c_{\min},f,n,1),
\]
where $n$ is the constant non-metric assignment. The last sign follows from Equation \eqref{EQ:omega(c)=xi}, $\omega(c_{\min},c)=\xi_{c_{\min}}$, and the description of $\xi_{c_{\min}}$ given in Remark \ref{REM:xi-by-composition}. Since, by Proposition \ref{contract}, all pairahedra are contractible, we can find a path $S(c,f,g,1)\in \Qa{1}$ with boundary endpoints $(q p)(c,f,g,1)=(c_{\min},f,n,1)$ and $(c,f,g,1)$, so that with $\partial_Q (c,f,g,1)=0$, we have:
$$\left(\partial_Q  S+S \partial_Q\right) (c,f,g,1)=\left(q p-\operatorname*{Id}{}_{\QA}\right)(c,f,g,1).$$ 
Note that $q p$ is $\circ_i$-multiplicative and we extend $S$ to all of $\Qa{0}$ as an $(\operatorname*{Id},q p)$-derivation. With this, $\partial_Q S+S \partial_Q=q p -\operatorname*{Id}_{\QA}$ holds on all of $\Qa{0}$.

Inductively, assume that for all $k<n,$ the map $S$ has been extended as a chain homotopy from $\operatorname*{Id}_{\QA}$ to $q p$ on $Q_{k}\mathcal{\hat{A}}$.
Then, on $\Qa{n}$,
\begin{multline*}
\partial_{Q}\left(  qp-\operatorname*{Id}{}_{\QA}-S\partial_{Q}\right) 
=\left(  qp-\operatorname*{Id}{}_{\QA}\right)    \partial_{Q}  -\partial_{Q}S\partial
_{Q} \\
  =\left(  qp-\operatorname*{Id}{}_{\QA}-\partial_{Q}S\right)    \partial_{Q}
=S\partial_{Q}^{2}=0.
\end{multline*}
%
%
%
Thus, since the pairahedra are contractible, 
$(qp-\operatorname*{Id}_{\QA}-S\partial_{Q})  (D,f,m,\omega) \in Q_{n}\mathcal{\hat{A}}$ is a boundary for every $(D,f,m,\omega)\in Q_{n}\mathcal{\hat{A}}$. Choose $(D',f',g',\omega')\in Q_{n+1}\mathcal{\hat{A}}$ such that $\partial_{Q} (D',f',g',\omega')  =(q p-\operatorname*{Id}_{\QA}-S\partial_{Q})   (D,f,m,\omega)  $ and define $S (D,f,m,\omega)   = (D',f',g',\omega') $. Then $(  qp-\operatorname*{Id}_{\QA}-S\partial_{Q}-\partial_{Q}S)  (D,f,m,\omega) =0$, so that $S\partial_{Q}+\partial_{Q}S=q p-\operatorname*{Id}_{\QA}$ holds on $Q_{n}\mathcal{\hat{A}}$.
%
%
Hence $S:q p\simeq\operatorname*{Id}_{\QA}$.
\end{proof}

\section{The diagonal $\Delta_{C}:C_{\ast}\hat{\mathcal{A}}\rightarrow
C_{\ast}\hat{\mathcal{A}}\otimes C_{\ast}\hat{\mathcal{A}}$}\label{SEC:diag}

Let $I^{n}$ denote the standard $n$-cube, let $I=\sqcup_{n}I^{n}$ and let
$C_{\ast}\left(  I\right)  $ denote cellular chains. The Serre diagonal
$\Delta_{I}:C_{\ast}\left(  I\right)  \rightarrow C_{\ast}\left(  I\right)
\otimes C_{\ast}\left(  I\right)  $ acts on each cube $I^{n}$ of the cubical
complex $Q_{\ast}\hat{\mathcal{A}}$ and induces a coassociative diagonal
$\Delta_{Q}$ on $\QA$. In turn, $\Delta_{Q}$ induces a non-coassociative diagonal $\Delta_C$ 
on $\CA$ via $p$ and $q$. 

Given a generator $(D,f,g,\omega)  \in \QA$ and a subset $X\subseteq \left\{e_1, \ldots, e_k\right\}= g^{-1}(m)$, let $\bar X =g^{-1}(m)  \smallsetminus X$ and $\rho(X)=\#\left\{(e_i,e_j)\in X\times \bar{X}\mid i<j\right\}$. Consider the following related generators: 
\begin{itemize}
\item
$(D/X,f,g_{D/X},\omega_{D/X})$, where 
$D/X$ is obtained from $D$ by contracting the edges of $X$, 
$g_{D/X}$ is the labeling induced by $g$, and $\omega_{D/X}$ is the orientation obtained from $\omega$ by deleting all factors in $X$;\smallskip

\item
$(D_{\bar{X}},f,g_{D_{\bar{X}}},\omega_{D_{\bar{X}}})$, where 
$D_{\bar{X}}$ is obtained from $D$ by reversing labels on the edges in $\bar{X}$,  
$g_{D_{\bar{X}}}$ agrees with $g$ except on $\bar{X}$, and 
$\omega_{D_{\bar{X}}}$ is the orientation obtained from $\omega$ by deleting all factors in $\bar{X}$.
\end{itemize}

\begin{definition}
Define $\Delta_{Q}:\QA\rightarrow \QA\otimes \QA$ on generators by
\[
\Delta_{Q}(D,f,g,\omega)=\sum_{X\subseteq {g}^{-1}(m)  }(-1)^{\rho(X)} (D/X,f,g_{D/X},\omega_{D/X})  \otimes(D_{\bar{X}},f,g_{D_{\bar{X}}},\omega_{D_{\bar{X}}}).
\]

\end{definition}
\noindent
In particular, if $c$ is a corolla, then $\Delta_{Q}(c,f,g,1)=(c,f,g,1) \otimes(c,f,g,1)  ,$ where $g$ is the empty map.

The restriction of $\Delta_{Q}$ to the submodule $\mathcal{A}_\infty$ generated by metric planar rooted trees defines a (strictly coassociative) DG coalgebra structure on $\mathcal{A}_\infty$ that commutes with the operadic structure, see \cite[Proposition 5.1]{MS}.

\begin{definition}
Define $\Delta _{C}:\CA\rightarrow \CA\otimes \CA$ to be the composition 
\begin{equation*}
\CA\overset{q}{\longrightarrow }\QA\overset{\Delta _{Q}}{\longrightarrow }\QA%
\otimes \QA\overset{p\otimes p}{\longrightarrow }\CA\otimes \CA.
\end{equation*}
\end{definition}

\noindent To simplify notation, we sometimes abuse notation and write $D\in\CAxy$ when we mean $(D,f,\omega_D)$.

\begin{proposition}
\label{P:Delta(c)}On a corolla $(c,f,+1)\in \Caxy{k}$ we have 
\begin{equation*}
\Delta _{C}(c,f,+1)=\sum_{\substack{ S\otimes T \in \Caxy{i}\otimes \Caxy{j} \\ S_{\max
}\leq T_{\min } \\ i+j=k}}\pm (S,f,\omega _{S})\otimes (T,f,\omega _{T}) .
\end{equation*}
\end{proposition}

\begin{proof}
We evaluate the composition $(p\otimes p)\circ \Delta _{Q}\circ q$: 
\begin{eqnarray*}
\Delta _{C}(c,f,+1) &=&(p\otimes p)\Delta _{Q}(q(c,f,+1))\\ && \\&=&\sum_{B\in\Caxy{0}}(p\otimes p)\Delta _{Q}(B,f,m,\ost_{B})\\
&&  \\
&=&\sum_{\substack{ B\in\Caxy{0} \\ X\subseteq \mathcal E(B)}}\pm \,\,p(B/X,f,g_{B/X},\omega_{B/X})\otimes p(B_{\bar{X}},f,g_{B_{\bar{X}}},\omega_{B_{\bar{X}}}).  
\end{eqnarray*}%
To evaluate $p(B_{\bar{X}},f,g_{B_{\bar{X}}},\omega_{B_{\bar{X}}}),$ note that the non-metric
edges of the binary diagram $B$ are exactly the edges in $\bar{X}=\left\{
e_{1},\ldots ,e_{k-1}\right\} ,$ and $B$ decomposes as a $\circ _{i}$%
-composition of fully metric binary diagrams $B_{1},\dots ,B_{k}$ along the
edges of $\bar{X}:$%
\begin{equation*}
(B_{\bar{X}},f,g_{B_{\bar{X}}},\omega_{B_{\bar{X}}})=\pm \sigma \cdot \left[(B_{1},f_{1},m,\omega
_{1})\circ _{e_{1}}\cdots \circ _{e_{k-1}}(B_{k},f_{k},m,\omega _{k})\right]
\end{equation*}%
(such decompositions are unique up to operadic associativity). Since $B_{i}$
is fully metric, Lemma \ref{p-values} implies that $p(B_{i},f_{i},m,\omega
_{i})$ is a corolla if $B_{i}$ is maximal and vanishes otherwise. Hence if $%
p(B_{\bar{X}},f,g_{B_{\bar{X}}},\omega_{B_{\bar{X}}})\neq 0,$ each $B_{i}$ is maximal and $p(B_{\bar{X}},f,g_{B_{\bar{X}}},\omega_{B_{\bar{X}}})=\pm (B/X,f,\omega_{B/X})$. Furthermore,
if $(T,f,\omega _{T})\in \CAxy,$ there is a unique $B$
with maximal $B_{i}$'s such that $T=B/X.$ Consequently,%
\begin{multline*}
\Delta _{C}(c,f,1)=\sum_{\substack{ B\in\Caxy{0} \\ %
X\subseteq \mathcal{E}(B) \\ \text{all $B_{i}$ maximal}}}\pm
p(B/X,f,g_{B/X},\omega_{B/X})\otimes (B/X,f,\omega_{B/X}) \\
=\sum_{\substack{ S\otimes T\in \Caxy{i}\otimes \Caxy{j} \\ S_{\max
}\leq T_{\min } \\ i+j=k}}\pm (S,f,\omega _{S})\otimes (T,f,\omega _{T}) .
\end{multline*}
\end{proof}

\begin{remark}
In \cite[Proposition 5.1]{MS}, Markl and Shnider proved the special case of
Proposition \ref{P:Delta(c)} with $\Delta _{C}$ restricted to the submodule $%
\mathcal{A}_{\infty }$ generated by planar rooted trees. Since $\Delta _{Q}$
is induced by the Serre diagonal on $I^{n},$ it is strictly coassociative on 
$\QA$. However, $\Delta _{C}$ is not coassociative. In fact, Markl and
Shnider also remarked that $\Delta _{C}$ cannot be chosen to be
coassociative on $\mathcal{A}_{\infty }$. Nevertheless, this structure is
homotopy coassociative, and the formula for $\Delta _{C}$ in Proposition \ref%
{P:Delta(c)} extends to a $A_{\infty }$-coalgebra structure on $\CA$ in a
natural way (for the special $\mathcal{A}_{\infty }$ case see \cite{L}).
\end{remark}

\section{Computations}\label{SEC:computations}
In this section, we make explicit computations and remarks about the diagonal $\Delta_C:\CA\rightarrow \CA\otimes \CA$ in various situations. In Example \ref{EX:Delta(I_2,0)} we calculate the diagonal of $I_{2,0}$, and in Example \ref{Delta-for-cyclic} we make some general remarks about the diagonal as it would appear for a cyclic $A_\infty$-algebra. Both of these examples are given on the level of operads, and may be transferred to the corresponding statements for homotopy inner products. In Remark \ref{REM:strong} we comment on the diagonal for strong homotopy inner products in the sense of \cite{C}. In all of these considerations, we ignore signs and always write ``$+$'' regardless of orientation.

\begin{example}\label{EX:Delta(I_2,0)}
To evaluate $\Delta_{C}=\left(  p\otimes p\right)  \circ\Delta_{Q}\circ q$ on the corolla $I_{2,0},$ note that
\[
q(I_{2,0})=\iia{m}{m}+\iib{m}{m}+\iie{m}{m}+\iic{m}{m}+\iid{m}{m}
 \]
is the sum of the five squares in the metric pairahedron pictured in Figure \ref{FIG-3}.
\begin{figure}[ht]
\[
\includegraphics[height=3.3in,width=3.65in]{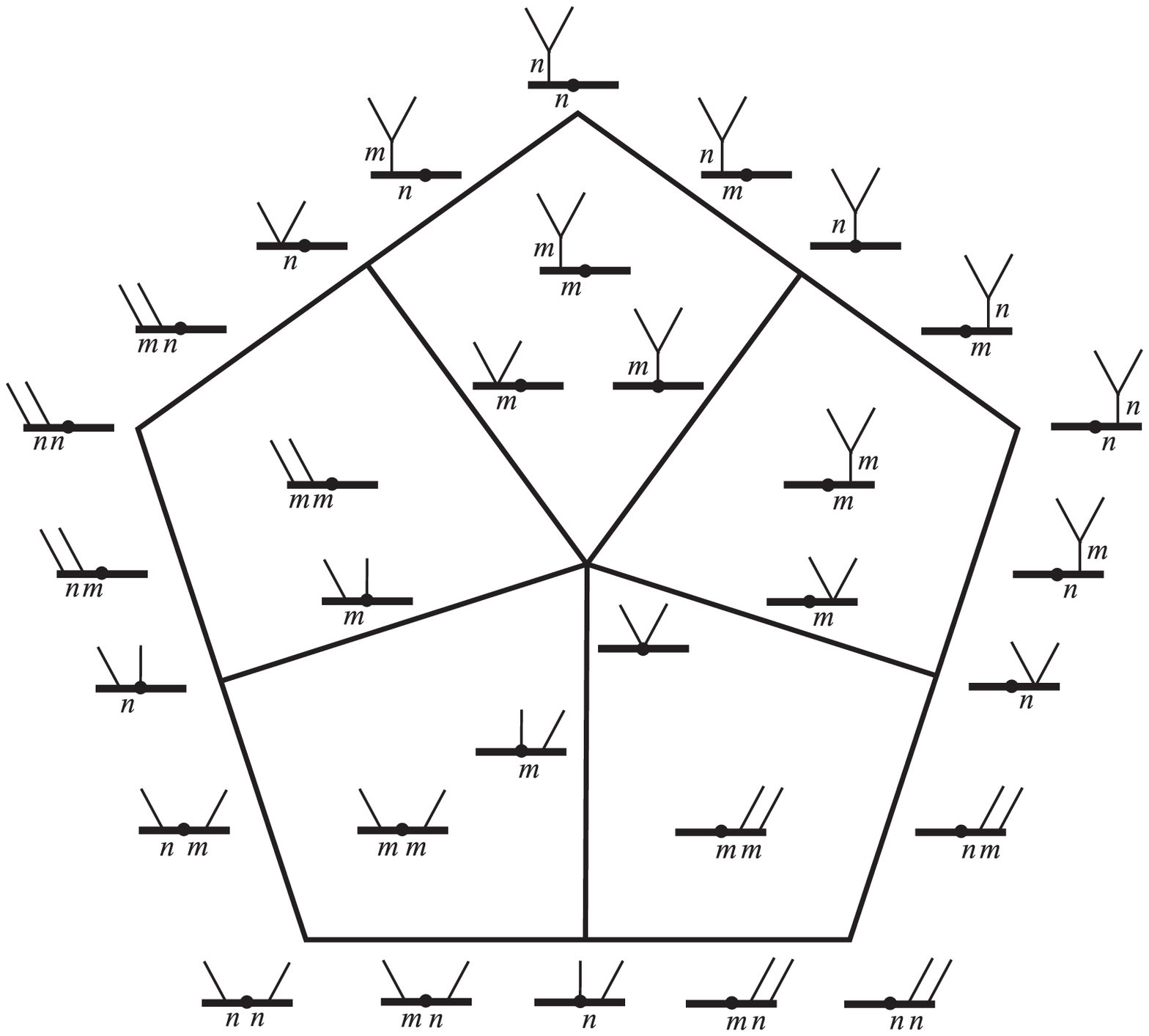}
\]
\caption{The metric pairahedron associated with $I_{2,0}$ (metric edges $m$ vary in length; non-metric edges $n$ have fixed length $1$).}\label{FIG-3}
\end{figure}
Evaluating the Serre diagonal on each of these squares gives
\begin{eqnarray*}
&  &  \Delta_Q\Big(\iia{m}{m}+\iib{m}{m}+\iie{m}{m}+\iic{m}{m}+\iid{m}{m}\Big)= \\
&  & \,\,\,\,\, \iik\otimes \iia{m}{m}+\iia{m}{m}\otimes \iia{n}{n}+\iif{m}\otimes \iia{m}{n}+\iij{m}\otimes \iia{n}{m} \\
&  & +\iik\otimes\iib{m}{m}+\iib{m}{m}\otimes \iib{n}{n}+\iif{m}\otimes \iib{m}{n}+\iig{m}\otimes \iib{n}{m} \\
&  & +\iik\otimes \iie{m}{m}+\iie{m}{m}\otimes \iie{n}{n}+\iii{m}\otimes \iie{m}{n}+\iij{m}\otimes \iie{n}{m} \\
&  & +\iik\otimes\iic{m}{m}+\iic{m}{m}\otimes \iic{n}{n}+ \iig{m}\otimes \iic{n}{m}+\iih{m}\otimes \iic{m}{n}\\
&  & +\iik\otimes \iid{m}{m}+\iid{m}{m}\otimes \iid{n}{n}+\iii{m}\otimes \iid{m}{n}+\iih{m}\otimes \iid{n}{m}.
\end{eqnarray*}
Next, we need to apply $p\otimes p$ to the above terms. It turns out, that many of the diagrams above lie in the kernel of $p$. For example, $p\big(\iij{m}\big)$ is given by a sum of diagrams $S$ of degree $1$ with $S_{\max}\leq \big(\iij{}\big)_{\min}=\iia{}{}$. However, $\iia{}{}$ is the global minimum of $\iik$, and there is no diagram $S$ with the required property.
An explicit check shows that the only terms that are not in the kernel of $p\otimes p$ are the following:
\[
\iig{m}\otimes \iic{n}{m}+\iik\otimes \iid{m}{m}+\iid{m}{m}\otimes \iid{n}{n}+\iii{m}\otimes \iid{m}{n}+\iih{m}\otimes \iid{n}{m}.
\]
Applying $p\otimes p$ we obtain
\begin{eqnarray*}
\Delta_C\Big(\iik \Big)&=&\iif{}\otimes\iig{} +\iia{}{}\otimes \iik+\iik\otimes \iid{}{}\\ 
&& +\iij{}\otimes \iii{}+\Big(\iif{}+\iig{}\Big)\otimes \iih{}.
\end{eqnarray*}
Note that if $D$ is a right-factor in a non-primitive term of $\Delta_C(I_{2,0})$, the set of all left-hand factors that pair off with $D$ in $\Delta_C(I_{2,0})$ form a path from the minimal vertex of $I_{2,0}$ to the maximal vertex of $I_{2,0}$ (see Figure \ref{FIG-4}).  
\begin{figure}[ht]
\includegraphics[height=2.2139in,width=2.5737in]{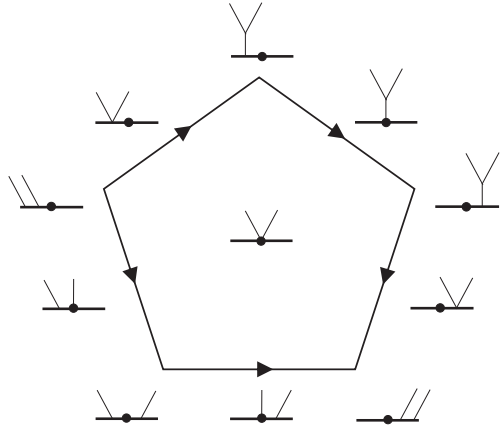}
\caption{The vertex poset of $I_{2,0}.$}\label{FIG-4}
\end{figure}
\end{example}
\medskip

\begin{example}\label{Delta-for-cyclic}
Let us calculate $\Delta_{C}$ in low dimensions as it would appear in the setting of a \emph{cyclic} $A_{\infty}$-algebra. In this case, the corolla $I_{k,l}$ corresponds to a zero map whenever $k+l>0$, and the non-trivial inner product diagram $I_{0,0}$ consists of two thick horizontal edges joined at a common vertex. Thus we compute $\Delta_{C}(c)=\sum\pm
S\otimes T$ by summing over all diagrams for inner products $S$ and $T$ such that $S_{\max}\leq T_{\min}$ (see Proposition \ref{P:Delta(c)}), and we will only keep track of terms involving $I_{0,0}$, that is we mod out by terms involving higher inner products $I_{k,l}$ with $k+l>0$. The results below can be obtained as in Example \ref{EX:Delta(I_2,0)}, and subsequently removing all solutions involving higher inner products $I_{k,l}$ with $k+l>0$. However, these calculations are much more involved than those in Example \ref{EX:Delta(I_2,0)}, and we leave them as an exercise to the reader. We obtain the following:
\begin{eqnarray*}
\Delta_C(I_{0,0})&=& I_{0,0}\otimes I_{0,0}, \\
\Delta_C(I_{1,0})&=&\Delta_C(I_{0,1})\,\,\, =\,\,\, \Delta_C(I_{1,1})\,\,\, =\,\,\, 0,\\
\Delta_C(I_{2,0})&\equiv&
\resizebox{!}{.5cm}{
 \begin{pspicture}(0,.8)(2,1.7)
 \pscircle*(1,1){.15}  \psline[linewidth=3pt](0,1)(2,1) 
 \psline(.5,1)(.3,1.6)  \psline(.5,1)(.7,1.6)
 \end{pspicture} }
 \otimes
\resizebox{!}{.5cm}{
 \begin{pspicture}(3,.8)(5,1.7)
 \pscircle*(4,1){.15}  \psline[linewidth=3pt](3,1)(5,1)
 \psline(4.5,1)(4.7,1.6) \psline(4.5,1)(4.3,1.6)
 \end{pspicture} } \quad\quad (\text{mod }\{I_{k,l}\,\,:\,\,k+l>0 \}).
\end{eqnarray*}
Note that $\Delta_C(I_{2,0})$ corresponds to the chain homotopy $\varrho_{2,0}$ mentioned in the introduction. More precisely, in the introduction, we considered the canonical example of the $\infty$-bimodule $A$ over the $A_\infty$-algebra $A$ as defined in the appendix, \emph{cf.} Equation \eqref{EQU:canon-example}, whose $\infty$-bimodule structure $\lambda_{j',j''}$ is given by the $A_\infty$-algebra structure $\mu_{j'+j''+1}$, and similarly for the $\infty$-bimodule structure for $B$.

The above expression is a special case of the following more general formula:
\[
\Delta_C(I_{k,l})=\sum_{r=0}^{k+l} \sum_{i=1}^{p_r} S_i^{(r)}\otimes T_i^{(k+l-r)},
\]
where $S_i^{(q)}, T_i^{(q)}\in \Ca{q}$ are of degree $q$. In the special case of concern in this example, that is for cyclic $A_\infty$-algebras, the $r=0$ and $r=k+l$ terms vanish, and the $r=1$ and $r=k+l-1$ terms are given by the formula:
\begin{eqnarray*}
\Delta_C(I_{k,l})& \equiv &
\resizebox{!}{.5cm}{  \begin{pspicture}(0,.9)(2,1.8)  
 \pscircle*(1,1){.15}  \psline[linewidth=3pt](0,1)(2,1)
 \psline(.5,1)(.2,1.6)  \psline(.5,1)(.4,1.6) \psline(.5,1)(.6,1.6) \psline(.5,1)(.8,1.6)
 \psline(.5,1)(.2,.4)  \psline(.5,1)(.4,.4)  \psline(.5,1)(.6,.4)  \psline(.5,1)(.8,.4)
 \end{pspicture} } 
 \otimes \left(
\resizebox{!}{.5cm}{  \begin{pspicture}(0,.9)(2,1.8) 
 \pscircle*(1,1){.15}  \psline[linewidth=3pt](0,1)(2,1)
 \psline(1.5,1)(1.3,1.3)  \psline(1.5,1)(1.95,1.9) \psline(1.65,1.3)(1.45,1.6) \psline(1.8,1.6)(1.6,1.9)
\psline(.8,1)(.6,.4) \psline(.6,1)(.4,.4) \psline(.4,1)(.2,.4) \psline(.2,1)(0,.4)
 \end{pspicture} }  
 +
\resizebox{!}{.5cm}{  \begin{pspicture}(0,.9)(2,1.8)  
 \pscircle*(1,1){.15}  \psline[linewidth=3pt](0,1)(2,1)
 \psline(1.7,1)(1.5,1.5)  \psline(1.7,1)(2.1,1.8) \psline(1.9,1.4)(1.7,1.8) \psline(1.45,1)(1.25,1.5)
\psline(.8,1)(.6,.4) \psline(.6,1)(.4,.4) \psline(.4,1)(.2,.4) \psline(.2,1)(0,.4)
 \end{pspicture} }  
 + \dots+
\resizebox{!}{.5cm}{  \begin{pspicture}(0,.9)(2,1.8)  
 \pscircle*(1,1){.15}  \psline[linewidth=3pt](0,1)(2,1)
\psline(1.3,1)(1.1,1.6) \psline(1.5,1)(1.3,1.6) \psline(1.7,1)(1.5,1.6) \psline(1.7,1)(1.9,1.6)
\psline(.8,1)(.6,.4) \psline(.6,1)(.4,.4) \psline(.4,1)(.2,.4) \psline(.2,1)(0,.4)
 \end{pspicture} }  
\right) \\ 
&& +
 \scalebox{-1}[1] {
\resizebox{!}{.5cm}{  \begin{pspicture}(0,.9)(2,1.8)  
 \pscircle*(1,1){.15}  \psline[linewidth=3pt](0,1)(2,1)
 \psline(.5,1)(.2,1.6)  \psline(.5,1)(.4,1.6) \psline(.5,1)(.6,1.6) \psline(.5,1)(.8,1.6)
 \psline(.5,1)(.2,.4)  \psline(.5,1)(.4,.4)  \psline(.5,1)(.6,.4)  \psline(.5,1)(.8,.4)
 \end{pspicture} } }
 \otimes \left(
   \scalebox{-1}{
\resizebox{!}{.5cm}{  \begin{pspicture}(0,1.2)(2,2.1) 
 \pscircle*(1,1){.15}  \psline[linewidth=3pt](0,1)(2,1)
\psline(1.3,1)(1.1,1.6) \psline(1.5,1)(1.3,1.6) \psline(1.7,1)(1.5,1.6) \psline(1.7,1)(1.9,1.6)
\psline(.8,1)(.6,.4) \psline(.6,1)(.4,.4) \psline(.4,1)(.2,.4) \psline(.2,1)(0,.4)
 \end{pspicture} }  }
 +
  \scalebox{-1}{
\resizebox{!}{.5cm}{  \begin{pspicture}(0,1.2)(2,2.1) 
 \pscircle*(1,1){.15}  \psline[linewidth=3pt](0,1)(2,1)
 \psline(1.7,1)(1.5,1.5)  \psline(1.7,1)(2.1,1.8) \psline(1.9,1.4)(1.7,1.8) \psline(1.45,1)(1.25,1.5)
\psline(.8,1)(.6,.4) \psline(.6,1)(.4,.4) \psline(.4,1)(.2,.4) \psline(.2,1)(0,.4)
 \end{pspicture} }  }
 +\dots+
  \scalebox{-1}{
\resizebox{!}{.5cm}{  \begin{pspicture}(0,1.2)(2,2.1) 
 \pscircle*(1,1){.15}  \psline[linewidth=3pt](0,1)(2,1)
 \psline(1.5,1)(1.3,1.3)  \psline(1.5,1)(1.95,1.9) \psline(1.65,1.3)(1.45,1.6) \psline(1.8,1.6)(1.6,1.9)
\psline(.8,1)(.6,.4) \psline(.6,1)(.4,.4) \psline(.4,1)(.2,.4) \psline(.2,1)(0,.4)
 \end{pspicture} }  }
\right) \\ 
&&+\sum_{r=2}^{k+l-2} \sum_{i} S_i^{(r)}\otimes T_i^{(k+l-r)}\\
&& + 
\left(
   \scalebox{1}[-1]{
\resizebox{!}{.5cm}{  \begin{pspicture}(0,1.2)(2,2.1) 
 \pscircle*(1,1){.15}  \psline[linewidth=3pt](0,1)(2,1)
\psline(1.3,1)(1.1,1.6) \psline(1.5,1)(1.3,1.6) \psline(1.7,1)(1.5,1.6) \psline(1.7,1)(1.9,1.6)
\psline(.8,1)(.6,.4) \psline(.6,1)(.4,.4) \psline(.4,1)(.2,.4) \psline(.2,1)(0,.4)
 \end{pspicture} }  }
 +
  \scalebox{1}[-1]{
\resizebox{!}{.5cm}{  \begin{pspicture}(0,1.2)(2,2.1) 
 \pscircle*(1,1){.15}  \psline[linewidth=3pt](0,1)(2,1)
 \psline(1.7,1)(1.5,1.5)  \psline(1.7,1)(2.1,1.8) \psline(1.9,1.4)(1.7,1.8) \psline(1.45,1)(1.25,1.5)
\psline(.8,1)(.6,.4) \psline(.6,1)(.4,.4) \psline(.4,1)(.2,.4) \psline(.2,1)(0,.4)
 \end{pspicture} }  }
 +\dots+
  \scalebox{1}[-1]{
\resizebox{!}{.5cm}{  \begin{pspicture}(0,1.2)(2,2.1) 
 \pscircle*(1,1){.15}  \psline[linewidth=3pt](0,1)(2,1)
 \psline(1.5,1)(1.3,1.3)  \psline(1.5,1)(1.95,1.9) \psline(1.65,1.3)(1.45,1.6) \psline(1.8,1.6)(1.6,1.9)
\psline(.8,1)(.6,.4) \psline(.6,1)(.4,.4) \psline(.4,1)(.2,.4) \psline(.2,1)(0,.4)
 \end{pspicture} }  }
\right) \otimes
\resizebox{!}{.5cm}{  \begin{pspicture}(0,.9)(2,1.8)  
 \pscircle*(1,1){.15}  \psline[linewidth=3pt](0,1)(2,1)
 \psline(.5,1)(.2,1.6)  \psline(.5,1)(.4,1.6) \psline(.5,1)(.6,1.6) \psline(.5,1)(.8,1.6)
 \psline(.5,1)(.2,.4)  \psline(.5,1)(.4,.4)  \psline(.5,1)(.6,.4)  \psline(.5,1)(.8,.4)
 \end{pspicture} } 
\\ && +
 \left(
 \scalebox{-1}[1]{
\resizebox{!}{.5cm}{  \begin{pspicture}(0,.9)(2,1.8) 
 \pscircle*(1,1){.15}  \psline[linewidth=3pt](0,1)(2,1)
 \psline(1.5,1)(1.3,1.3)  \psline(1.5,1)(1.95,1.9) \psline(1.65,1.3)(1.45,1.6) \psline(1.8,1.6)(1.6,1.9)
\psline(.8,1)(.6,.4) \psline(.6,1)(.4,.4) \psline(.4,1)(.2,.4) \psline(.2,1)(0,.4)
 \end{pspicture} }  }
 +
 \scalebox{-1}[1]{
\resizebox{!}{.5cm}{  \begin{pspicture}(0,.9)(2,1.8)  
 \pscircle*(1,1){.15}  \psline[linewidth=3pt](0,1)(2,1)
 \psline(1.7,1)(1.5,1.5)  \psline(1.7,1)(2.1,1.8) \psline(1.9,1.4)(1.7,1.8) \psline(1.45,1)(1.25,1.5)
\psline(.8,1)(.6,.4) \psline(.6,1)(.4,.4) \psline(.4,1)(.2,.4) \psline(.2,1)(0,.4)
 \end{pspicture} }  }
 +\dots+
 \scalebox{-1}[1]{
\resizebox{!}{.5cm}{  \begin{pspicture}(0,.9)(2,1.8)  
 \pscircle*(1,1){.15}  \psline[linewidth=3pt](0,1)(2,1)
\psline(1.3,1)(1.1,1.6) \psline(1.5,1)(1.3,1.6) \psline(1.7,1)(1.5,1.6) \psline(1.7,1)(1.9,1.6)
\psline(.8,1)(.6,.4) \psline(.6,1)(.4,.4) \psline(.4,1)(.2,.4) \psline(.2,1)(0,.4)
 \end{pspicture} }  }
\right) \otimes
\scalebox{-1}{
\resizebox{!}{.5cm}{  \begin{pspicture}(0,1.1)(2,2)  
 \pscircle*(1,1){.15}  \psline[linewidth=3pt](0,1)(2,1)
 \psline(.5,1)(.2,1.6)  \psline(.5,1)(.4,1.6) \psline(.5,1)(.6,1.6) \psline(.5,1)(.8,1.6)
 \psline(.5,1)(.2,.4)  \psline(.5,1)(.4,.4)  \psline(.5,1)(.6,.4)  \psline(.5,1)(.8,.4)
 \end{pspicture} }  }\\
 && \hspace{2in}   (\text{mod }\{I_{k,l}\,\,:\,\,k+l>0 \}) .
\end{eqnarray*}
Computing $\Delta_C$ in general requires more work, as the formula for $\Delta_C(I_{4,0})$ shows:
\begin{eqnarray*}
\Delta_C(I_{2,1})&\equiv&
\resizebox{!}{.5cm}{  \begin{pspicture}(0,.8)(2,1.7) 
 \pscircle*(1,1){.15}  \psline[linewidth=3pt](0,1)(2,1)
 \psline(.5,1)(.3,1.6)  \psline(.5,1)(.7,1.6) \psline(.5,1)(.5,.4)
 \end{pspicture} } 
 \otimes
\resizebox{!}{.5cm}{  \begin{pspicture}(0,.8)(2,1.7) 
 \pscircle*(1,1){.15}  \psline[linewidth=3pt](0,1)(2,1)
 \psline(1.5,1)(1.3,1.6)  \psline(1.5,1)(1.7,1.6) \psline(.5,1)(.5,.4)
 \end{pspicture} } 
  +
\resizebox{!}{.5cm}{  \begin{pspicture}(0,.8)(2,1.7) 
 \pscircle*(1,1){.15}  \psline[linewidth=3pt](0,1)(2,1)
  \psline(.5,1)(.3,1.6)  \psline(.5,1)(.7,1.6)  \psline(1.5,1)(1.5,.4)
 \end{pspicture} } 
 \otimes
\resizebox{!}{.5cm}{  \begin{pspicture}(0,.8)(2,1.7) 
 \pscircle*(1,1){.15}  \psline[linewidth=3pt](0,1)(2,1)
  \psline(1.5,1)(1.3,1.6)  \psline(1.5,1)(1.7,1.6) \psline(1.5,1)(1.5,.4)
 \end{pspicture} } 
   (\text{mod }\{I_{k,l}\,\,:\,\,k+l>0 \}), \\
 \\
\Delta_C(I_{3,0})&\equiv&
\resizebox{!}{.5cm}{  \begin{pspicture}(0,.9)(2,1.8) 
 \pscircle*(1,1){.15}  \psline[linewidth=3pt](0,1)(2,1)
 \psline(.5,1)(.3,1.6)  \psline(.5,1)(.7,1.6) \psline(.5,1)(.5,1.6)
 \end{pspicture} } 
 \otimes \left(
\resizebox{!}{.5cm}{  \begin{pspicture}(0,.9)(2,1.8)  
 \pscircle*(1,1){.15}  \psline[linewidth=3pt](0,1)(2,1)
 \psline(1.5,1)(1.3,1.5)  \psline(1.5,1)(1.9,1.8) \psline(1.7,1.4)(1.5,1.8)
 \end{pspicture} }  
 +
\resizebox{!}{.5cm}{  \begin{pspicture}(0,.9)(2,1.8) 
 \pscircle*(1,1){.15}  \psline[linewidth=3pt](0,1)(2,1)
 \psline(1.7,1)(1.5,1.6)  \psline(1.7,1)(1.9,1.6) \psline(1.3,1)(1.3,1.6)
 \end{pspicture} }  
\right) \\ 
&&+\left(
 \scalebox{-1}[1] {
 \resizebox{!}{.5cm}{  \begin{pspicture}(0,.9)(2,1.8) 
 \pscircle*(1,1){.15}  \psline[linewidth=3pt](0,1)(2,1)
 \psline(1.7,1)(1.5,1.6)  \psline(1.7,1)(1.9,1.6) \psline(1.3,1)(1.3,1.6)
 \end{pspicture} } }
 +
  \scalebox{-1}[1] { 
\resizebox{!}{.5cm}{  \begin{pspicture}(0,.9)(2,1.8)  
 \pscircle*(1,1){.15}  \psline[linewidth=3pt](0,1)(2,1)
 \psline(1.5,1)(1.3,1.5)  \psline(1.5,1)(1.9,1.8) \psline(1.7,1.4)(1.5,1.8)
 \end{pspicture} }  }
  \right)\otimes 
\resizebox{!}{.5cm}{  \begin{pspicture}(0,.9)(2,1.8) 
 \pscircle*(1,1){.15}  \psline[linewidth=3pt](0,1)(2,1)
 \psline(1.5,1)(1.3,1.6)  \psline(1.5,1)(1.7,1.6) \psline(1.5,1)(1.5,1.6)
 \end{pspicture} } 
\quad (\text{mod }\{I_{k,l}\,\,:\,\,k+l>0 \}),\\
\Delta_C(I_{4,0})&\equiv&
\resizebox{!}{.5cm}{  \begin{pspicture}(0,.9)(2,1.8) 
 \pscircle*(1,1){.15}  \psline[linewidth=3pt](0,1)(2,1)
 \psline(.5,1)(.2,1.6)  \psline(.5,1)(.4,1.6) \psline(.5,1)(.6,1.6) \psline(.5,1)(.8,1.6)
 \end{pspicture} } 
 \otimes \left(
\resizebox{!}{.5cm}{  \begin{pspicture}(0,.9)(2,1.8) 
 \pscircle*(1,1){.15}  \psline[linewidth=3pt](0,1)(2,1)
 \psline(1.5,1)(1.3,1.3)  \psline(1.5,1)(1.95,1.9) \psline(1.65,1.3)(1.45,1.6) \psline(1.8,1.6)(1.6,1.9)
 \end{pspicture} }  
 +
\resizebox{!}{.5cm}{  \begin{pspicture}(0,.9)(2,1.8)  
 \pscircle*(1,1){.15}  \psline[linewidth=3pt](0,1)(2,1)
 \psline(1.7,1)(1.5,1.5)  \psline(1.7,1)(2.1,1.8) \psline(1.9,1.4)(1.7,1.8) \psline(1.45,1)(1.25,1.5)
 \end{pspicture} }  
 +
\resizebox{!}{.5cm}{  \begin{pspicture}(0,.9)(2,1.8)  
 \pscircle*(1,1){.15}  \psline[linewidth=3pt](0,1)(2,1)
\psline(1.3,1)(1.1,1.6) \psline(1.5,1)(1.3,1.6) \psline(1.7,1)(1.5,1.6) \psline(1.7,1)(1.9,1.6)
 \end{pspicture} }  
\right) \\ 
&& + 
\resizebox{!}{.5cm}{  \begin{pspicture}(0,.9)(2,1.8) 
 \pscircle*(1,1){.15}  \psline[linewidth=3pt](0,1)(2,1)
 \psline(.5,1)(.1,1.8) \psline(.5,1)(.75,1.5) \psline(.3,1.4)(.3,1.8) \psline(.3,1.4)(.5,1.8)
 \end{pspicture} }  
\otimes
\resizebox{!}{.5cm}{  \begin{pspicture}(0,.9)(2,1.8) 
 \pscircle*(1,1){.15}  \psline[linewidth=3pt](0,1)(2,1)
 \psline(1.3,1)(1.3,1.6) \psline(1.7,1)(1.5,1.6)  \psline(1.7,1)(1.9,1.6) \psline(1.7,1)(1.7,1.6)
 \end{pspicture} } 
\quad\quad +
 \scalebox{-1}[1]{
\resizebox{!}{.5cm}{  \begin{pspicture}(0,.9)(2,1.8) 
 \pscircle*(1,1){.15}  \psline[linewidth=3pt](0,1)(2,1)
 \psline(1.3,1)(1.3,1.6) \psline(1.7,1)(1.5,1.6)  \psline(1.7,1)(1.9,1.6) \psline(1.7,1)(1.7,1.6)
 \end{pspicture} } }
 \otimes
 \scalebox{-1}[1]{
\resizebox{!}{.5cm}{  \begin{pspicture}(0,.9)(2,1.8) 
 \pscircle*(1,1){.15}  \psline[linewidth=3pt](0,1)(2,1)
 \psline(.5,1)(.1,1.8) \psline(.5,1)(.75,1.5) \psline(.3,1.4)(.3,1.8) \psline(.3,1.4)(.5,1.8)
 \end{pspicture} }   }
  \\
  && + 
\resizebox{!}{.5cm}{  \begin{pspicture}(0,.9)(2,1.8) 
 \pscircle*(1,1){.15}  \psline[linewidth=3pt](0,1)(2,1)
 \psline(.5,1)(.25,1.5) \psline(.5,1)(.75,1.5) \psline(.5,1)(.5,1.5)
 \psline(.5,1.5)(.3,1.8)  \psline(.5,1.5)(.7,1.8)
 \end{pspicture} }  
\otimes
\resizebox{!}{.5cm}{  \begin{pspicture}(0,.9)(2,1.8) 
 \pscircle*(1,1){.15}  \psline[linewidth=3pt](0,1)(2,1)
 \psline(1.3,1)(1.3,1.6) \psline(1.7,1)(1.5,1.6)  \psline(1.7,1)(1.9,1.6) \psline(1.7,1)(1.7,1.6)
 \end{pspicture} } 
\quad\quad +
 \scalebox{-1}[1]{
\resizebox{!}{.5cm}{  \begin{pspicture}(0,.9)(2,1.8) 
 \pscircle*(1,1){.15}  \psline[linewidth=3pt](0,1)(2,1)
 \psline(1.3,1)(1.3,1.6) \psline(1.7,1)(1.5,1.6)  \psline(1.7,1)(1.9,1.6) \psline(1.7,1)(1.7,1.6)
 \end{pspicture} } }
 \otimes
 \scalebox{-1}[1]{
\resizebox{!}{.5cm}{  \begin{pspicture}(0,.9)(2,1.8) 
 \pscircle*(1,1){.15}  \psline[linewidth=3pt](0,1)(2,1)
 \psline(.5,1)(.25,1.5) \psline(.5,1)(.75,1.5) \psline(.5,1)(.5,1.5)
 \psline(.5,1.5)(.3,1.8)  \psline(.5,1.5)(.7,1.8)
 \end{pspicture} }  }
 \\
&& +
\resizebox{!}{.5cm}{  \begin{pspicture}(0,.9)(2,1.8) 
 \pscircle*(1,1){.15}  \psline[linewidth=3pt](0,1)(2,1)
 \psline(.5,1)(.25,1.5) \psline(.5,1)(.75,1.5) \psline(.5,1)(.5,1.5)
 \psline(.25,1.5)(.05,1.8)  \psline(.25,1.5)(.45,1.8)
 \end{pspicture} }  
\otimes
\resizebox{!}{.5cm}{  \begin{pspicture}(0,.9)(2,1.8) 
 \pscircle*(1,1){.15}  \psline[linewidth=3pt](0,1)(2,1)
 \psline(1.35,1)(1.2,1.6) \psline(1.35,1)(1.5,1.6)  \psline(1.8,1)(1.65,1.6) \psline(1.8,1)(1.95,1.6)
 \end{pspicture} } 
\quad\quad +
 \scalebox{-1}[1]{
\resizebox{!}{.5cm}{  \begin{pspicture}(0,.9)(2,1.8) 
 \pscircle*(1,1){.15}  \psline[linewidth=3pt](0,1)(2,1)
 \psline(1.35,1)(1.2,1.6) \psline(1.35,1)(1.5,1.6)  \psline(1.8,1)(1.65,1.6) \psline(1.8,1)(1.95,1.6)
 \end{pspicture} } }
 \otimes
 \scalebox{-1}[1]{
\resizebox{!}{.5cm}{  \begin{pspicture}(0,.9)(2,1.8) 
 \pscircle*(1,1){.15}  \psline[linewidth=3pt](0,1)(2,1)
 \psline(.5,1)(.25,1.5) \psline(.5,1)(.75,1.5) \psline(.5,1)(.5,1.5)
 \psline(.25,1.5)(.05,1.8)  \psline(.25,1.5)(.45,1.8)
 \end{pspicture} }  }
  \\
&&+\left(
 \scalebox{-1}[1] { 
\resizebox{!}{.5cm}{  \begin{pspicture}(0,.9)(2,1.8)  
 \pscircle*(1,1){.15}  \psline[linewidth=3pt](0,1)(2,1)
\psline(1.3,1)(1.1,1.6) \psline(1.5,1)(1.3,1.6) \psline(1.7,1)(1.5,1.6) \psline(1.7,1)(1.9,1.6)
 \end{pspicture} }   }
 +
 \scalebox{-1}[1] {
\resizebox{!}{.5cm}{  \begin{pspicture}(0,.9)(2,1.8)  
 \pscircle*(1,1){.15}  \psline[linewidth=3pt](0,1)(2,1)
 \psline(1.7,1)(1.5,1.5)  \psline(1.7,1)(2.1,1.8) \psline(1.9,1.4)(1.7,1.8) \psline(1.45,1)(1.25,1.5)
 \end{pspicture} }    }
 +
 \scalebox{-1}[1] {
 \resizebox{!}{.5cm}{  \begin{pspicture}(0,.9)(2,1.8) 
 \pscircle*(1,1){.15}  \psline[linewidth=3pt](0,1)(2,1)
 \psline(1.5,1)(1.3,1.3)  \psline(1.5,1)(1.95,1.9) \psline(1.65,1.3)(1.45,1.6) \psline(1.8,1.6)(1.6,1.9)
 \end{pspicture} }  }
  \right)\otimes 
 \scalebox{-1}[1] {
\resizebox{!}{.5cm}{  \begin{pspicture}(0,.9)(2,1.8) 
 \pscircle*(1,1){.15}  \psline[linewidth=3pt](0,1)(2,1)
 \psline(.5,1)(.2,1.6)  \psline(.5,1)(.4,1.6) \psline(.5,1)(.6,1.6) \psline(.5,1)(.8,1.6)
 \end{pspicture} } }  \\
 && \hspace{2in} (\text{mod }\{I_{k,l}\,\,:\,\,k+l>0 \})
\end{eqnarray*}
\end{example}

\begin{remark}\label{REM:strong}
In Example \ref{Delta-for-cyclic} we observed that the tensor product of cyclic $A_\infty$-algebras is, in general, not cyclic. One could also consider strong homotopy inner products considered by Cho \cite{C}, which are cyclic $A_\infty$-algebras up to homotopy, and ask whether or not the tensor product preserves such structures. In \cite[Theorem 5.1]{C}, Cho showed that a homotopy inner product transforms into a strong homotopy inner product if and only if it satisfies the following three conditions:
\begin{enumerate}
\item Skew Symmetry: \quad $\rho_{k,l}(a, \underline b,c, \underline d)=\pm \rho_{l,k}(c, \underline d, a, \underline b)$,
\item Closedness: \quad 
\begin{multline*}
\quad \quad \quad \rho_{k+l+1,m}(\dots,a, \dots, \underline b,\dots, \underline c)\pm \rho_{k,l+m+1}(\dots, \underline a,\dots,b,\dots, \underline c)\\
\pm \rho_{l+m+1,k}(\dots,c, \dots, \underline a,\dots, \underline b)=0,
\end{multline*}
\item Homological non-degeneracy: \quad $(\rho_{0,0})_*$ is non-degenerate.
\end{enumerate}
For more details on the notation and signs, we refer the reader to \cite{C} or \cite{T2}. Now the symmetrical nature of the definition of $\Delta_C$ implies that the tensor product of two skew-symmetric homotopy inner products (satisfying condition (1)) is also skew-symmetric: If $S^*$ denotes the diagram obtained from $S$ by rotating $180^\circ$, then $(S^{*})_{\max}=(S_{\max})^{*}$, $(D^{*})_{\min}=(D_{\min})^{*}$, and $S_{\max}\leq D_{\min} \Leftrightarrow (S_{\max})^*\leq (D_{\min})^*$. 
Furthermore, under reasonable conditions, it is clear that the tensor product of two homologically non-degenerate homotopy inner products (satisfying property (3)) is also homologically non-degenerate.  Thus we conjecture that tensor products also preserve property (2), and that the tensor product is closed in the strong homotopy inner product category. 
\end{remark}

\begin{acknowledgments} 
We wish to thank Jean-Louis Loday and Jim Stasheff for sharing their thoughts and insights with us during discussions related to this topic.
\end{acknowledgments}

\medskip 

\appendix

\section{Signs}\label{APP:signs}
In this appendix we discuss various issues related to the definition and calculation of signs in this paper. To begin, we review the signs in an $A_\infty$-algebra with homotopy inner products, define the canonical orientation of a binary diagram, and check the signs in Proposition \ref{q-bound}.

\smallskip 

\subsection{Signs in an $A_\infty$-algebra with homotopy inner product}\label{APP:A-infty-sign}
Let $A=\bigoplus_{i\in \Z} A_i$ be a differential graded $R$-module (DGM) with differential $d:A\to A$ of degree $+1$.  The tensor product of DGM maps $f$ and $g$ satisfies $(f\otimes g)(a\otimes b)=(-1)^{|g|\cdot |a|}f(a)\otimes g(b)$. If $(A^1,d^1), \dots, (A^{k+1},d^{k+1})$ are DGMs, the induced differential $d$ on $A^1\otimes \cdots \otimes A^k$ is given by 
\[
d = \sum_{i=1}^{k} \mathbf{1}^{\otimes (i-1)} \otimes d^i \otimes \mathbf{1}^{\otimes (k-i)},
\]
and the induced differential $D$ on $Hom(A^1\otimes \cdots \otimes A^k,A^{k+1})$ is given by the commutator
\[
[D,f]=d^{k+1} f-(-1)^{|f|} f d.
\]

An $A_\infty$-\emph{algebra structure on} $A$ consists of a family of maps $\{\mu_k:A^{\otimes k}\to A\}_{k\geq 2}$ such that $|\mu_k|=2-k$ and
\begin{equation}\label{EQ:A-infty-sign}
[D,\mu_k] 
= \sum_{ j+\ell=k+1} \sum_{i=1}^{k-j+1} (-1)^{i (j+1)+j \ell}\thinspace \mu_\ell \circ \left( \mathbf{1}^{\otimes(i-1)}\otimes \mu_j \otimes \mathbf{1}^{\otimes(k-j-i+1)}\right).
\end{equation}
 
\begin{remark}
There are various choices of signs in the $A_\infty$-algebra structure relations. For example, one could define an $A_\infty$-algebras in terms of maps $\mu'_k:A^{\otimes k}\to A$ such that 
\begin{equation*}
[d,\mu'_k] 
= \sum_{ j+\ell-1=k} \sum_{i=1}^{k-j+1} (-1)^{(i-1)\cdot (j+1)+ \ell} \cdot \mu'_\ell \circ \left(\mathbf{1}^{\otimes (i-1)} \otimes \mu'_j \otimes \mathbf{1}^{\otimes (k-j-i+1)}
\right).
\end{equation*}
The two definitions are related via the relation $\mu'_k=(-1)^{\frac{k(k+1)}{2}+1} \cdot \mu_k$.
Note that the simplest signs arise by shifting dimension in $A$ up by $1$ and removing all signs. We refer the reader \emph{e.g.} to \cite{T} and \cite{SU} for details.
\end{remark}

An $\infty$-\emph{bimodule over} $A$ consists of a DGM $M$ together with a family of module maps $\{\lambda_{j',j''}:A^{\otimes j'}\otimes M\otimes A^{\otimes j''}\to M\}$ such that
\begin{equation}\label{EQ:A-module-sign}
[D, \lambda_{k',k''}] 
= \sum_{j+\ell=k+1} \sum_{i=1}^{k-j+1} (-1)^{i (j+1)+j \ell}\thinspace \lambda_{\ell',\ell''} \left( \mathbf{1}^{\otimes(i-1)} \otimes q_J \otimes \mathbf{1}^{\otimes(k-j-i+1)}\right),
\end{equation}
where $k=k'+k''+1$, $\ell=\ell'+\ell''+1$, and $q_J$ is either $\lambda_{j',j''}$ or $\mu_j$, which is determined by the $i\T$ to the $(k-j+i)\T$ input in $A^{\otimes k'}\otimes M\otimes A^{\otimes k''}$.
An important example of an $\infty$-bimodule is given by setting $M=A$ and 
\begin{equation}\label{EQU:canon-example}
 \lambda_{j',j''}=\mu_{j'+j''+1}.
\end{equation}
This example also helps to clarify the signs in the formula above.

Finally, given $\infty$-bimodule over $A$, a \emph{homotopy inner product} consists of a family of maps $\{\varrho_{j',j''}:M\otimes A^{\otimes j'}\otimes M\otimes A^{\otimes j''}\to R\}$, such that 
\begin{multline}\label{EQ:Inner-product-sign}
[D, \varrho_{k',k''}] = \\
\sum_{j'+j''+\ell=k} (-1)^{(j+1)+j \ell+j' (j''+\ell)} \thinspace\varrho_{\ell',\ell''} \left(\lambda_{j',j''} \otimes \mathbf{1}^{\otimes (k-j'-j''-1)}\right) \underbrace{(\tau_\#\circ \cdots\circ\tau_\#)}_{j' \text{ cyclic permut.}}
\\
+\sum_{j+\ell-1=k} \sum_{i=2}^{k-j+1} (-1)^{i (j+1)+j \ell} \thinspace\varrho_{\ell',\ell''} \left( \mathbf{1}^{\otimes (i-1)}\otimes q_J \otimes \mathbf{1}^{\otimes (k-j-i+1)} \right),
\end{multline}
where $k=k'+k''+2$, $\ell=\ell'+\ell''+2$, and $q_J$ is either $\lambda_{j',j''}$ or $\mu_j$ depending on the inputs.  The first line of the formula involves composition in the first position ($i=1$) after cyclically permuting $j'$ elements from back-to-front, i.e,  
\[
\tau_\#: A^{\otimes i'}\otimes M\otimes A^{\otimes i''}\otimes M\otimes A^{\otimes i'''+1}\to A^{\otimes i'+1}\otimes M\otimes A^{\otimes i''}\otimes M\otimes A^{\otimes i'''},
\] 
then applying $\lambda_{j',j''}$.  This cyclical rotation of $j'$ elements gives rise to the additional sign coefficient here.  

\smallskip

The signs \eqref{EQ:A-infty-sign}, \eqref{EQ:A-module-sign}, and \eqref{EQ:Inner-product-sign}, which appear in the definitions of an $A_\infty$-algebra and a homotopy inner product, also appear in the definition of $\CA$. Let $\mathcal End_{(A,M,R)}$ denote the endomorphism operad of the triple $(A,M,R)$.  Then under the conventions above, a morphism of operads $\CA\to \mathcal End_{(A,M,R)}$ defines an $A_\infty$-algebra $A$ with homotopy inner product structure on $M$, as is evident in the next definition.

\begin{definition}\label{DEF:F:CA->End}
Given an $A_\infty$-algebra with $\infty$-bimodule and homotopy inner product structures $(A,M,R,\{\mu_i\}_i,\{\lambda_{i,j}\}_{i,j}, \{\varrho_{i,j}\}_{i,j})$, define the operad map $F:\CA\to \mathcal End_{(A,M,R)}$ as follows: For a corolla $c=T_i, M_{i,j}$ or $I_{i,j}$ (see Figure \ref{F:3-corollas}), the canonical clockwise assignment $\fcan$ of inputs (see Figure \ref{FIG:canon-leaf-order}), and the canonical orientation $\omega=+1$ of the empty set of edges, define $F(c,\fcan,+1)$ to be the structure associated with this corolla, i.e., 
\[
F(T_i,\fcan,+1)=\mu_i, \quad F(M_{i,j},\fcan,+1)=\lambda_{i,j}, \quad F(I_{i,j},\fcan,+1)=\varrho_{i,j}.
\]
For a diagram $D$ with exactly one edge $e$ and the canonical labeling $\fcan$, the signs in $F$ are determined by \eqref{EQ:A-infty-sign}, \eqref{EQ:A-module-sign}, and \eqref{EQ:Inner-product-sign}. For example, if $D\in \CA^{11\cdots 1}_{1}$ is a tree and the edge $e=e(i,j)$ determines the subtree $T_j$ attached to $T_\ell$ at position $i$, then, from \eqref{EQ:A-infty-sign},
\[
F(D,\fcan,e(i,j))=d_{e(i,j)}( \mu_k) := (-1)^{i(j+1)+j \ell} \mu_\ell (1^{\otimes i-1}\otimes \mu_j\otimes 1^{\otimes \ell-i}).
\]

Similarly, for a general diagram $D$ and any edge $e\in \mathcal E(D)$, we have an operation $d_e$ of degree $+1$,
and set
\[
F(D,\fcan,e_1\wedge\dots\wedge e_r) = d_{e_1} \circ\dots\circ d_{e_r} (q_J),
\]
where $q_J$ is one of the maps $\mu_i, \lambda_{i,j}$, or $\varrho_{i,j}$. When $d_e$ passes a structure map $q'_{J'}$, the usual Koszul sign commutation rule applies: $d_e \circ q'_{J'}=(-1)^{|q'_{J'}|}\thinspace q'_{J'} \circ d_e$. 

Finally, for a non-trivial labeling $f:\{1,\dots,k\}\to \mathcal L(D)$, we uniquely write $f$ as a composition of a permutation $\sigma\in S_k$ and the clockwise labeling, $f= \fcan\circ \sigma$, and denoting by $\sigma_\#(x_1\otimes\dots\otimes x_k)=x_{\sigma^{-1}(1)}\otimes\dots\otimes x_{\sigma^{-1}(k)}$ a permutation of tensor factors, we set 
\[ F(D,\fcan\circ \sigma,\omega)=\sgn(\sigma)\cdot F(D,\fcan,\omega)\circ \sigma_\#. \]
\end{definition}
\noindent With this, one can check that the signs appearing in Equation \eqref{EQ:circ-i} make $F$ into an operad map. For example, for $(D,\fcan,\omega_D)\in C_n\hat{\mathcal A}^{11\cdots 1}_1,\text{ } k=\#\mathcal L(D)$, and $(E,id,\omega_E)\in C_m\hat{\mathcal A}^{11\cdots 1}_1,\text{ } l=\#\mathcal L(E)$, we have \vspace{.1in}

\noindent $F(( D,\fcan, e_1^D\wedge\dots \wedge e_{k-n-2}^D)  \circ_{i}(  E,\fcan,e_1^E\wedge\dots\wedge e_{l-m-2}^E))$
\begin{eqnarray*}
&=& (-1)^{km+i(l+1)}F \left(  D\circ_{i}E,\fcan, e_1^D\wedge\dots \wedge e_{k-n-2}^D\wedge e_1^E\wedge\dots\wedge e_{l-m-2}^E \wedge e \right)
\\
&=& (-1)^{km+i(l+1)} d_{e_1^D}\circ\dots \circ d_{e_{k-n-2}^D}\circ d_{e_1^E}\circ\dots\circ d_{e_{l-m-2}^E} \circ d_e (\mu_{k+l-1})
\\
&=& (-1)^{k(m+ l)} d_{e_1^D}\circ\dots \circ d_{e_{k-n-2}^D}\circ d_{e_1^E}\circ\dots\circ d_{e_{l-m-2}^E} (\mu_{k}\circ (1^{\otimes i-1}\otimes \mu_l\otimes 1^{\otimes l-j}))
\\
&=& d_{e_1^D}\circ\dots \circ d_{e_{k-n-2}^D} \big(\mu_{k}\circ (1^{\otimes i-1}\otimes d_{e_1^E}\circ\dots\circ d_{e_{l-m-2}^E} (\mu_l) \otimes 1^{\otimes l-j})\big)
\\
&=& F( D,\fcan, e_1^D\wedge\dots \wedge e_{k-n-2}^D) \circ_i F( E,\fcan,e_1^E\wedge\dots\wedge e_{l-m-2}^E).
\end{eqnarray*}
A similar calculation applies in the other cases.

\smallskip

\subsection{Orientation on binary trees}\label{APP:binary-standard-orientation}
In this appendix, we describe the canonical orientation of a binary diagram, either as element of $C_{\ast}\hat{\mathcal{A}}$ or as a non-metric element of $Q_{\ast}\hat{\mathcal{A}}$. The main ingredient is an extension of Mac Lane's Coherence Theorem \cite{MacL} to homotopy inner products, by which any two paths of binary diagrams are connected via sequences of pentagons, hexagons, and squares.

\begin{definition}\label{DEF:local-move}
A \textbf{path of binary diagrams} is a sequence of binary diagrams $\beta=(B_1,\dots,B_n)$ such that $(B_i,B_{i+1})$ is an edge-pair for all $i$ (see Definition \ref{DEF:order}).
\end{definition}

\noindent
We consider paths up to equivalence, where the equivalence relation is generated by,
\[
(\dots,B_{i-1},B_i,B_{i+1},\dots) \sim(\dots,B_{i-1},B_i,\widetilde B_i,B_i,B_{i+1},\dots),
\]
where $\tilde B_i$ is related to $B_i$ by a local move.

For example, the boundaries of the pentagons $\A_4, M_{0,3}, M_{1,2}, M_{2,1}, M_{3,0}, I_{2,0}, I_{0,2}$  and the hexagon $I_{1,1}$ are paths for any choice of starting/ending point.  Furthermore, two disjoint local moves define a square path (called a \emph{naturality square}) by applying move 1, then move 2, then undoing move 1, and undoing move 2.  A \emph{fundamental path} is a naturality square or the boundary paths of one of the aforementioned corollas. 

\begin{definition}
Two (equivalence classes of) paths are \textbf{one-step-connected} if they differ (locally) by a fundamental path. Two paths $\beta_1$ and $\beta_m$ are \textbf{connected} if there is a sequence of paths $\beta_1,\dots, \beta_m$ such that $\beta_i$ and $\beta_{i+1}$ are one-step-connected for all $i$.
\end{definition}

For example, consider a diagram $P_1$ and a sequence of diagrams \linebreak $(P_1,P_2,P_3,P_4,P_5,P_1)$ that differ locally from $P_1$ by a fundamental path, which is the boundary of the Stasheff pentagon. Then the following paths $\beta_1$ and $\beta_2$ are one-step-connected: $\beta_1=(B_1,\dots,B_{k-1},P_1,B_{k+1},\dots, B_n)$  and  $\beta_2=(B_1,\dots,B_{k-1},\linebreak P_1, P_2,P_3,P_4,P_5,P_1,B_{k+1},\dots, B_n)$. 

\begin{lemma}[Coherence Lemma]\label{LEM:coherence}
Any two paths $\beta=(B_1,\dots,B_n)$ and $\beta'=(B'_1,\dots,B'_{n'})$ such that $B_1=B'_1$ and $B_n=B'_{n'}$ are connected.
\end{lemma}
The lemma can be proved in the same manner as the ``associative Coherence Lemma'' in \cite[Section VII.2.]{MacL}. Thus to define a concept on binary diagrams via paths, it is sufficient to check that the definition is independent of path with respect to fundamental paths.  Let us do this for the notion of the standard orientation.  First note that each local move from $B\leadsto B'$ (as given by $(1)$-$(6)$ in Definition \ref{DEF:order} with its induced identification of edges) transfers an orientation $\omega$ from $B$ to $B'$ by setting $\omega=e_1\wedge\dots\wedge e_k=-\omega'$. One can check that  traversing any of the fundamental paths preserves orientation (see Figure \ref{FIG:pent-hex-orient} for example).
\begin{figure}[ht]
\[ 
 \begin{pspicture}(0,-.6)(1.2,1) \rput(.6,-.4){$e\wedge e'$} 
 \psline(.6,0)(.6,.4) \psline(.6,.4)(0,1) \psline(.6,.4)(1.2,1) \rput(.3,.4){$e$}  \rput(1,.5){$e'$}  
 \psline(.2,.8)(.4,1) \psline(1,.8)(.8,1)
 \end{pspicture} 
\,\, \twig{-1}\,\, %
 \begin{pspicture}(0,-.6)(1.2,1) \rput(.5,-.4){$-e\wedge e'$} 
\psline(.6,0)(.6,.4) \psline(.6,.4)(0,1) \psline(.6,.4)(1.2,1)\rput(.8,.4){$e$} \rput(1.05,.6){$e'$}  
 \psline(.8,.6)(.4,1) \psline(1,.8)(.8,1)
 \end{pspicture} 
\,\, \twig{-1}\,\, %
 \begin{pspicture}(0,-.6)(1.2,1) \rput(.6,-.4){$e\wedge e'$} 
\psline(.6,0)(.6,.4) \psline(.6,.4)(0,1) \psline(.6,.4)(1.2,1)\rput(.8,.4){$e$} \rput(.85,.85){$e'$}  
 \psline(.8,.6)(.4,1) \psline(.6,.8)(.8,1)
 \end{pspicture} 
\,\, \twig{-1}\,\, %
 \begin{pspicture}(0,-.6)(1.2,1) \rput(.5,-.4){$-e\wedge e'$} 
 \psline(.6,0)(.6,.4) \psline(.6,.4)(0,1) \psline(.6,.4)(1.2,1) \rput(.4,.4){$e$}  \rput(.7,.75){$e'$}  
 \psline(.6,.8)(.4,1) \psline(.4,.6)(.8,1)
 \end{pspicture} 
\,\, \twig{-1}\,\, %
 \begin{pspicture}(0,-.6)(1.2,1) \rput(.6,-.4){$e\wedge e'$}  
 \psline(.6,0)(.6,.4) \psline(.6,.4)(0,1) \psline(.6,.4)(1.2,1) \rput(.4,.4){$e$}  \rput(.15,.55){$e'$}  
 \psline(.2,.8)(.4,1)  \psline(.4,.6)(.8,1)
 \end{pspicture} 
\,\, \twig{-1}\,\, %
 \begin{pspicture}(0,-.6)(1.2,1) \rput(.5,-.4){$-e\wedge e'$} 
 \psline(.6,0)(.6,.4) \psline(.6,.4)(0,1) \psline(.6,.4)(1.2,1) \rput(.2,.4){$e'$}  \rput(1,.5){$e$}  
 \psline(.2,.8)(.4,1) \psline(1,.8)(.8,1)
 \end{pspicture} 
\,\, {\begin{pspicture}(0,-1)(0.4,0.2) \rput(0.2,0){$=$} \end{pspicture}} \,\, %
 \begin{pspicture}(0,-.6)(1.2,1) \rput(.6,-.4){$e\wedge e'$} 
 \psline(.6,0)(.6,.4) \psline(.6,.4)(0,1) \psline(.6,.4)(1.2,1) \rput(.3,.5){$e$}  \rput(1,.5){$e'$}  
 \psline(.2,.8)(.4,1) \psline(1,.8)(.8,1)
 \end{pspicture} 
 \]\[ 
 \begin{pspicture}(0,-.6)(1.2,1.4) \rput(.6,-.4){$e\wedge e'$} 
  \psline[linewidth=2pt](0,.2)(1.2,.2) \pscircle*(.6,.2){.1} \rput(.75,0){$e$}  \rput(.5,.5){$e'$}  
 \psline(.4,.2)(.2,.6) \psline(.8,.2)(1,-.2)
 \end{pspicture} 
\,\, \twig{-.8}\,\, %
 \begin{pspicture}(0,-.6)(1.2,1) \rput(.5,-.4){$-e\wedge e'$} 
  \psline[linewidth=2pt](0,.2)(1.2,.2) \pscircle*(.6,.2){.1} \rput(.9,0){$e$}  \rput(.75,.5){$e'$}  
 \psline(.8,.2)(1,.6) \psline(1,.2)(1.2,-.2)
 \end{pspicture} 
\,\, \twig{-.8}\,\, %
 \begin{pspicture}(0,-.6)(1.2,1) \rput(.6,-.4){$e\wedge e'$} 
  \psline[linewidth=2pt](0,.2)(1.2,.2) \pscircle*(.6,.2){.1} \rput(1.05,0){$e$}  \rput(.75,.5){$e'$}  
 \psline(1,.2)(1.2,.6) \psline(.8,.2)(1,-.2)
 \end{pspicture} 
\,\, \twig{-.8}\,\, %
 \begin{pspicture}(0,-.6)(1.2,1) \rput(.5,-.4){$-e\wedge e'$} 
  \psline[linewidth=2pt](0,.2)(1.2,.2) \pscircle*(.6,.2){.1} \rput(.75,0){$e$}  \rput(.5,.5){$e'$}  
 \psline(.8,.2)(1,.6) \psline(.4,.2)(.2,-.2)
 \end{pspicture} 
\,\, \twig{-.8}\,\, %
 \begin{pspicture}(0,-.6)(1.2,1) \rput(.6,-.4){$e\wedge e'$}  
  \psline[linewidth=2pt](0,.2)(1.2,.2) \pscircle*(.6,.2){.1} \rput(.45,0){$e$}  \rput(.1,.4){$e'$}  
 \psline(.4,.2)(.2,.6) \psline(.2,.2)(0,-.2)
 \end{pspicture} 
\,\, \twig{-.8}\,\, %
 \begin{pspicture}(0,-.6)(1.2,1) \rput(.5,-.4){$-e\wedge e'$} 
  \psline[linewidth=2pt](0,.2)(1.2,.2) \pscircle*(.6,.2){.1} \rput(.45,0){$e$}  \rput(.3,.5){$e'$}  
 \psline(.2,.2)(0,.6) \psline(.4,.2)(.2,-.2)
 \end{pspicture} 
\,\, \twig{-.8}\,\, %
 \begin{pspicture}(0,-.6)(1.2,1) \rput(.6,-.4){$e\wedge e'$} 
  \psline[linewidth=2pt](0,.2)(1.2,.2) \pscircle*(.6,.2){.1} \rput(.75,0){$e$}  \rput(.5,.5){$e'$}  
 \psline(.4,.2)(.2,.6) \psline(.8,.2)(1,-.2)
 \end{pspicture} 
\]
\caption{Orientation is preserved along boundaries of corollas $\A_4$ (top line) and $I_{1,1}$ (bottom line).}\label{FIG:pent-hex-orient}
\end{figure}
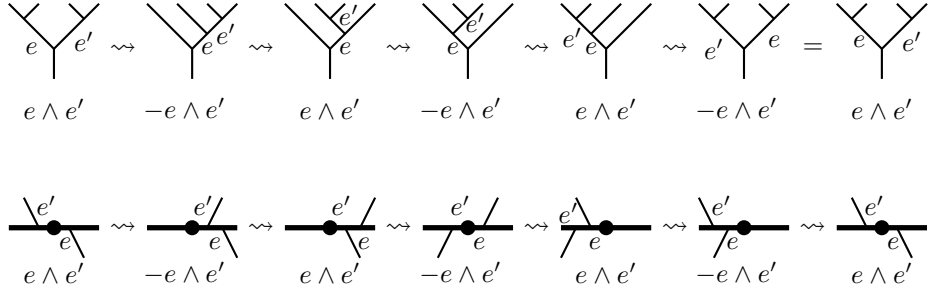

\noindent
Thus if $\beta=(B,\dots,B')$ is a path from diagram $B$ to diagram $B'$, 
an orientation $\omega$ on $B$ transfers to an orientation $\omega_\beta$ on $B'$.  And furthermore, Lemma \ref{LEM:coherence} assures us that $\omega_\beta$ is independent of path since it is preserved along paths coming from fundamental paths.  Let us use this idea to define the standard orientation $\ost_B$ on a binary diagram $B$.


\begin{definition}
Let $B$ be a binary diagram of the three types pictured in Figure \ref{FIG:standard-orientation}.  Define the \textbf{standard orientation} $\ost_{B}$ on $B$ by 
\[
\ost_{B}:=e_1\wedge \dots \wedge e_k \in \bigwedge^k \mathcal E(B)
\] 
The standard orientation on a general binary diagram $B$ is induced by any path from one of the diagrams in Figure \ref{FIG:standard-orientation} to $B$.
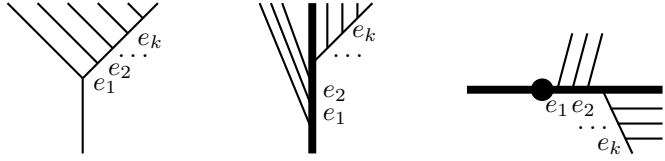
\begin{figure}[ht]
\[ 
 \begin{pspicture}(2,0)(4,2)
 \rput(3.3,.9){$e_1$}  \rput(3.5,1.1){$e_2$}   \rput(3.7,1.3){$\dots$} \rput(3.9,1.5){$e_k$}
 \psline(3,0)(3,1) \psline(3,1)(2,2) \psline(3,1)(4,2)
 \psline(3.8,1.8)(3.6,2)  \psline(3.6,1.6)(3.2,2)  \psline(3.4,1.4)(2.8,2)  \psline(3.2,1.2)(2.4,2)
 \end{pspicture} 
 \quad \quad \quad
 \begin{pspicture}(2,0)(4,2)
 \rput(3.3,.5){$e_1$}  \rput(3.3,.8){$e_2$}   \rput(3.5,1.3){$\dots$} \rput(3.7,1.6){$e_k$}
 \psline[linewidth=3pt](3,0)(3,2) \psline(3,1.2)(3.8,2)
 \psline(3,.9)(2.6,2)  \psline(3,.6)(2.45,2)   \psline(3,.3)(2.3,2) 
 \psline(3.2,1.4)(3.2,2) \psline(3.4,1.6)(3.4,2) \psline(3.6,1.8)(3.6,2) 
 \end{pspicture} 
 \quad \quad \quad
 \begin{pspicture}(2,0.4)(4.6,2)
 \rput(3.2,1){$e_1$}  \rput(3.55,1){$e_2$}   \rput(3.7,.75){$\dots$} \rput(3.9,.5){$e_k$}
 \psline[linewidth=3pt](2,1.25)(4.6,1.25) \pscircle*(3,1.25){.15}
 \psline(3.2,1.25)(3.4,2) \psline(3.4,1.25)(3.6,2) \psline(3.6,1.25)(3.8,2) 
 \psline(3.8,1.25)(4.2,.4) 
 \psline(3.9,1)(4.6,1) \psline(4,.8)(4.6,.8) \psline(4.1,.6)(4.6,.6)
 \end{pspicture}
\]\caption{Diagrams with standard orientation $e_1\wedge\dots\wedge e_k$}\label{FIG:standard-orientation}
\end{figure}
\end{definition}

\begin{example}
\label{EXA:I_kl-orientation}
As an example, we calculate the standard orientation of the diagram $(I_{k,l})_{\max}$:
\[
\resizebox{!}{1.7cm}{  \begin{pspicture}(-.2,.5)(4,2)  
 \psline[linewidth=2pt](2,1.25)(4,1.25) \pscircle*(3,1.25){.12} 
 \rput(.6,1.3){$(I_{k,l})_{\max}=$}
 \rput(3.5,1){$_{e_1 \dots e_{k}}$} \rput(2.5,1.5){$_{e_{k+l} \dots e_{k+1}}$}
\psline(3.2,1.25)(3.4,2)\psline(3.4,1.25)(3.6,2)\psline(3.6,1.25)(3.8,2)\psline(3.8,1.25)(4,2)
\psline(2.8,1.25)(2.6,.5)\psline(2.6,1.25)(2.4,.5)\psline(2.4,1.25)(2.2,.5)\psline(2.2,1.25)(2,.5)  \end{pspicture}}
\]
where $k$ leaves are attached at the top right and $l$ leaves are attached at the bottom left, and we labeled the edges by $e_1,\dots,e_{k+l}$ as shown. We claim that the standard orientation is:
\[\ost_{(I_{k,l})_{\max}}=(-1)^l\cdot e_1\wedge\dots\wedge e_{k+l}\]
\end{example}
\begin{proof}
Starting from the orientation of the right diagram in Figure \ref{FIG:standard-orientation}, we move the lower branch across the upper $k$ leaves, with the following induced orientation:
\[
\resizebox{!}{2.2cm}{  \begin{pspicture}(-1,-1.6)(3,.6)
 \psline[linewidth=2pt](-1,0)(3,0) \pscircle*(0,0){.12} 
 \rput(.65,-.2){$_{e_1 \dots \, e_{k}}$} 
\psline(.4,0)(.6,.6)\psline(.6,0)(.8,.6)\psline(.8,0)(1,.6)\psline(1,0)(1.2,.6)
\psline(1.4,0)(1.8,-1.2)  \rput(1.45,.15){$_{e_{k+1}}$}
\psline(1.74,-1)(2.74,-1) \rput(1.8,-.2){$_{e_{k+2}}$}
\psline(1.5,-.33)(2.5,-.33) \rput(1.8,-.5){$_{\dots}$}
\psline(1.62,-.66)(2.62,-.66) \rput(2.05,-.85){$_{e_{k+l}}$}
\rput(0.2,-1.5){$e_1\wedge\dots\wedge e_{k+l}$}
 \end{pspicture}}
\resizebox{!}{2.2cm}{  \begin{pspicture}(-3.6,-1.6)(3,.6) 
 \psline[linewidth=2pt](-1,0)(3,0) \pscircle*(0,0){.12}  \rput(-2.4,0){$\leadsto\dots\leadsto$}
 \rput(2.52,-.2){$_{e_3 \dots \, e_{k+1}}$}  \rput(1.5,.15){$_{e_2}$} 
\psline(2,0)(2.2,.6)\psline(2.2,0)(2.4,.6)\psline(2.4,0)(2.6,.6)\psline(2.6,0)(2.8,.6)
\psline(.4,0)(.8,-1.2)  \rput(.3,.15){$_{e_{1}}$}
\psline(.74,-1)(1.74,-1) \rput(.8,-.2){$_{e_{k+2}}$}
\psline(.5,-.33)(1.5,-.33) \rput(.8,-.5){$_{\dots}$}
\psline(.62,-.66)(1.62,-.66) \rput(1.05,-.85){$_{e_{k+l}}$}
\rput(1.2,-1.5){$(-1)^k\cdot e_1\wedge\dots\wedge e_{k+l}$}
 \end{pspicture}}
\]
Moving the lower branch over the thick vertex, and relabeling the edges, we obtain the induced standard orientation:
\[
\resizebox{!}{2.2cm}{  \begin{pspicture}(-3.4,-1.6)(2,.6)
 \psline[linewidth=2pt](-2,0)(2,0) \pscircle*(0,0){.12} \rput(-3,0){$\leadsto$}
 \rput(1.32,-.2){$_{e_3 \dots  e_{k+1}}$}  \rput(.45,-.2){$_{e_2}$} 
\psline(.8,0)(1,.6)\psline(1,0)(1.2,.6)\psline(1.2,0)(1.4,.6)\psline(1.4,0)(1.6,.6)
\psline(-1.4,0)(-1.8,-1.2)  \rput(-.6,.15){$_{e_{1}}$}
\psline(-1.74,-1)(-.74,-1) \rput(-1.1,-.2){$_{e_{k+2}}$}
\psline(-1.5,-.33)(-.5,-.33) \rput(-1.1,-.5){$_{\dots}$}
\psline(-1.62,-.66)(-.62,-.66) \rput(-1.32,-.85){$_{e_{k+l}}$}
\rput(0,-1.5){$(-1)^{k+1}\cdot e_1\wedge\dots\wedge e_{k+l}$}
 \end{pspicture}}
 \resizebox{!}{2.2cm}{  \begin{pspicture}(-4,-1.6)(2,.6)
 \psline[linewidth=2pt](-2,0)(2,0) \pscircle*(0,0){.12} \rput(-3,0){$=$}
 \rput(1.16,-.2){$_{e_2 \dots  e_{k}}$}  \rput(.45,-.2){$_{e_1}$} 
\psline(.8,0)(1,.6)\psline(1,0)(1.2,.6)\psline(1.2,0)(1.4,.6)\psline(1.4,0)(1.6,.6)
\psline(-1.4,0)(-1.8,-1.2)  \rput(-.6,.15){$_{e_{k+1}}$}
\psline(-1.74,-1)(-.74,-1) \rput(-1.1,-.2){$_{e_{k+2}}$}
\psline(-1.5,-.33)(-.5,-.33) \rput(-1.1,-.5){$_{\dots}$}
\psline(-1.62,-.66)(-.62,-.66) \rput(-1.32,-.85){$_{e_{k+l}}$}
\rput(0,-1.5){$(-1)^1\cdot e_1\wedge\dots\wedge e_{k+l}$}
 \end{pspicture}}
 \]
Finally, we move the edges $e_{k+2},\dots, e_{k+l}$ over to the thick left edge.
\[
\resizebox{!}{2.2cm}{  \begin{pspicture}(-4.4,-1.6)(1.6,.6)
 \psline[linewidth=2pt](-3,0)(1.6,0) \pscircle*(0,0){.12} \rput(-4,0){$\leadsto$}
 \rput(.65,-.2){$_{e_1 \dots \, e_{k}}$} 
\psline(.4,0)(.6,.6)\psline(.6,0)(.8,.6)\psline(.8,0)(1,.6)\psline(1,0)(1.2,.6)
\psline(-2.4,0)(-2.8,-1.2)  \rput(-.4,.15){$_{e_{k+1}}$}
\psline(-.8,0)(-1,-.6)  \rput(-1.2,.15){$_{e_{k+2}}$}
\psline(-2.74,-1)(-1.74,-1) \rput(-2.1,-.2){$_{e_{k+3}}$}
\psline(-2.5,-.33)(-1.5,-.33) \rput(-2.1,-.5){$_{\dots}$}
\psline(-2.62,-.66)(-1.62,-.66) \rput(-2.32,-.85){$_{e_{k+l}}$}
\rput(-0.2,-1.5){$(-1)^2\cdot e_1\wedge\dots\wedge e_{k+l}$}
 \end{pspicture}}
 \resizebox{!}{2.2cm}{  \begin{pspicture}(-4.4,-1.6)(1.6,.6)
 \psline[linewidth=2pt](-1.6,0)(1.6,0) \pscircle*(0,0){.12} 
 \rput(-3.1,0){$\leadsto\dots\leadsto$}
\psline(.4,0)(.6,.6)\psline(.6,0)(.8,.6)\psline(.8,0)(1,.6)\psline(1,0)(1.2,.6)
 \rput(.65,-.2){$_{e_1 \dots \, e_{k}}$} 
\psline(-.4,0)(-.6,-.6)\psline(-.6,0)(-.8,-.6)\psline(-.8,0)(-1,-.6)\psline(-1,0)(-1.2,-.6)
 \rput(-.65,.2){$_{e_{k+l} \dots  e_{k+1}}$} 
\rput(-0.2,-1.5){$(-1)^l\cdot e_1\wedge\dots\wedge e_{k+l}$}
 \end{pspicture}}
 \]
Thus, $\ost_{(I_{k,l})_{\max}}=(-1)^l\cdot e_1\wedge\dots\wedge e_{k+l}$ is the standard orientation with edges labeled as described above. This completes the proof.
\end{proof}


Our last task in this subsection is to define the orientation $\omega(S,D)$ that is needed in the definition of the morphism $p$ for a generator $(D,\fcan,m,\omega_D)\in \Qa{k}$ of degree $k$, and a diagram $S$ with $S_{\max}\leq D_{\min}$. This will be done in four steps: First, we define the orientation $\xi_B$ for binary diagrams $B$, second, we define the contraction $\omega\rfloor \xi_{B}$, third, we define notion of positive and negative edges and show how they are relevant for $p$, and fourth we use the notion of positive and negative edges to find the orientation $\omega(S,D)$ on $S$ by contracting $\omega\rfloor \xi_{S_{\max}}$ for some $\omega$.

{\bf Step 1.} 
Following \cite{MS}, we first define the orientation 
$$ \xi_B:=(-1)^{\frac{(n-2)(n-3)}{2}}\cdot \ost_B =(-1)^{1+2+\dots+(n-3)}\cdot \ost_B $$ 
for a binary diagram $B$ with standard orientation $\ost_B$, $n$ leaves, and $n-2$ edges.

\begin{remark}\label{REM:xi-by-composition}
There is an alternative description of $\xi_B$ given in \cite{MS}. For this, first define $\xi_B=+1$ for the binary diagrams $B$ in $\Ca{0}$ with $2$ leaves and no edges  (\emph{i.e.} $B\, =\A_2,M_{1,0}, M_{0,1},$ or $I_{0,0}$). Then for a general binary diagram $B$, the orientation $\xi_B$ is determined by the composition relation in $\Ca{0}$,
\[
(B'\circ_{\fcan(i)}B'',\fcan,\xi_{B'\circ_{\fcan(i)}B''})=\sigma\cdot( (B',\fcan,\xi_{B'})\circ_i(B'',\fcan,\xi_{B''}) ).
\]
We can derive this formula by comparing $\xi_B$ for two binary diagrams related by one of the local moves from Definition \ref{DEF:local-move}.
\end{remark}

{\bf Step 2.} 
If $B$ is a binary tree with orientation $\omega$, and $\omega'$ is an orientation on a subset of edges of $B$, then define $\omega'\rfloor \omega$ by the relation $<e',e>=\delta_{e',e}\in \{0,1\}$, where $\delta_{e',e}$ denotes the Kronecker delta. In particular, if $\omega=e_1\wedge\dots\wedge e_r\wedge e_{r+1}\wedge \dots\wedge e_k$ and $\omega'=e_1\wedge\dots\wedge e_r$, then
\[
\omega'\rfloor \omega =(-1)^{\frac{r(r-1)}{2}} \cdot e_{r+1}\wedge \dots\wedge e_k.
\]
Now, if $S_{\max}=D_{\min}$, then as we shall see in Step 3, $S=D_{\min}/\{$edges in $D\}$, in which case we define $\omega(S,D)=\omega_D\rfloor\xi_{D_{\min}}$. Now, consider a corolla $c$ for example. Since $1\rfloor \omega=\omega$ and $\ost_B\rfloor\xi_B=+1$, we obtain
\begin{eqnarray}
\label{EQ:omega(c)=xi}
S=c_{\min}, \,\, D=c ,\,\, \omega_D=1 & \Rightarrow & \omega(S,D)=1\rfloor \xi_{c_{\min}}=\xi_{c_{\min}}, \\
\label{EQ:omega(max)=1}
S=c,\,\, D=c_{\max}, \,\, \omega_D=\ost_{c_{\max}} & \Rightarrow & \omega(S,D)=\ost_{c_{\max}}\rfloor\xi_{c_{\max}}=+1.
\end{eqnarray}

{\bf Step 3.} 
To further analyze the condition $S_{\max}\leq D_{\min}$ we need to introduce the notion of positive and negative edges in a binary diagram. Markl and Shnider (\cite{MS}) refer to these notions as left-leaning and right-leaning, respectively.

\begin{definition}\label{DEF:pos-neg-edge}
Let $B$ be a binary diagram. We define an edge to be positive, denoted by $\Pos$, respectively negative, denoted by $\Neg$, if it appears in $B$ in the following way,
\[
\resizebox{!}{1.5cm}{
\begin{pspicture}(3,0)(6,2) \rput(3.6,.8){$\Neg$} 
 \psline(4,0)(4,.8) \psline(4,.8)(3.2,1.6) \psline(4,.8)(4.8,1.6) \psline(3.6,1.2)(4.4,2)
 \end{pspicture}
\begin{pspicture}(-.5,0)(3.5,2) \rput(1.4,.8){$\Pos$} 
 \psline(1,0)(1,.8) \psline(1,.8)(.2,1.6) \psline(1,.8)(1.8,1.6) \psline(1.4,1.2)(.6,2)
 \end{pspicture}
\begin{pspicture}(3,0)(6,2) \rput(4.3,1.2){$\Neg$} 
 \psline[linewidth=3pt](4,0)(4,2) \psline(4,.6)(3.1, 1.5) \psline(4,1.4)(4.5,1.9)
 \end{pspicture}
\begin{pspicture}(-.5,0)(2,2) \rput(.7,1.2){$\Pos$} 
  \psline[linewidth=3pt](1,0)(1,2)  \psline(1,.6)(1.9, 1.5) \psline(1,1.4)(.5,1.9)
 \end{pspicture}}
\]
\[
\resizebox{!}{1.5cm}{
\begin{pspicture}(3,0)(6,2) \rput(3.6,.6){$\Neg$} 
 \psline[linewidth=3pt](4,0)(4,2) \psline(4,.6)(3.1, 1.5) \psline(3.6,1)(3.4,1.8)
 \end{pspicture}
\begin{pspicture}(-.5,0)(3.5,2) \rput(1.3,1.1){$\Pos$} 
  \psline[linewidth=3pt](1,0)(1,2)  \psline(1,.6)(.1, 1.5) \psline(1,1.4)(.5,1.9)
 \end{pspicture}
\begin{pspicture}(3,0)(6,2) \rput(3.7,1.1){$\Neg$} 
 \psline[linewidth=3pt](4,0)(4,2) \psline(4,.6)(4.9, 1.5) \psline(4,1.4)(4.5,1.9)
 \end{pspicture}
\begin{pspicture}(-.5,0)(2,2) \rput(1.4,.6){$\Pos$} 
  \psline[linewidth=3pt](1,0)(1,2) \psline(1,.6)(1.9, 1.5) \psline(1.4,1)(1.6,1.9)
 \end{pspicture} }
\]
\[
\resizebox{!}{1.5cm}{
\begin{pspicture}(3,0)(6,2) \rput(3.7,.7){$\Neg$} 
 \pscircle*(4,1){.15} \psline[linewidth=3pt](3,1)(5,1) \psline(3.6,1)(3.2,1.6)
 \end{pspicture}
\begin{pspicture}(-.5,0)(3.5,2) \rput(1.3,.7){$\Pos$} 
 \pscircle*(1,1){.15} \psline[linewidth=3pt](0,1)(2,1) \psline(1.4,1)(1.8,1.6)
 \end{pspicture}
\begin{pspicture}(3,0)(6,2) \rput(4.3,1.3){$\Neg$} 
 \pscircle*(4,1){.15}\psline[linewidth=3pt](3,1)(5,1) \psline(4.4,1)(4.8,.4)
 \end{pspicture} 
\begin{pspicture}(-.5,0)(2,2) \rput(.7,1.3){$\Pos$} 
 \pscircle*(1,1){.15}  \psline[linewidth=3pt](0,1)(2,1)  \psline(.6,1)(.2,.4) 
 \end{pspicture} }
\]
We denote the number of positive edges in a binary diagram $B$ by $|B|_\Pos$. 
\end{definition}
\begin{lemma}\label{neq-k}
Positive and negative edges have the following properties:
\begin{enumerate}
\item If $B\leq B'$, then $|B|_\Pos \leq |B'|_\Pos$, \emph{i.e.}, $|.|_\Pos$ preserves the order. The maximal (resp. minimal) binary diagram $c_{\max}$ (resp. $c_{\min}$) of a corolla $c$ given by Lemma \ref{max-min} is the unique diagram all of whose edges are positive (resp. negative).
\item If $S$ is a diagram, then $S=S_{\max}/\{e_1,\dots, e_r\}$ is a quotient in which only positive edges of $S_{\max}$ are collapsed. If $D$ is a diagram, then $D=D_{\min}/\{e_1,\dots, e_s\}$ is a quotient in which only negative edges of $D_{\min}$ are collapsed.
\item
Let $(D,\fcan,m,\omega_D)\in \Qa{k}$. If $|D_{\min}|_\Pos\neq k$, then $p(D,\fcan,m,\omega_D)=0$. If $|D_{\min}|_\Pos= k$, then $D=D_{\min}/\{$negative edges$\}$, and if 
$(S,f,\omega)\in \Ca{k}$ is a summand of $p(D,\fcan,m,\omega_D)$, then $S_{\max}$ has exactly $k$ positive edges and $S=S_{\max}/\{$positive edges$\}$.
\end{enumerate}
\end{lemma}
\smallskip
\begin{proof} 
\quad
\begin{enumerate}
\item
This can be checked by direct inspection.
\item
For a diagram $S$, we obtain $S_{\max}$ by inserting positive edges at every non-binary vertex. For a diagram $D$, we obtain $D_{\min}$ by inserting negative edges at every non-binary vertex.
\item
Since $D$ has $k$ edges and $D_{\min}$ is obtained from $D$ by inserting only negative edges, $D_{\min}$ has at most $k$ positive edges. If $(S,f,\omega)\in \Ca{k}$, then $S_{\max}$ is obtained from $S$ by inserting $k$ positive edges; consequently $S_{max}$ has at least $k$ positive edges. Since $S_{\max}\leq D_{\min}$, we have $k\leq |S_{\max}|_\Pos\leq |D_{\min}|_\Pos\leq k$. Therefore $S_{\max}$ and $D_{\min}$ must have exactly $k$ positive edges. Furthermore, the $k$ positive edges in $D_{\min}$ are exactly the ones coming from $D$, and the $k$ positive edges in $S_{\max}$ are exactly the ones inserted in $S$.
\end{enumerate}
\end{proof}

{\bf Step 4.} 
Now, let $(D,\fcan,m,\omega_D)\in \Qa{k}$ be a generator of degree $k$, and let $S$ be a diagram with $S_{\max}\leq D_{\min}$. In order to define $\omega(S,D)$, we may assume by Lemma \ref{neq-k} (3), that exactly $k$ edges $e_1,\dots,e_k$ of $D_{\min}$ are positive, and these are the edges of $D$, \emph{i.e.}, $\omega_D=\eta\cdot e_1\wedge\dots\wedge e_k$ for some $\eta\in \{+1,-1\}$.

We can now define $\omega(S,D)$ in the general case.
\begin{definition}
If $S_{\max}=D_{\min}$, then $S=D_{\min}/\{e_1,\dots,e_k\}$, and we set $\omega(S,D):=\omega_D\rfloor \xi_{D_{\min}}$ as in step 2. If $S_{\max}<D_{\min}$, then $S=S_{\max}/\{$positive edges$\}$ by Lemma \ref{neq-k} (3), and we set $\omega(S,D):=(\eta\cdot e_1\wedge \dots\wedge e_k)\rfloor \xi_{S_{\max}}$, where the positive edges in $D_{\min}$ and $S_{\max}$ are identified using the local moves from Definition \ref{DEF:local-move}. The ambiguity of identifying positive edges under paths is given (according to the Coherence Lemma \ref{LEM:coherence}) by pentagons, hexagons and squares. Since the pentagons and hexagons change the number of positive edges, changing a path $\beta=(D_{\min},\dots,S_{\max})$ to another path $\beta'=(D_{\min},\dots,S_{\max})$ only consists of squares, for which the positive edges remain identified uniquely.
(Note, that for a local move $B \leadsto B'$ from Definition \ref{DEF:local-move} with constant number of positive edges, the procedure of keeping track of the positive edge by renaming a positive edge $e_\Pos$ by one negative edge $e_\Neg$ gives $\dots \wedge e_\Neg\wedge\dots \wedge e_\Pos\wedge\dots=-\dots \wedge e_\Pos\wedge\dots \wedge e_\Neg\wedge\dots$, which is the induced orientation $\xi_{B'}$ on $B'$.)
\end{definition}

\smallskip

\subsection{Sign check for Proposition \ref{q-bound}}\label{APP:(3.1)-signs}

We now give the remaining sign details for Proposition \ref{q-bound}. More precisely, in the notation of the proof of Proposition \ref{q-bound}, we will show that $(-1)^{i+1}\omega_B^{\hat {e_i}}$ equals $(-1)^{\epsilon_2}\omega_j$. 

We calculate $(-1)^{\epsilon_2}\omega_j$ of the binary diagram $B$ from the proof of Proposition \ref{q-bound} and compare it to $\omega_B^{\hat {e_i}}$.  Recall that $B$ is a composition of $B'$ and $B''$ with standard orientations corresponding to $c'$ and $c''$ with $(-1)^{\epsilon_1}\cdot \sigma\cdot\big((c',\fcan,1)\circ_j (c'',\fcan',1)\big)=(D',\fcan'',e')$ and $D'/e'=c$. If the corolla $c'$ has $r$ leaves and the corolla $c''$ has $s$ leaves, then the original corolla $c$ has $k=r+s-1$ leaves. Hence, from the definition of the composition in Equation \eqref{EQ:circ-i}, and of the $S_k$ action, we see that $(-1)^{\epsilon_1}=\sgn(\sigma)\cdot (-1)^{j\cdot (s+1)+r\cdot s}$. Furthermore, from the definition of the composition in $\QA$, we have $\omega_j=\ost_{B'}\wedge \ost_{B''}$. The only other signs come from comparing $\omega_j$ to $\ost_B$, and a possible application of $\sigma$. We consider two cases: Case 1: $\sigma=id_k$. Case 2: $\sigma\neq id_k$.
\begin{itemize}
\item[Case 1:]
Either the corolla $c$ is \emph{not} an inner product diagram, or $c$ is an inner product diagram and the composition $\circ_j$ is \emph{not} on the thick left module input.
\item[Case 2:]
The corolla $c$ is an inner product diagram and the composition $\circ_j$ is on the thick left module input.
\end{itemize}
In Case 1, the only way to obtain $(c,\fcan,1)$ as a composition of $c'$ and $c''$ is via the canonical labelings $\fcan$ for $c'$ and $c''$ and $\sigma=id$. In this case, the proof follows as in \cite[Proposition 4.2]{MS}. 
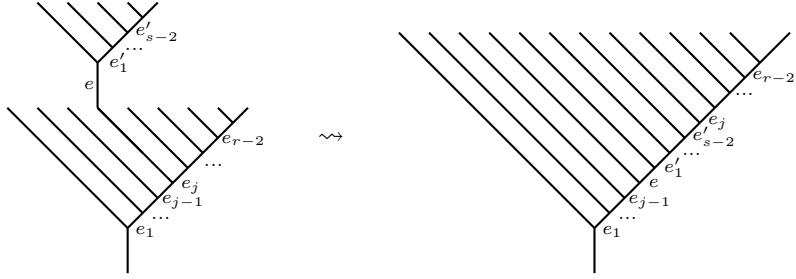
\begin{figure}[ht]
\[
 \begin{pspicture}(0.2,0.6)(3.8,4.4) 
 \psline(2,.8)(2,1.4) \psline(2,1.4)(.4,3) \psline(2,1.4)(3.6,3)
 \psline(2.2,1.6)(.8,3) \psline(2.4,1.8)(1.2,3) \psline(2.6,2)(1.6,3)
 \psline(2.8,2.2)(2,3) \psline(3,2.4)(2.4,3) \psline(3.2,2.6)(2.8,3) \psline(3.4,2.8)(3.2,3) 
\psline(2.2,2.4)(1.6,3) 
 \rput(2.25,1.35){$_{e_1}$} \rput(2.45,1.55){$_{\dots}$}  \rput(2.75,1.75){$_{e_{j-1}}$} 
 \rput(2.85,1.95){$_{e_j}$}  \rput(3.15,2.25){$_{\dots}$}  \rput(3.55,2.55){$_{e_{r-2}}$} 
 \psline(1.6,3)(1.6,3.6) \rput(1.5,3.3){$_e$}
 \psline(1.6,3.6)(.8,4.4) \psline(1.6,3.6)(2.4,4.4)
 \psline(1.8,3.8)(1.2,4.4) \psline(2,4)(1.6,4.4) \psline(2.2,4.2)(2,4.4)
 \rput(1.9,3.6){$_{e'_1}$} \rput(2.1,3.8){$_{\dots}$}  \rput(2.4,4){$_{e'_{s-2}}$} 
 \end{pspicture}
 \quad \quad \twig{-2} \quad \quad 
 \begin{pspicture}(-.6,0.6)(4.7,4) 
 \psline(2,.8)(2,1.4)  \psline(2,1.4)(4.6,4) 
 \psline(2,1.4)(-.6,4) \psline(2.2,1.6)(-.2,4) \psline(2.4,1.8)(.2,4) \psline(2.6,2)(.6,4)
 \psline(2.8,2.2)(1,4) \psline(3,2.4)(1.4,4) \psline(3.2,2.6)(1.8,4) \psline(3.4,2.8)(2.2,4) 
 \psline(3.6,3)(2.6,4) \psline(3.8,3.2)(3,4) \psline(4,3.4)(3.4,4) \psline(4.2,3.6)(3.8,4) 
 \rput(2.25,1.35){$_{e_1}$} \rput(2.45,1.55){$_{\dots}$}  \rput(2.75,1.75){$_{e_{j-1}}$} 
\rput(3.65,2.8){$_{e_{j}}$}  \rput(4,3.2){$_{\dots}$}  \rput(4.4,3.4){$_{e_{r-2}}$} 
  \rput(2.8,2){$_e$}
 \rput(3.1,2.2){$_{e'_1}\,$} \rput(3.3,2.4){$_{\dots}$}  \rput(3.6,2.6){$_{e'_{s-2}}$} 
 \end{pspicture}
\]
\caption{Move the edges $e, e'_1,\dots,e'_{s-2}$ to the right.}\label{FIG:q-signs-case1}
\end{figure}
More precisely, with the notation from the left diagram in Figure \ref{FIG:q-signs-case1}, we have
\begin{eqnarray*}
(-1)^{\epsilon_2} \omega_j &=& (-1)^{j(s+1)+rs}\cdot e_1\wedge\dots\wedge e_{r-2}\wedge e'_1\wedge \dots\wedge e'_{s-2}\\
&=& (-1)^{j+s}\cdot e_1\wedge\dots\wedge e_{j-1}\wedge e'_1\wedge \dots\wedge e'_{s-2}\wedge e_j\wedge\dots\wedge e_{r-2}.
\end{eqnarray*}
On the other hand, the standard orientation for the left diagram in Figure \ref{FIG:q-signs-case1} is $\ost_B=(-1)^{s-1}\cdot e_1\wedge\dots\wedge e_{j-1}\wedge e\wedge e'_1\wedge \dots\wedge e'_{s-2}\wedge e_j\wedge\dots\wedge e_{r-2}$, which can be seen via the $(s-1)$ local moves to the diagram on the right in Figure \ref{FIG:q-signs-case1}. Setting $i=j$, we obtain,
\[
(-1)^{i+1} \omega_B^{\hat{e}} = (-1)^{j+s}\cdot e_1\wedge\dots\wedge e_{j-1}\wedge e'_1\wedge \dots\wedge e'_{s-2}\wedge e_j\wedge\dots\wedge e_{r-2}=(-1)^{\epsilon_2} \omega_j.
\]
The considerations in other types of diagrams for Case 1 are similar to the one in Figure \ref{FIG:q-signs-case1}.

Now, for Case 2, where $c'=I_{r',r''}$ is an inner product corolla and $c''=M_{s',s''}$ is a module tree with $r=r'+r''+2$ and $s=s'+s''+1$, we have
\[
 (I_{r',r''},\fcan,1)\circ_1 (M_{s',s''},\fcan,1) = (-1)^{(s+1)+r\cdot s}\cdot (I_{r'+s'',r''+s'},\fcan\circ \tau^{s'},1),
\]
\[
\begin{pspicture}(.2,1)(2.8,3)  
 \psline(1.6,1.4)(2.4,2.6)  \psline(2.4,1.4)(1.6,2.6) \psline(2,1.4)(2,2.6)
 \psline(1,2)(.8,2.6) \psline(1,2)(.6,2.6) \psline(1,2)(.4,2.6) 
 \rput(2,2.9){$\overbrace{\quad\quad\quad}^{r'}$}  \rput(2,1.1){$\underbrace{\quad\quad\quad}_{r''}$}
 \rput(.6,2.9){$\overbrace{\quad}^{s''}$}  \rput(.6,1.1){$\underbrace{\quad}_{s'}$}
 \psline(1,2)(.8,1.4) \psline(1,2)(.6,1.4) \psline(1,2)(.4,1.4)
 \psline[linewidth=3pt](0.2,2)(2.8,2)
 \pscircle*(2,2){.15}
\end{pspicture} 
\]
where $\tau\in \Z_k\subset S_k$ is the cyclic rotation ``$-1\text{ (mod }k$)''. Thus, $\sigma=\tau^{s'}$ with $\sgn(\sigma)=(-1)^{s'\cdot(r+s'')}$, and thus $(-1)^{\epsilon_1}=(-1)^{s+1+rs+s'(r+s'')}$. Similarly we obtain (see Figure \ref{FIG:q-signs-case2})
\[
(-1)^{\epsilon_2} \omega_1 = (-1)^{s+1+rs+s'(r+s'')}\cdot e_1\wedge\dots\wedge e_{r-2}\wedge e'_1\wedge \dots\wedge e'_{s-2}.
\]
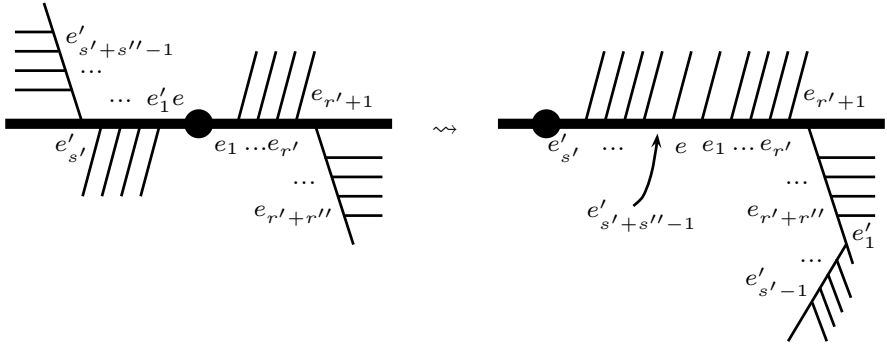
\begin{figure}[ht]
\[
\resizebox{!}{4.5cm}{
 \begin{pspicture}(1,-1)(5,2.5) 
 \psline[linewidth=3pt](1,1.25)(5,1.25) \pscircle*(3,1.25){.15}
 \psline(4,1.25)(4.2,2) \psline(3.4,1.25)(3.6,2) \psline(3.6,1.25)(3.8,2)
  \psline(3.8,1.25)(4,2) \psline(4.2,1.25)(4.6,0)  \psline(4.3,.9)(4.9,.9)
  \psline(4.37,.7)(4.9,.7)  \psline(4.45,.5)(4.9,.5)  \psline(4.5,.3)(4.9,.3)
  \rput(3.3,1){$_{e_1}$}  \rput(3.6,1){$_{\dots}$}  \rput(3.9,1){$_{e_{r'}}$}
 \rput(4.5,1.5){$_{e_{r'+1}}$} \rput(4.1,.65){$_{\dots}$} \rput(4,.3){$_{e_{r'+r''}}$}
 \psline(2,1.25)(1.8,.5) \psline(2.6,1.25)(2.4,.5) \psline(2.4,1.25)(2.2,.5)
  \psline(2.2,1.25)(2,.5) \psline(1.8,1.25)(1.4,2.5) \psline(1.7,1.6)(1.1,1.6)
  \psline(1.63,1.8)(1.1,1.8)\psline(1.55,2)(1.1,2)\psline(1.5,2.2)(1.1,2.2)
 \rput(2.6,1.5){$_{e'_1}$} \rput(2.2,1.5){$_{\dots}$} \rput(1.7,1){$_{e'_{s'}}$}
 \rput(1.9,1.8){$_{\dots}$} \rput(2.2,2.1){$_{e'_{s'+s''-1}}$}
 \rput(2.8,1.5){$_e$}
 \end{pspicture} }
 \quad \twig{-2.8} \quad
 \resizebox{!}{4.5cm}{
 \begin{pspicture}(1,-1)(5,2.5)  
 \psline[linewidth=3pt](1,1.25)(5,1.25) \pscircle*(1.5,1.25){.15}
 \psline(4,1.25)(4.2,2) \psline(3.4,1.25)(3.6,2) \psline(3.6,1.25)(3.8,2)
  \psline(3.8,1.25)(4,2) \psline(4.2,1.25)(4.67,-.2)  \psline(4.3,.9)(4.9,.9)
  \psline(4.37,.7)(4.9,.7)  \psline(4.45,.5)(4.9,.5)  \psline(4.5,.3)(4.9,.3)
  \rput(3.25,1){$_{e_1}$}  \rput(3.55,1){$_{\dots}$}  \rput(3.9,1){$_{e_{r'}}$}
 \rput(4.5,1.5){$_{e_{r'+1}}$} \rput(4.1,.65){$_{\dots}$} \rput(4,.3){$_{e_{r'+r''}}$}
\psline(3.1,1.25)(3.3,2) \psline(2.8,1.25)(3,2) \psline(2.5,1.25)(2.7,2)
\psline(2.3,1.25)(2.5,2)\psline(2.1,1.25)(2.3,2)\psline(1.9,1.25)(2.1,2)
 \rput(4.8,0.1){$_{e'_1}$} \rput(2.2,1){$_{\dots}$} \rput(1.7,1){$_{e'_{s'}}$}
 \rput(4.25,-.15){$_{\dots}$} \rput(2.5,.3){$_{e'_{s'+s''-1}}$} 
 \pscurve{->}(2.4,.4)(2.5,.5)(2.65,1.15)  \rput(2.9,1){$_e$} \psline(4.6,0)(4,-1)
 \psline(4.25,-.6)(4.4,-1) \psline(4.33,-.45)(4.48,-.85) \psline(4.41,-.3)(4.56,-.7)
  \psline(4.5,-.15)(4.65,-.55) \rput(3.9,-.4){$_{e'_{s'-1}}$}
 \end{pspicture} }
\]
\caption{The edges on the left are brought to the right in $(s+r-3)$ moves: $s'-1$ to move $e'_1,\dots, e'_{s'}$ on one branch; $2$ to move the obtained two main branches from left to right; $r'+r''$ to move the ``$s'$''-branch all the way to the right; and $s''-1$ to move the edges from the $``s''$''-branch to the thick right edge.}\label{FIG:q-signs-case2}
\end{figure}
On the other hand, the standard orientation of the left diagram in Figure \ref{FIG:q-signs-case2} is 
\begin{multline*}
 \ost_B = (-1)^{s-1+r}\cdot e'_{s'}\wedge\dots\wedge e_{s'+s''-1}\wedge e \wedge e_1\wedge \dots \wedge e_{r-2}\wedge e'_1\wedge\dots\wedge e'_{s'-1}\\
= (-1)^{s-1+r+s''(r+s')}\cdot e\wedge e_1\wedge\dots\wedge e_{r-2}\wedge e'_1\wedge \dots\wedge e'_{s'-1}\wedge e'_{s'}\wedge \dots\wedge e'_{s-2},
 \end{multline*}
which can be seen by performing $(s+r-3)$ local moves yielding the right diagram in Figure \ref{FIG:q-signs-case2}. Thus for $i=1$, we obtain
\[
(-1)^{i+1} \cdot \omega_B^{\hat e}=  (-1)^{s+r-3+s''(r+s')}\cdot e_1\wedge\dots\wedge e_{r-2}\wedge e'_1\wedge \dots\wedge e'_{s-2} = (-1)^{\epsilon_2}\cdot \omega_1,
\]
where we used the fact that $s=s'+s''+1$, so that $(-1)^{r+s''r}=(-1)^{rs+s'r}$.

This completes the check of both cases, and with this also the proof of Proposition \ref{q-bound}.

\section{Proof of Proposition \ref{contract}}\label{I_k,l k,l>0}

In this appendix, we prove Proposition \ref{contract}: \textit{The cellular
complexes associated with }$T_{n}$, $M_{k,l}$, and $I_{k,l}$ \textit{are
contractible.} The cellular complex associated with $T_{n}$ is the
associahedron $K_{n}$, whose contractiblity was proved by J. Stasheff in
\cite[Proposition 3]{S}. Since the cellular complexes associated with
$M_{k,n-k},I_{n,0}$ and $I_{0,n}$ are also isomorphic to $K_{n}$, our task is to
verify the contractibility of the cellular complex $\left\vert I_{k,l}%
\right\vert $ associated with $I_{k,l}$ for each $k,l\geq1$. To this end, we
derive an analog of Stasheff's result, which shows that $\left\vert I_{k,l}%
\right\vert $ is homeomorphic to the (closed) $(k+l)$-ball $\bar{B}^{k+l}$.

We begin with an outline of the ideas involved. Fix natural numbers $k,l\geq
1$, and let $n=k+l$. First, we identify the boundary of $\left\vert
I_{k,l}\right\vert $ with the union of two closed $(n-1)$-balls glued together
along their bounding $(n-2)$-spheres. Roughly speaking, these
\textquotedblleft upper\textquotedblright\ and \textquotedblleft
lower\textquotedblright\ boundary components are obtained by inserting edges
into the upper and the lower parts of $I_{k,l}$, respectively, \emph{cf.}
Definition \ref{edge-number-set}. Second, we identify this
common $\left(  n-2\right)  $-sphere with those diagrams obtained
from $I_{k,l}$ by inserting at least one upper and one lower edge. Third, recalling that
Stasheff represented the associahedron $K_{n+2}$ as a subdivision of a cube
\cite[Section 6]{S}, we establish bijections between the upper and lower
boundary components of $\left\vert I_{k,l}\right\vert $ and the interior of the union of the boundary components $\partial_2,\dots,\partial_{k+2}$ of $K_{n+2}$ (see Figure
\ref{I-and-K} on page \pageref{I-and-K}), which we denote by 
$\left\vert K_{n+2}^{\{2,\dots,k+2\}}\right\vert$. 

\begin{definition}
\label{edge-number} Label the leaves of $I_{k,l}$, $K_{n+2}$, and diagrams $D
$ in $\partial I_{k,l}$ or $\partial K_{n+2}$, using the canonical labeling
$f_{\circlearrowright}$. Given such a diagram $D$ and an edge $e$ of $D$, let
$t $ be the smallest positive integer that labels a leaf outward from $e$, and
assign the label $t$ to $e$. If $D$ is in $\partial I_{k,l}$ and $t
\in\{2,\dots, k+2\}$, we say that $e$ is an \emph{upper edge} of $D$ ;
otherwise $e$ is a \emph{lower edge} of $D$. Likewise, a leaf of $I_{k,l}$ or
$D$ with label in $\{2,\dots, k+2\}$ is an \emph{upper leaf}; otherwise it is
a \emph{lower leaf}.
\end{definition}

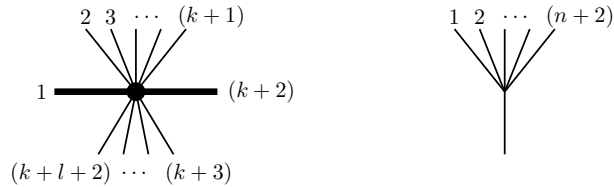
\begin{figure}[h]%
\[
\resizebox{!}{2.5cm}{ \begin{pspicture}(-.5,.5)(5,3.5) \psline(2,2)(1.6,3) \psline(2,2)(1.2,3) \psline(2,2)(2.8,3) \psline(2,2)(2,3) \psline(2,2)(2.4,3) \psline(2,2)(1.4,1) \psline(2,2)(1.8,1) \psline(2,2)(2.2,1) \psline(2,2)(2.6,1) \psline[linewidth=3pt](.7,2)(3.3,2) \rput(.5,2){$1$} \rput(1.2,3.2){$2$} \rput(1.6,3.2){$3$} \rput(2.2,3.2){$\dots$} \rput(3.2,3.2){$(k+1)$} \rput(4,2){$(k+2)$} \rput(3,.7){$(k+3)$} \rput(2,.7){$\dots$} \rput(.8,.7){$(k+l+2)$} \pscircle*(2,2){.15} \end{pspicture}}
\quad\quad\quad
\resizebox{!}{2.5cm}{ \begin{pspicture}(.5,.5)(4,3.5) \psline(2,2)(1.2,3) \psline(2,2)(1.6,3) \psline(2,1)(2,3) \psline(2,2)(2.4,3) \psline(2,2)(2.8,3) \rput(1.2,3.2){$1$} \rput(1.6,3.2){$2$} \rput(2.2,3.2){$\dots$} \rput(3.2,3.2){$(n+2)$} \end{pspicture}}
\]
\caption{The canonical labeling of leaves}%
\label{FIG:leaf-order}%
\end{figure}

The leaves and edges of the diagrams in Figure \ref{EQ:numbering}, for
example, are labeled as specified by Definition \ref{edge-number}; the upper
edges and leaves of the left-hand diagram are labeled by elements of
$\{2,\dots, 17\}$. 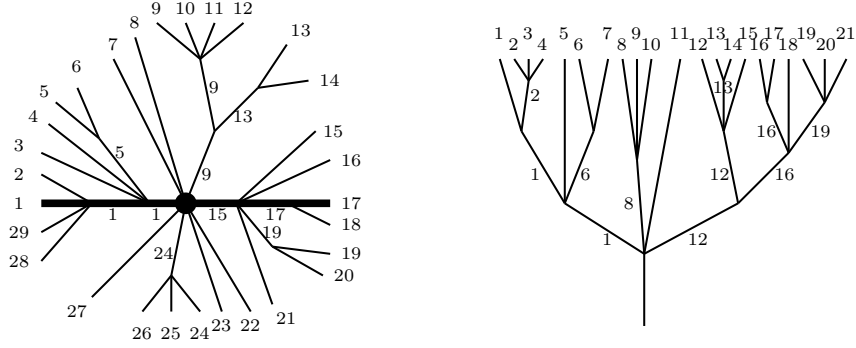
\begin{figure}[th]%
\[
\resizebox{!}{4.8cm}{ \begin{pspicture}(-1,0)(5,5) \psline[linewidth=3pt](0,2)(4,2) \pscircle*(2,2){.15} \psline(2,2)(1,4) \psline(2,2)(1.3,4.3) \psline(2,2)(2.4,3) \psline(2,2)(.7,.7) \psline(2,2)(1.8,1) \psline(2,2)(2.5,.5) \psline(2,2)(2.9,.5) \psline(2.7,2)(3.2,.6) \psline(2.7,2)(3.2,1.4) \psline(3.4,2)(4,1.7) \psline(2.7,2)(4,2.6) \psline(2.7,2)(3.8,3) \psline(3.2,1.4)(4,1.3) \psline(3.2,1.4)(3.9,1) \psline(1.8,1)(1.8,.5) \psline(1.8,1)(1.4,.5) \psline(1.8,1)(2.2,.5) \psline(1.5,2)(0,2.7) \psline(1.5,2)(.1,3.1) \psline(1.5,2)(0.8,2.9) \psline(0.7,2)(0,1.6) \psline(0.7,2)(0,1.2) \psline(0.7,2)(0,2.4) \psline(0.8,2.9)(0.2,3.4) \psline(0.8,2.9)(0.5,3.6) \psline(3,3.6)(3.4,4.2) \psline(3,3.6)(3.7,3.7) \psline(2.4,3)(2.2,4) \psline(2.4,3)(3,3.6) \psline(2.2,4)(2.0,4.5) \psline(2.2,4)(2.4,4.5) \psline(2.2,4)(2.8,4.5) \psline(2.2,4)(1.6,4.5) \rput(-.3,2){$_1$} \rput(-.3,2.4){$_2$} \rput(-.3,2.8){$_3$} \rput(-.1,3.2){$_4$} \rput(.05,3.55){$_5$} \rput(.5,3.9){$_6$} \rput(1,4.2){$_7$} \rput(1.3,4.5){$_8$} \rput(1.6,4.7){$_9$} \rput(2,4.7){$_{10}$} \rput(2.4,4.7){$_{11}$} \rput(2.8,4.7){$_{12}$} \rput(3.6,4.4){$_{13}$} \rput(4,3.7){$_{14}$} \rput(4.05,3){$_{15}$} \rput(4.3,2.6){$_{16}$} \rput(4.3,2){$_{17}$} \rput(4.3,1.7){$_{18}$} \rput(4.3,1.3){$_{19}$} \rput(4.2,1){$_{20}$} \rput(3.4,.4){$_{21}$} \rput(2.9,.3){$_{22}$} \rput(2.5,.3){$_{23}$} \rput(2.2,.2){$_{24}$} \rput(1.8,.2){$_{25}$} \rput(1.4,.2){$_{26}$} \rput(.5,.5){$_{27}$} \rput(-.3,1.2){$_{28}$} \rput(-.3,1.6){$_{29}$} \rput(1,1.85){$_1$} \rput(1.6,1.85){$_1$} \rput(1.1,2.7){$_5$} \rput(2.3,2.4){$_9$} \rput(2.4,3.6){$_9$} \rput(2.8,3.2){$_{13}$} \rput(2.45,1.85){$_{15}$} \rput(3.25,1.85){$_{17}$} \rput(3.2,1.6){$_{19}$} \rput(1.7,1.3){$_{24}$} \end{pspicture}}
\quad\quad
\resizebox{!}{4.8cm}{ \begin{pspicture}(0,0)(5.5,5) \psline(2.5,.3)(2.5,1.3) \psline(2.5,1.3)(1.4,2) \psline(2.5,1.3)(2.4,2.6) \psline(2.5,1.3)(3,4) \psline(2.5,1.3)(3.8,2) \psline(1.4,2)(.8,3) \psline(1.4,2)(1.4,4) \psline(1.4,2)(1.8,3) \psline(.8,3)(.5,4) \psline(.8,3)(.9,3.7) \psline(.9,3.7)(.9,4) \psline(.9,3.7)(1.1,4) \psline(.9,3.7)(.7,4) \psline(1.8, 3)(1.6,4) \psline(1.8, 3)(2,4) \psline(2.4,2.6)(2.2,4) \psline(2.4,2.6)(2.6,4) \psline(2.4,2.6)(2.4,4) \psline(3.8,2)(3.6,3) \psline(3.8,2)(4.5,2.7) \psline(3.6,3)(3.3,4) \psline(3.6,3.7)(3.5,4) \psline(3.6,3.7)(3.7,4) \psline(3.6,3)(3.6,3.7) \psline(3.6,3)(3.9,4) \psline(4.5,2.7)(4.2,3.4) \psline(4.5,2.7)(5,3.4) \psline(4.2,3.4)(4.1,4) \psline(4.2,3.4)(4.3,4) \psline(4.5,2.7)(4.5,4) \psline(5,3.4)(5,4) \psline(5,3.4)(4.7,4) \psline(5,3.4)(5.3,4) \rput(.5,4.35){$_1$} \rput(.7,4.2){$_2$} \rput(.9,4.35){$_3$} \rput(1.1,4.2){$_4$} \rput(1.4,4.35){$_5$} \rput(1.6,4.2){$_6$} \rput(2,4.35){$_7$} \rput(2.2,4.2){$_8$} \rput(2.4,4.35){$_9$} \rput(2.6,4.2){$_{10}$} \rput(3,4.35){$_{11}$} \rput(3.25,4.2){$_{12}$} \rput(3.5,4.35){$_{13}$} \rput(3.73,4.2){$_{14}$} \rput(3.95,4.35){$_{15}$} \rput(4.1,4.2){$_{16}$} \rput(4.3,4.35){$_{17}$} \rput(4.5,4.2){$_{18}$} \rput(4.75,4.35){$_{19}$} \rput(5.025,4.2){$_{20}$} \rput(5.3,4.35){$_{21}$} \rput(1,2.4){$_1$} \rput(2,1.5){$_1$} \rput(1.7,2.4){$_6$} \rput(2.3,2){$_8$} \rput(3.25,1.5){$_{12}$} \rput(3.55,2.4){$_{12}$} \rput(3.6,3.6){$_{13}$} \rput(4.45,2.4){$_{16}$} \rput(4.2,3){$_{16}$} \rput(4.95,3){$_{19}$} \rput(1,3.5){$_2$} \end{pspicture} }
\]
\caption{The labeling of the edges}%
\label{EQ:numbering}%
\end{figure}

Since the label of an edge $e$ in a labeled diagram $D$ is determined by the
labels of leaves outward from $e$, inserting an edge into $D$ or contracting
an edge of $D\smallsetminus\{e\}$ preserves the label of $e$. In particular,
the boundary operator $\partial$, which is defined by summing over all possible ways of inserting an edge, preserves labels.

\begin{definition}
\label{edge-number-set} Let $X\subset\{1,\dots,n+2\}$ and let $I_{k,l}^{X}$
(respectively $K_{n+2}^{X}$) denote the module generated by all diagrams in
$\partial I_{k,l}$ (respectively $\partial K_{n+2}$) whose edges are labeled
by elements of $X$. For example, the diagrams in Figure \ref{EQ:numbering} lie
in $I_{15,12}^{\{1,5,9,13,15,17,19,24\}}$ and $K_{21}%
^{\{1,2,6,8,12,13,16,19\}}$, respectively. Note that many edges may have the
same label, and $I_{k,l}^{X}=0$ for many subsets $X$. In particular,
$I_{k,l}^{+}:=I_{k,l}^{\{2,\dots,k+2\}}$ is the module generated by all
diagrams with \emph{at least one upper edge but no lower edges}, and
$I_{k,l}^{-}:=I_{k,l}^{\{1,k+3,\dots,k+l+2\}}$ is the module generated by all
diagrams with \emph{at least one lower edge but no upper edges}. Generators of
$I_{k,l}^{+}$ are called \emph{upper diagrams}; generators of $I_{k,l}^{-}$
are called \emph{lower diagrams}. Of course, $I_{k,l}^{+}\cap I_{k,l}%
^{-}=\varnothing$. Furthermore, denote the module generated by diagrams with
\emph{at least one upper edge} by $\overline{I_{k,l}^{+}}$; denote the module
generated by diagrams with \emph{at least one lower edge} by $\overline
{I_{k,l}^{-}}$. Then $I_{k,l}^{+}\subset\overline{I_{k,l}^{+}}$ and
$I_{k,l}^{-}\subset\overline{I_{k,l}^{-}},$ and in fact, the geometric
realzations $\left\vert \overline{I_{k,l}^{+}}\right\vert $ and $\left\vert
\overline{I_{k,l}^{-}}\right\vert $ are the respective closures of
$\left\vert I_{k,l}^{+}\right\vert $ and $\left\vert I_{k,l}^{-}\right\vert $,
\emph{i.e.,} $\overline{\left\vert I_{k,l}^{+}\right\vert }=\left\vert
\overline{I_{k,l}^{+}}\right\vert $ and $\overline{\left\vert I_{k,l}%
^{-}\right\vert }=\left\vert \overline{I_{k,l}^{-}}\right\vert .$ Finally,
$C_{\ast}(I_{k,l}):=\langle I_{k,l}\rangle\oplus\overline{I_{k,l}^{+}}%
\oplus\overline{I_{k,l}^{-}}$ is the module generated by $I_{k,l}$ and all
diagrams obtained from $I_{k,l}$ by inserting edges.
\end{definition}

There is the following important bijection $I^{+}_{k,l}\leftrightarrow
K^{\{2,\dots,k+2\}}_{k+l+2}$: Given an upper diagram $D$, produce the
corresponding tree $T\in K^{\{2,\dots,k+2\}}_{k+l+2}$ by changing thick colors
to thin and attaching a root between the first and last leaf (see Figure
\ref{FIG:attach-detach}). Recover $D$ from $T$ by detaching the root and
changing the color of the branches containing the first and $(k+2)^{nd}$
leaves from thin to thick. The tools we need to prove Proposition
\ref{contract} are now in place.

\begin{figure}[h]%
\[
\begin{pspicture}(-2,.5)(7.5,3.5) 
\pscircle*(2,2){.15} \psline(1.7,2.5)(1.6,3) \psline(1.7,2.5)(1.2,3) \psline(2,2)(1.7,2.5) \psline(2,2)(2.3,2.5) \psline(2.3,2.5)(2.8,3) \psline(2.3,2.5)(2,3) \psline(2.3,2.5)(2.4,3) \psline(2,2)(1.5,1) \psline(2,2)(2.5,1) \psline(2.5,2)(3.1,2.3) \psline(2.5,2)(3.1,1.7) \psline[linewidth=3pt](.7,2)(3.3,2) \psline[linestyle=dashed](0,1)(1,1.4) \rput(-.85,1.45){(attach root)} \rput(4.2,2){$\leftrightarrow$} \psline[linestyle=dashed](6,1)(6,1.8) \psline[linewidth=2pt](6,1.8)(5.0,3) \psline(6,1.8)(5.5,2.6) \psline(5.5,2.6)(5.2,3) \psline(5.5,2.6)(5.4,3) \psline(6,1.8)(5.9,2.6) \psline(5.9,2.6)(5.6,3) \psline(5.9,2.6)(5.8,3) \psline(5.9,2.6)(6.0,3) \psline[linewidth=2pt](6,1.8)(6.3,2.6) \psline(6.3,2.6)(6.2,3) \psline[linewidth=2pt](6.3,2.6)(6.4,3) \psline(6.3,2.6)(6.6,3) \psline(6,1.8)(6.8,3) \psline(6,1.8)(7.0,3) \end{pspicture} \begin{pspicture}(4.3,.5)(7.5,3.5) \rput(4.2,2){$\leftrightarrow$} \psline(6,1)(6,1.8) \psline(6,1.8)(5.0,3) \psline(6,1.8)(5.5,2.6) \psline(5.5,2.6)(5.2,3) \psline(5.5,2.6)(5.4,3) \psline(6,1.8)(5.9,2.6) \psline(5.9,2.6)(5.6,3) \psline(5.9,2.6)(5.8,3) \psline(5.9,2.6)(6.0,3) \psline(6,1.8)(6.3,2.6) \psline(6.3,2.6)(6.2,3) \psline(6.3,2.6)(6.4,3) \psline(6.3,2.6)(6.6,3) \psline(6,1.8)(6.8,3) \psline(6,1.8)(7.0,3) \end{pspicture}
\]
\caption{The correspondence $I^{+}_{k,l} \leftrightarrow K^{\{2,\dots
,k+2\}}_{k+l+2}$}%
\label{FIG:attach-detach}%
\end{figure}
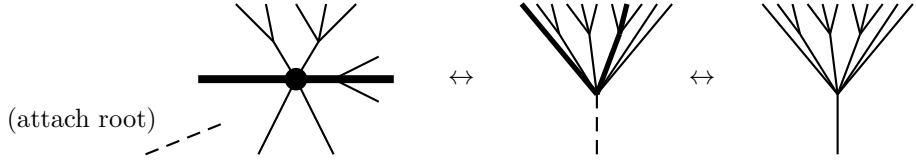

\begin{proof}
[Proof of Proposition \ref{contract}]We wish to realize $C_{\ast}\left(
I_{k,l}\right)$ as a cellular complex $|I_{k,l}|\cong \bar{B}^{k+l}$, for each $k,l\geq 1$. First note that $|I_{1,1}| \cong \bar{B}^{2}$ (see Figure \ref{FIG:the-I-(1-1)-hexagon}). 
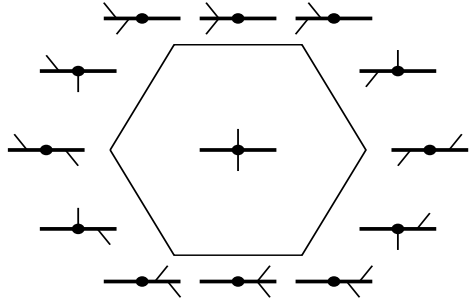
\begin{figure}[h]
 \scalebox{.85}[.7] {
\begin{pspicture}(1,0)(9,6) 
\psline(3,3)(4,1)(6,1)(7,3)(6,5)(4,5)(3,3)
 \pscircle*(5,3){.1}\psline[linewidth=2pt](4.4,3)(5.6,3)
 \psline(5,3)(5,3.4)\psline(5,3)(5,2.6)
\pscircle*(2,3){.1}\psline[linewidth=2pt](1.4,3)(2.6,3)
\psline(2.3,3)(2.5,2.7)\psline(1.7,3)(1.5,3.3)
\pscircle*(8,3){.1}\psline[linewidth=2pt](7.4,3)(8.6,3)
\psline(8.3,3)(8.5,3.3)\psline(7.7,3)(7.5,2.7)
\pscircle*(3.5,5.5){.1}\psline[linewidth=2pt](2.9,5.5)(4.1,5.5)
\psline(3.3,5.5)(3.1,5.2)\psline(3.1,5.5)(2.9,5.8)
\pscircle*(3.5,0.5){.1}\psline[linewidth=2pt](2.9,0.5)(4.1,0.5)
\psline(3.7,0.5)(3.9,0.8)\psline(3.9,0.5)(4.1,0.2)
\pscircle*(6.5,5.5){.1}\psline[linewidth=2pt](5.9,5.5)(7.1,5.5)
\psline(6.3,5.5)(6.1,5.8)\psline(6.1,5.5)(5.9,5.2)
\pscircle*(6.5,0.5){.1}\psline[linewidth=2pt](5.9,0.5)(7.1,0.5)
\psline(6.7,0.5)(6.9,0.2)\psline(6.9,0.5)(7.1,0.8)
\pscircle*(5,5.5){.1}\psline[linewidth=2pt](4.4,5.5)(5.6,5.5)
\psline(4.7,5.5)(4.5,5.8)\psline(4.7,5.5)(4.5,5.2)
\pscircle*(5,0.5){.1}\psline[linewidth=2pt](4.4,0.5)(5.6,0.5)
\psline(5.3,0.5)(5.5,0.8)\psline(5.3,0.5)(5.5,0.2)
\pscircle*(2.5,4.5){.1}\psline[linewidth=2pt](1.9,4.5)(3.1,4.5)
\psline(2.5,4.5)(2.5,4.1)\psline(2.2,4.5)(2,4.8)
\pscircle*(2.5,1.5){.1}\psline[linewidth=2pt](1.9,1.5)(3.1,1.5)
\psline(2.8,1.5)(3,1.2)\psline(2.5,1.5)(2.5,1.9)
\pscircle*(7.5,4.5){.1}\psline[linewidth=2pt](6.9,4.5)(8.1,4.5)
\psline(7.2,4.5)(7,4.2)\psline(7.5,4.5)(7.5,4.9)
\pscircle*(7.5,1.5){.1}\psline[linewidth=2pt](6.9,1.5)(8.1,1.5)
\psline(7.5,1.5)(7.5,1.1)\psline(7.8,1.5)(8,1.8)
\end{pspicture} }
\caption{The hexagon induced by $I_{1,1}$}\label{FIG:the-I-(1-1)-hexagon}
\end{figure}
Inductively, assume that $|I_{k^{\prime},l^{\prime}}|\cong\bar{B}^{k^{\prime}+l^{\prime}}$ for all
$k^{\prime}+l^{\prime}<k+l$, and note that $\partial C_{\ast}\left(
I_{k,l}\right)  =I_{k,l}^{+}\oplus I_{k,l}^{-}\oplus\left(  \overline
{I_{k,l}^{+}}\cap\overline{I_{k,l}^{-}}\right)  ,$ \emph{i.e., }all upper
diagrams, lower diagrams, and diagrams with both upper and lower edges. The
proof follows in three steps:

\begin{itemize}
\item[Step 1:] $\left\vert I_{k,l}^{+}\right\vert $ and $\left\vert
I_{k,l}^{-}\right\vert $ are homeomorphic to the open ball $B^{k+l-1}$.

\item[Step 2:] $\left\vert \overline{I_{k,l}^{+}}\right\vert $ and$\left\vert
\overline{I_{k,l}^{-}}\right\vert $ are homeomorphic to the closed ball
$\bar{B}^{k+l-1}$.

\item[Step 3:] $\left\vert \overline{I_{k,l}^{+}}\cap\overline{I_{k,l}^{-}%
}\right\vert \cong S^{k+l-2}$.
\end{itemize}

Steps 2 and 3 imply that $\left\vert \partial C_{\ast}(I_{k,l})\right\vert $
is homeomorphic to the union of two closed $(k+l-1)$-balls glued together along their
bounding $(k+l-2)$-spheres. Thus $|I_{k,l}|\cong \bar{B}^{k+l}$ and the proof
is complete.
\end{proof}

\begin{proof}
[Proof of step 1]In \cite[Section 6]{S}, Stasheff gave
an explicit realization of $K_{n+2}^{\{2,\dots,k+2\}}$ homeomorphic to an
(open) ball $B^{n-1}$ in $\partial I^{n}$. Since the correspondence $I_{k,l}^{+}\leftrightarrow
K_{n+2}^{\{2,\dots,k+2\}}$ preserves boundary, it defines a homeomorphism of
geometric realizations. Thus $\left\vert I_{k,l}^{+}\right\vert \cong\left\vert
K_{n+2}^{\{2,\dots,k+2\}}\right\vert \cong \bar{B}^{n-1}$. The identification $I_{k,l}%
^{-}\leftrightarrow K_{n+2}^{\{2,\dots,l+2\}}$ is established by relabeling
leaves: Assign \textquotedblleft$1$\textquotedblright\ to the right-most leaf
and continue clockwise, thereby replacing the original labels $k+3,\dots
,k+l+2,1$ with $2,\dots,l+2$, respectively; thus $\left\vert I_{k,l}%
^{-}\right\vert \cong\left\vert K_{n+2}^{\{2,\dots,k+2\}}\right\vert $ and the
conclusion follows.
\end{proof}

\begin{proof}
[Proof of step 2]Lemma \ref{boundary-expands} below asserts that the
homeomorphisms of realizations $\left\vert I_{k,l}^{+}\right\vert
\cong\left\vert I_{k,l}^{-}\right\vert \cong\left\vert K_{n+2}^{\{2,\dots
,k+2\}}\right\vert \cong B^{n-1}$ extend to homeomorphisms of closures. Thus
$\left\vert \overline{I_{k,l}^{+}}\right\vert \cong\left\vert
\overline{I_{k,l}^{-}}\right\vert \cong\overline{\left\vert K_{n+2}%
^{\{2,\dots,k+1\}}\right\vert }\cong\bar{B}^{n-1}$.
\end{proof}

\begin{proof}
[Proof of step 3]Note that the sets $\partial\overline{I_{k,l}^{+}}=\overline
{I_{k,l}^{+}}\smallsetminus I_{k,l}^{+}$ and $\partial
\overline{I_{k,l}^{-}}=\overline{I_{k,l}^{-}}\smallsetminus I_{k,l}^{-}$ are identical and consist of all diagrams with at
least one upper and one lower edge. Thus $\left\vert \overline{I_{k,l}^{+}%
}\cap\overline{I_{k,l}^{-}}\right\vert =\left\vert \partial\overline
{I_{k,l}^{+}}\right\vert =\partial\left\vert \overline{I_{k,l}^{+}}\right\vert
\cong S^{k+l-2}$.
\end{proof}

\begin{lemma}
\label{boundary-expands} The homeomorphism of realizations $\left\vert
I_{k,l}^{+}\right\vert \cong\left\vert K_{n+2}^{\{2,\dots,k+2\}}\right\vert $
defined above extends to a homeomorphism of closures $\left\vert
\overline{I_{k,l}^{+}}\right\vert \cong\overline{\left\vert K_{n+2}%
^{\{2,\dots,k+1\}}\right\vert }$.
\end{lemma}
\begin{proof}
In the proof of \cite[Proposition 3]{S}, Stasheff showed how to construct a cellular homeomorphism $\left\vert
K_{n+2}^{\{2,\dots,k+2\}}\right\vert \cong B^{n-1}$ and extend it to closures $\overline{\left\vert
K_{n+2}^{\{2,\dots,k+2\}}\right\vert}\cong \bar{B}^{n-1}$. We wish to realize $\left\vert \overline
{I_{k,l}^{+}}\right \vert$ in a similar way (see Figure \ref{I-and-K}). 
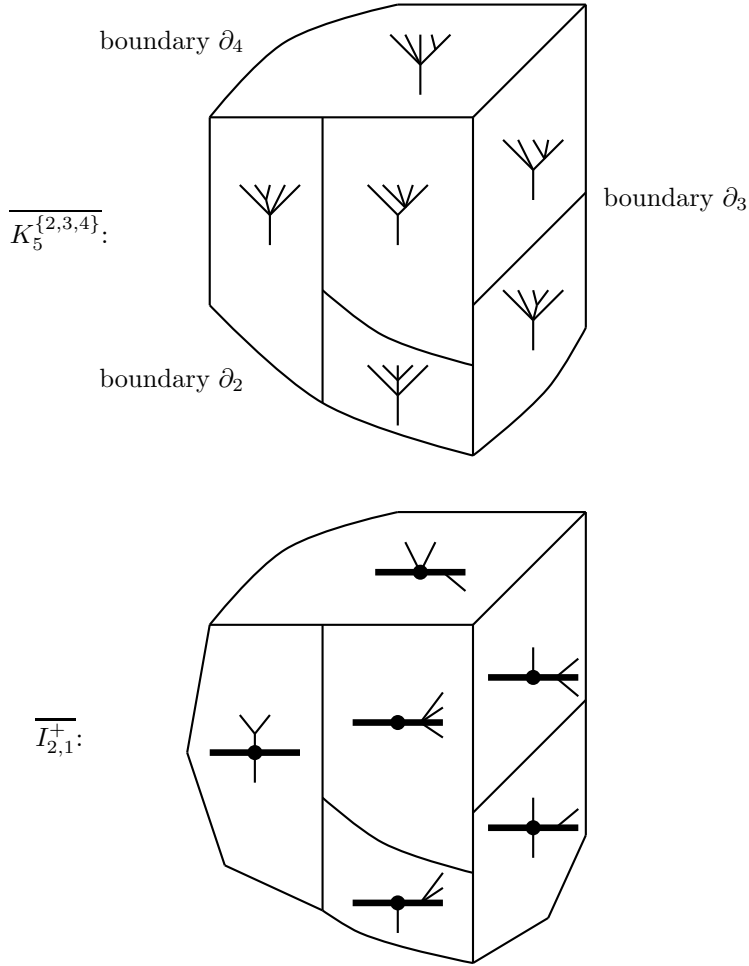
\begin{figure}[h]%
\[
\begin{pspicture}(-1,0)(9.5,6) \rput(0,3){$\overline{K^{\{2,3,4\}}_{5}}$:} \psline(4.5,6)(7,6) \pscurve(2,4.5)(3,5.5)(4.5,6) \psline(2,4.5)(5.5,4.5) \psline(5.5,4.5)(7,6) \psline(7,6)(7,1.7) \psline(2,4.5)(2,2) \psline(3.5,4.5)(3.5,.7) \psline(5.5,4.5)(5.5,0) \psline(5.5,2)(7,3.5) \pscurve(2,2)(3.5,.7)(5.5,0) \pscurve(5.5,0)(6.5,.9)(7,1.7) \pscurve(3.5,2.2)(4.3,1.6)(5.5,1.2) \rput(1.5,5.5){boundary $\partial_4$} \rput(8.2,3.4){boundary $\partial_3$} \rput(1.5,1){boundary $\partial_2$} \psline(4.8,5.2)(4.8,4.8) \psline(4.8,5.2)(4.8,5.6)\psline(4.8,5.2)(4.6,5.6) \psline(4.8,5.2)(4.4,5.6)\psline(5,5.4)(4.95,5.6)\psline(4.8,5.2)(5.2,5.6) \psline(2.8,3.2)(2.8,2.8) \psline(2.8,3.2)(2.4,3.6) \psline(2.75,3.4)(2.6,3.6) \psline(2.75,3.4)(2.8,3.6) \psline(2.8,3.2)(3,3.6) \psline(2.8,3.2)(3.2,3.6) \psline(2.75,3.4)(2.8,3.2) \psline(4.5,3.2)(4.5,2.8) \psline(4.5,3.2)(4.1,3.6) \psline(4.6,3.3)(4.3,3.6) \psline(4.6,3.3)(4.5,3.6) \psline(4.6,3.3)(4.7,3.6) \psline(4.5,3.2)(4.9,3.6) \psline(4.5,.8)(4.5,.4) \psline(4.5,.8)(4.1,1.2) \psline(4.5,1)(4.3,1.2) \psline(4.5,.8)(4.5,1.2) \psline(4.5,1)(4.7,1.2) \psline(4.5,.8)(4.9,1.2) \psline(6.3,3.8)(6.3,3.4) \psline(6.3,3.8)(5.9,4.2) \psline(6.3,3.8)(6.1,4.2) \psline(6.45,3.95)(6.3,4.2) \psline(6.45,3.95)(6.5,4.2) \psline(6.3,3.8)(6.7,4.2) \psline(6.3,1.8)(6.3,1.4) \psline(6.3,1.8)(5.9,2.2) \psline(6.3,1.8)(6.1,2.2) \psline(6.35,2)(6.3,2.2) \psline(6.35,2)(6.5,2.2) \psline(6.3,1.8)(6.7,2.2) \psline(6.35,2)(6.3,1.8) \end{pspicture}
\]
\par%
\[
\begin{pspicture}(-1,0)(9.5,6) \rput(0,3){$\overline{I^{+}_{2,1}}$:} \psline(4.5,6)(7,6) \pscurve(2,4.5)(3,5.5)(4.5,6) \psline(2,4.5)(5.5,4.5) \psline(5.5,4.5)(7,6) \psline(7,6)(7,1.7) \psline(3.5,4.5)(3.5,.7) \psline(5.5,4.5)(5.5,0) \psline(5.5,2)(7,3.5) \pscurve(3.5,.7)(4,.4)(5.5,0) \psline(2,4.5)(1.7,2.8) \psline(1.7,2.8)(2.2,1.3) \psline(2.2,1.3)(3.5,.7) \psline(5.5,0)(6.5,.6) \psline(6.5,.6)(7,1.7) \pscurve(3.5,2.2)(4.3,1.6)(5.5,1.2) \pscircle*(4.8,5.2){.1} \psline[linewidth=2.5pt](4.2,5.2)(5.4,5.2) \psline(4.8,5.2)(4.6,5.6) \psline(4.8,5.2)(5,5.6) \psline(5.1,5.2)(5.4,4.95) \pscircle*(2.6,2.8){.1} \psline[linewidth=2.5pt](2,2.8)(3.2,2.8) \psline(2.6,2.4)(2.6,3.05) \psline(2.6,3.05)(2.4,3.3)\psline(2.6,3.05)(2.8,3.3) \pscircle*(4.5,3.2){.1} \psline[linewidth=2.5pt](3.9,3.2)(5.1,3.2) \psline(4.8,3.2)(5.1,3.4) \psline(4.8,3.2)(5.1,3.6) \psline(4.8,3.2)(5.1,3) \pscircle*(4.5,.8){.1} \psline[linewidth=2.5pt](3.9,.8)(5.1,.8) \psline(4.8,.8)(5.1,1) \psline(4.8,.8)(5.1,1.2) \psline(4.5,.8)(4.5,.4) \pscircle*(6.3,3.8){.1} \psline[linewidth=2.5pt](5.7,3.8)(6.9,3.8) \psline(6.3,3.8)(6.3,4.2) \psline(6.6,3.8)(6.9,4.05) \psline(6.6,3.8)(6.9,3.55) \pscircle*(6.3,1.8){.1} \psline[linewidth=2.5pt](5.7,1.8)(6.9,1.8) \psline(6.3,1.4)(6.3,2.2) \psline(6.6,1.8)(6.9,2.05) \end{pspicture}
\]
\caption{Identification between $I_{2,1}^{+}$ and $K_{5}^{\{2,3,4\}}$ (see
also \cite[Figure 18]{S})}%
\label{I-and-K}%
\end{figure}
Given a diagram $D\in I_{k,l}^{+},$ let $T$
be the corresponding tree under the isomorphism $I_{k,l}^{+}\approx
K_{n+2}^{\{2,\dots,k+2\}}$ established in step 1, and let $\alpha$ be the cell of
$B^{n-1}$ identified with $T$ under Stasheff's identification. Then $D$ is
identified with $\alpha$. Our goal is to extend this identification to $\overline{I_{k,l}^{+}}\leftrightarrow \bar{B}^{n-1}$.

Consider a cell $\alpha\subset B^{n-1}$ and the corresponding diagram $D\in
I_{k,l}^{+}.$ If $\bar{\alpha}\subset B^{n-1}$, the cells of $\bar{\alpha}$
have already been identified with generators of $I_{k,l}^{+}$. So assume that
$\bar{\alpha}\cap\partial\bar{B}^{n-1}\neq\varnothing$. If $\dim\alpha=1$, the
$1$-cell of $\alpha$ corresponds to an upper diagram $D\in I_{k,l}^{+}$, which differs from a binary diagram at exactly one vertex, which is either the root with valence $3$ or another vertex of valence $4$. Some edge insertion at this particular vertex produces a binary diagram $D'\in\overline{I_{k,l}^{+}}$ identified with an endpoint of $\bar{\alpha}$ in $\partial \bar{B}^{n-1}.$

Inductively, if $\alpha^{\prime}\subset B^{n-1}$ is a cell of dimension less
than $r,$ assume that all cells in its closure $\bar{\alpha}^{\prime}$ have
been identified with diagrams in $\overline{I_{k,l}^{+}}$, and consider a
diagram $D\in I_{k,l}^{+}$ whose corresponding cell $\alpha\subset B^{n-1}$
has dimension $r$ and satisfies ${\bar{\alpha}\cap\partial \bar{B}^{n-1}}=\varnothing$. 
Since $\alpha$ also corresponds to a tree in $K_{n+2}%
^{\{2,\dots,k+2\}}$, we know that ${\bar{\alpha}\cong\bar{B}^{r}}$ and
$\overline{\partial\bar{\alpha}\cap{B}^{n-1}}\cong\bar{B}^{r-1}$. Using Stasheff's
notation in \cite[Sections 3 and 6]{S}, recall that $\left\vert K_{n+2}%
^{\left\{  2,\ldots,k+2\right\}  }\right\vert $ is homeomorphic to the
complement of the union of the closures of boundary components $\partial
_{j}\left(  s,t\right)  \left(  K_{s}\times K_{t}\right)$ over all $s+t=n+1$
with either $j=1$ or $j>k+2$. Now if $\alpha$ corresponds to a product of cells $\partial_{j}\left(  s,t\right) \left(  K_{s}\times K_{t}\right)$, there is the corresponding identification with some product of diagrams $D'\times D''$, and
Stasheff's decomposition of faces of associahedra as Cartesian products of associahedra identifies $D'\times D''$ with the corresponding diagram $D \in I^+_{k,l}$, the exact form of which is unimportant here.  But what \emph{is} important here is the fact that $|D| \cong B^r$. 
Furthermore, the cells in the boundary component $\overline{\partial\overline{ \alpha}\cap B^{n-1}}$ can be realized by a union of cells $\partial_{j}\left(  s,t\right) \left(  \partial K_{s}\times K_{t}\right)$ or $\partial_{j}\left(  s,t\right) \left(  K_{s}\times \partial K_{t}\right)$ whose union is a cellular complex homeomorphic to $\bar{B}^{r-1}$. By the induction hypothesis, there is a corresponding sum of products of diagrams, which corrsponds to a sum of diagrams in $\overline{I^+_{k,l}}$ such that $\overline{\partial\overline{ \alpha}\cap B^{n-1}}\cong \bar B^{r-1}$. Thus the cellular structure of $\alpha$ extends to all of $\bar \alpha$ in such a way that the cells of $\partial \bar\alpha$ are realizations of the diagrams in $\overline{I^+_{k,l}}$ obtained by inserting edges into $D$ in all possible ways.

Furthermore there is a bijective correspondence between cells of
$\overline{I_{k,l}^{+}}$ and those in $\bar{B}^{n-1}$, which realize
$\overline{I_{k,l}^{+}}$ as a closed $(n-1)$-ball. To see this, note that the
assignment of diagrams from $\overline{I_{k,l}^{+}}$ to cells of $\bar
{B}^{n-1}$ is surjective by construction, and respects the boundary, since it
does so locally for each closed cell. We also claim that two diagrams $D_{1}$
and $D_{2}$ that correspond to the same cell in $\bar{B}^{n-1}$ must be equal.
First, this is clear for $\alpha\subset B^{n-1}$. To see this for $\alpha\subset\partial\bar{B}^{n-1}$, assume $\alpha$ corresponds to
$D_{1}$ and $D_{2}$. Then $\alpha$ is in the boundary of a unique smallest
dimensional cell $\alpha^{\prime}$ in the interior of $B^{n-1}$. Denote by
$D^{\prime}$ the diagram corresponding to $\alpha^{\prime}$. Now, $D_{1}$
itself is in the boundary of some diagram $D^{\prime\prime}$ in $I_{k,l}^{+}$, and we denote by $\alpha^{\prime\prime}$ the cell
corresponding to $D^{\prime\prime}$. Since $D_{1}$ is in the boundary of
$D^{\prime\prime}$, $\alpha$ must be in the boundary of $\alpha^{\prime\prime
}$. Since $\alpha^{\prime}$ is the lowest dimensional cell containing $\alpha
$, $\alpha^{\prime}$ must be in the boundary of $\alpha^{\prime\prime}$ or
equal to it. Since $B^{n-1}$ realizes $I_{k,l}^{+}$, it must also be that
$D^{\prime}$ is in the boundary of $D^{\prime\prime}$ or equal to it. In any
case, since $D^{\prime\prime}$ is realized by $\alpha^{\prime\prime}$, we see
that $D_{1}$ is in the boundary of $D^{\prime}$. A similar argument shows that
$D_{2}$ is also in the boundary of $D^{\prime}$. Again, using that $D^{\prime
}$ is realized by $\alpha^{\prime}$, we see that there is only one diagram
$D_{1}=D_{2}$ corresponding to the cell $\alpha$. Therefore $\overline
{I_{k,l}^{+}}$ is realized by a cell decomposition of $\bar{B}^{n-1}$.

This completes the proof of the lemma.
\end{proof}

\section{Proof that ``$\leq$'' is a partial order}\label{APP:partial-order}

In this appendix, we show that the relation defined Definition \ref{DEF:order} induces a well-defined partial ordering on binary diagrams. For diagrams $\CAxy$ with a coloring $\mathbf{x}=(1,\dots,1), y=1$, this is a well-known fact known as the Tamari partial order, see \cite{Ta}. We will extend this to binary diagrams in $\CA$ of all colors.

It is sufficient to prove antisymmetry: $D=D'$ whenever $D\leq D'$ and $D'\leq D$. For a module diagram $D\in \CAxy$ of coloring $\mathbf{x}=(1,\dots,2,\dots,1),  y=2$ let $\ell(D)$ (respectively $r(D)$) be the number of edges attached to the thick vertical edge coming from the left (respectively from the right), see Figure \ref{FIG:l-r-u}.
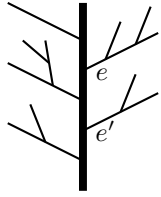
\begin{figure}[ht]
\[
\begin{pspicture}(3,0)(5,2.5) 
 \psline[linewidth=3pt](4,0)(4,2.5) 
 \psline(4,.4)(3,.9)  \psline(4,.8)(5,1.3) \psline(4,1.2)(3,1.7) 
        \psline(4,1.6)(5,2.1) \psline(4,2)(3,2.5)
  \rput(4.3,.8){$e'$}   \rput(4.25,1.55){$e$}
  \psline(4.5,1.05)(4.7,1.55) \psline(4.3,1.75)(4.5,2.25) \psline(4.7,1.95)(4.9,2.45)
  \psline(3.5,.65)(3.3, 1.15) \psline(3.6,1.4)(3.5,2) \psline(3.55,1.7)(3.2,2)
 \end{pspicture}
\]
\caption{An example with $\ell(D)=3$, $r(D)=2$, $u(e)=1$, $u(e')=2$, $u(D)=1+2=3$}\label{FIG:l-r-u}
\end{figure}
 Note that $\ell$ and $r$ respect the order in the sense, if $D\leq D'$ then $\ell(D)\leq \ell(D')$ and $r(D)\geq r(D')$. Now, consider an edge $e$ of $D$ that is attached the the thick vertical edge from the right, and let $u(e)$ denote the number of edges that are attached to the thick edge from the left {\it and} that lie above $e$. Define $u(D)=\sum_{e} u(e)$, where the sum is taken over all edges $e$ that are attached to the thick edge from the right (see Figure \ref{FIG:l-r-u}). Note, that relation (1) from Definition \ref{DEF:order} leaves $u$ invariant, whereas (2) preserves the order, \emph{i.e.}, if $D < D'$ via the relation $(2)$ then $u(D)< u(D')$.

Now, let $D, D' \in \CAxy$ with $\mathbf{x}=(1,\dots,2,\dots,1),  y=2$. If $D\leq D'$ and $D'\leq D$ then we have $\ell(D)=\ell(D')$ and $r(D)=r(D')$. Since relation (3) strictly increases $\ell$ and relation (4) strictly decreases $r$, only relations (1) and (2) can be applied to generate $D\leq D'$ and $D'\leq D$. Furthermore, since (2) strictly increases $u$, we obtain that $u(D)= u(D')$, and only relation (1) can be applied. Using the known fact that (1) is a partial order implies the claim $D=D'$ for module diagrams $D,D'\in \CAxy$.

Finally, for inner product diagrams $D\in \CAxy$ with $\mathbf{x}=(1,\dots,2,\dots,2,\dots,1)$,  $y=0$, let $t\ell(D)$ (resp. $b\ell(D)$, $tr(D)$, $br(D)$) be the number of edges attached to the top of the thick horizontal edge and left of the the thick vertex (resp. bottom left, top right, and bottom right), see Figure \ref{FIG:tl-bl-tr-br}.
\begin{figure}[ht]
\[
\begin{pspicture}(-.3,.2)(3.8,1.8) 
 \psline[linewidth=3pt](0,1)(3.8,1) \pscircle*(1.8,1){.15}
 \psline(.5,1)(0,.2)  \psline(.7,1)(0.2,1.8)  \psline(.9,1)(0.4,.2)
   \psline(1.1,1)(.6,1.8)  \psline(1.5,1)(1,.2) 
  \psline(2.3,1)(2.8,.2) \psline(2.5,1)(3,.2) \psline(2.7,1)(3.2,1.8) 
   \psline(2.9,1)(3.4,.2) \psline(3.1,1)(3.6,.2) \psline(3.3,1)(3.8,.2) 
   \psline(2.55,.6)(2.4,.2) \psline(2.95,1.4)(2.8,1.8) 
  \psline(.65,.6)(.8,.2) \psline(1.25,.6)(1.4,.2) 
  \psline(.85,1.4)(1.2,1.8)  \psline(1,1.57)(.85,1.8)
\end{pspicture}
\]
\caption{An example with $t\ell(D)=2$, $b\ell(D)=3$, $tr(D)=1$, and $br(D)=5$}\label{FIG:tl-bl-tr-br}
\end{figure}
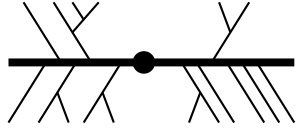
For $D\leq D'$, we have that $t\ell(D)\geq t\ell(D')$, $b\ell(D)\leq b\ell(D')$, $tr(D)\leq tr(D')$, and $br(D)\geq br(D')$. Thus, for $D\leq D'$ and $D'\leq D$, all of $t\ell$, $b\ell$, $tr$, and $br$ coincide on $D$ and $D'$. Since relations (3)-(6) change at least one of $t\ell$, $b\ell$, $tr$, or $br$, these cannot be applied to generate $D\leq D'$ and $D'\leq D$. To finish the proof, we refer back to the case of the module diagram with  $\mathbf{x}=(1,\dots,2,\dots,1),  y=2$, using the function $u$ and the Tamari partial ordering as before to see that $D=D'$.

\section{Proof of Proposition \ref{p-chain} ($p$ is a chain map)}\label{p-chain-proof}

In this appendix, we proof that $p:\QA\to\CA$ is a chain map. The proof is an extension of Markl and Shnider's proof of Proposition 4.6 in \cite{MS} and uses the notion of positive or negative edges in a binary inner product diagram as defined in Definition \ref{DEF:pos-neg-edge} in Appendix \ref{APP:binary-standard-orientation}. The maximal binary diagrams is the one with only positive edges, and the minimal binary diagram is the one with only negative edges. (In \cite{MS} these edges were called left and of right leaning edges, respectively.) The main part of the proof amounts to checking that for $(D,\fcan,m,\omega_D)\in \Qa{k}$ such that $D_{\min}$ has either $k$ or $(k-1)$ positive edges,  $\partial_C \circ p(D,\fcan,m,\omega_D)=p\circ \partial_Q (D,\fcan,m,\omega_D)$.  The case of $(k-1)$ positive edges will be checked in Lemma \ref{k-1-lemma}, and the case of $k$ positive edges will be checked by induction beginning the induction in Lemma \ref{B_max-lemma}. 

\begin{proof}[Proof of Proposition \ref{p-chain}]
Since $\partial_C$ and $\partial_Q$ are derivations and $p$ is multiplicative with respect to the composition $\circ_i$, it is enough to show that $\partial_C\circ p=p\circ \partial_Q$ on fully metric diagrams in $D\in\QA$. We do this by induction on the number of leaves $n=\mathcal L(D)$. For $n=2$, it is trivial to check that $p:\Qa{*}\to \Ca{*}$ is a chain map, since in this case $\partial_Q=0$ and $\partial_C=0$. We now assume that $p:\Qa{*}\to \Ca{*}$ is a chain map when applied to any diagram with $\mathcal L(D)<n$ leaves, or compositions of diagrams with $\mathcal L(D)<n$ leaves. We need to show the same is true for metric diagrams with $\mathcal L(D)=n$ leaves.

For $(D,\fcan,m,\omega_D)\in \Qa{k}$ with $n=\mathcal L(D)$ leaves, we have by Lemma \ref{neq-k}(3), that $p(D,\fcan,m,\omega_D)=0$ when $|D_{\min}|_\Pos\neq k$. Furthermore, if $\partial_Q (D,\fcan,m,\omega_D)=\sum_i (D_i,\dots)$ and $|(D_i)_{\min}|_\Pos\neq k-1$ then $p(D_i,\dots)=0$. Since $\partial_Q$ either preserves $|.|_\Pos$ or decreases it by one, we see that $|(D_i)_{\min}|_\Pos$ either equals $|D_{\min}|_{\Pos}$ or $|D_{\min}|_\Pos-1$. Thus, $p(\partial_Q (D,\fcan,m,\omega_D))=0$ when $|D_{\min}|_\Pos\notin \{k,k-1\}$. In particular for $|D_{\min}|_\Pos\notin \{k,k-1\}$, we have $p(\partial_Q(D,\fcan,m,\omega_D))=0=\partial_C(p(D,\fcan,m,\omega_D))$. In the case of  $|D_{\min}|_\Pos=k-1$, $p(\partial_Q(D,\fcan,m,\omega_D))=\partial_C(p(D,\fcan,m,\omega_D))$ is the statement of Lemma \ref{k-1-lemma} below.

It remains to check the case $|D_{\min}|_\Pos=k$. The proof is a second induction on $k$ starting from $k=n$. 
If $(B,\fcan,m,\omega_B)\in \Qa{n}$ is a binary diagram with $|B|_\Pos=n$, then $B$ is the maximal binary diagram, and $p(\partial_Q(B,\fcan,m,\omega_B))=\partial_C(p(B,\fcan,m,\omega_B))$ is checked in Lemma \ref{B_max-lemma} below. Thus, we may now assume $p(\partial_Q(D,\fcan,m,\omega_D))=\partial_C(p(D,\fcan,m,\omega_D))$ for all $(D,\fcan,m,\omega_D)\in \QA$ with either
\begin{itemize}
\item $\mathcal L(D)<n$ leaves, or compositions thereof,
\item $\mathcal L(D)=n$ and $(D,\fcan,m,\omega_D)\in \Qa{k}$ but $|D_{\min}|_\Pos\neq k$, or
\item $\mathcal L(D)=n$ and $(D,\fcan,m,\omega_D)\in \Qa{l}$ with $k<l\leq n$.
\end{itemize}
We wish to prove $p(\partial_Q(D,\fcan,m,\omega_D))=\partial_C(p(D,\fcan,m,\omega_D))$ for $(D,\fcan,m,\omega_D)\in \Qa{k}$ with $n=\mathcal L(D)$ leaves and $|D_{\min}|_\Pos=k$ (\emph{i.e.} all $k$ edges in $D$ are inserted as positive edges). 

Since $D$ is non-binary, there exists a ternary (or higher) vertex $v$ of $D$. Following \cite{MS}, we denote by $D^+$ the diagram given by inserting $v$ in $D$ in such a way that $D^+_{\min}$ has an extra negative edge $e_v$. This insertion is always possible. One approach is to use the following scheme:
\[
\resizebox{!}{1.5cm}{
\begin{pspicture}(0,1)(3,3)
 \rput(.5,2){$D=$} \pscircle(2,2){.9}
 \psline(2,2)(1.2,3)
  \psline(2,2)(1.6,3)
 \psline(2,1)(2,3)
 \psline(2,2)(2.4,3)
 \psline(2,2)(2.8,3)
\end{pspicture}}
\quad\quad\quad
\resizebox{!}{1.5cm}{
\begin{pspicture}(0,1)(3,3)
 \rput(0,2){$\Rightarrow D^+=$}  \pscircle(2,2){.9}
 \psline(2,2)(1.2,3)
 \psline(1.7,2.375)(1.6,3)
 \psline(2,1)(2,2)
 \psline(1.7,2.375)(2,3)
 \psline(1.7,2.375)(2.4,3)
 \psline(2,2)(2.8,3)
\end{pspicture}}
\]
\[
\resizebox{!}{1.5cm}{
\begin{pspicture}(0,1)(3,3)
 \rput(.5,2){$D=$}\pscircle(2,2){.9}
 \psline(2,2)(1.1,3)
 \psline(2,2)(1.4,3)
 \psline(2,2)(1.7,3)
 \psline(2,2)(2.3,3)
 \psline(2,2)(2.6,3)
 \psline(2,2)(2.9,3)
 \psline[linewidth=3pt](2,1)(2,3)
\end{pspicture}}
\quad\quad\quad
\resizebox{!}{1.5cm}{
\begin{pspicture}(0,1)(3,3)
 \rput(0,2){$\Rightarrow D^+=$}\pscircle(2,2){.9}
 \psline(2,2)(1.1,3)
 \psline(1.65,2.38)(1.4,3)
 \psline(1.65,2.38)(1.7,3)
 \psline(2,2)(2.3,3)
 \psline(2,2)(2.6,3)
 \psline(2,2)(2.9,3)
 \psline[linewidth=3pt](2,1)(2,3)
\end{pspicture}}
\resizebox{!}{1.5cm}{
\begin{pspicture}(0,1)(3,3)
 \rput(.5,2){or}\pscircle(2,2){.9}
 \psline(2,2)(1.1,3)
 \psline(2,2)(1.4,3)
 \psline(2,2)(1.7,3)
 \psline(2,2.5)(2.3,3)
 \psline(2,2.5)(2.6,3)
 \psline(2,2.5)(2.9,3)
 \psline[linewidth=3pt](2,1)(2,3)
\end{pspicture}}
\resizebox{!}{1.5cm}{
\begin{pspicture}(0,1)(3,3)
 \rput(.5,2){or}\pscircle(2,2){.9}
 \psline(2,2)(1.1,3)
 \psline(2,2)(1.4,3)
 \psline(2,2)(1.7,3)
 \psline(2,2.5)(2.3,3)
 \psline(2,2.5)(2.6,3)
 \psline(2,2)(2.9,3)
 \psline[linewidth=3pt](2,1)(2,3)
\end{pspicture}}
\]
\[
\resizebox{!}{1.5cm}{
\begin{pspicture}(-.5,1)(3.5,3)
 \rput(0,2){$D=$}  \pscircle(2,2){.9}
 \psline(2,2)(1.6,3)
 \psline(2,2)(2,3)
 \psline(2,2)(2.4,3)
 \psline(2,2)(1.4,1)
 \psline(2,2)(1.8,1)
 \psline(2,2)(2.2,1)
 \psline(2,2)(2.6,1)
 \psline[linewidth=3pt](.7,2)(3.3,2)
 \pscircle*(2,2){.15}
\end{pspicture}}
\quad\quad\quad
\resizebox{!}{1.5cm}{
\begin{pspicture}(-.5,1)(3.5,3)
 \rput(-.5,2){$\Rightarrow D^+=$}  \pscircle(2,2){.9}
 \psline(1.5,2)(1.6,3)
 \psline(1.5,2)(2,3)
 \psline(1.5,2)(2.4,3)
 \psline(2,2)(1.4,1) 
 \psline(2,2)(1.8,1)
 \psline(2,2)(2.2,1)
 \psline(2,2)(2.6,1)
 \psline[linewidth=3pt](.7,2)(3.3,2)
 \pscircle*(2,2){.15}
\end{pspicture}}
\resizebox{!}{1.5cm}{
\begin{pspicture}(-.5,1)(3.5,3)
 \rput(0,2){or}  \pscircle(2,2){.9}
 \psline(2,2)(1.6,3)
 \psline(2,2)(2,3)
 \psline(2,2)(2.4,3)
 \psline(2.5,2)(1.4,1)
 \psline(2.5,2)(1.8,1)
 \psline(2.5,2)(2.2,1)
 \psline(2.5,2)(2.6,1)
 \psline[linewidth=3pt](.7,2)(3.3,2)
 \pscircle*(2,2){.15}
\end{pspicture}}
\]
Here, the diagram may be completed outside the circle in an arbitrary way, and the choice of $D^+$ may be determined by the number of incoming edges of a certain type.

We label the new edge $e_v$ as metric, so that we calculate $\partial_Q(D^+,\fcan,m,\omega)=\sum_e (D^+_e,\dots)+\sum_{e\neq e_v} (D^+/e,\dots) + (D^+/{e_v},\dots)$, where $D^+_e$ is obtained by making the edge $e$ non-metric, $D^+/e$ is obtained by collapsing the edge $e$, and $D^+/e_v=D$. Note, that $D^+_e$ has a non-metric edge and is thus a composition of diagrams with fewer than $n$ leaves, $(D^+/e,\dots)\in \Qa{k}$ with $|(D^+/e)_{\min}|_\Pos=k-1$, and $(D^+,\dots)\in \Qa{k+1}$ all satisfy the chain condition $\partial_C p=p\partial_Q$ by the above hypothesis. Furthermore, $(D^+,\dots)$ satisfying the chain condition implies the same for $\partial_Q(D^+,\dots)$, since $\partial_C p \partial_Q (D^+,\dots)=\partial_C\partial_C p (D^+,\dots) =0=p \partial_Q\partial_Q (D^+,\dots)$. Thus, $(D,\dots)=\partial_Q(D^+,\dots)-\sum_e (D^+_e,\dots)-\sum_{e\neq e_v} (D^+/e,\dots)$ also satisfies the chain condition. This completes the inductive step.
\end{proof}

\begin{lemma}\label{k-1-lemma}
Let $(D,\fcan,m,\omega_D)\in \Qa{k}$ be a fully metric diagram ($D$ has $k$ metric edges) with $|D_{\min}|_\Pos=k-1$. Then
\[p(\partial_Q(D,\fcan,m,\omega_D))=0=\partial_C(p(D,\fcan,m,\omega_D)). \]
\end{lemma}
\begin{proof}
Lemma \ref{neq-k}(3) implies that $\partial_C(p(D,\fcan,m,\omega_D))=0$; it is sufficient to show that $p(\partial_Q(D,\fcan,m,\omega_D))=0$. Recall that $D_{\min}$ is given by inserting negative edges in $D$. Consequently, the positive edges in $D_{\min}$ do not come from inserting edges into $D$. Let $e_1,\dots,e_{k-1}$ be the edges of $D$ that become positive edges in $D_{\min}$, and let $e_0$ be the edge of $D$ that becomes a negative edge in $D_{\min}$. Direct inspection shows that an inserted edge $e_0$ in $D_{\min}$ can be negative in six possible ways.
Figure \ref{5cases} shows the ``local'' pictures for these cases and their minima, where ``local'' means within a bigger diagram.
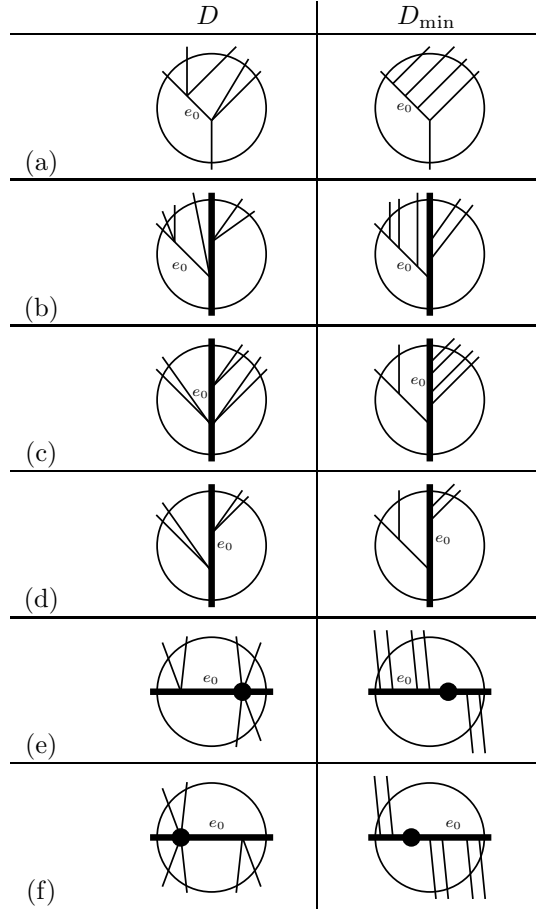
\begin{figure*}
$\begin{array}{cc|c}
 & D & D_{\min} \\ \hline
 \text{(a)}\quad & 
 \resizebox{!}{1.8cm}{ \begin{pspicture}(2.5,0)(5.5,2.2) 
\pscircle(4,1){.9} \psline(4,.8)(4.6,1.8) \psline(3.6,1.2)(3.6,2) \rput(3.7,.9){$_{e_0}$}
 \psline(4,0)(4,.8) \psline(4,.8)(3.2,1.6) \psline(4,.8)(4.8,1.6) \psline(3.6,1.2)(4.4,2)
 \end{pspicture}}
 & 
 \resizebox{!}{1.8cm}{ \begin{pspicture}(2.5,0)(5.5,2.2) 
\pscircle(4,1){.9} \psline(3.8,1)(4.6,1.8) \psline(3.4,1.4)(4,2) \rput(3.6,1){$_{e_0}$}
 \psline(4,0)(4,.8) \psline(4,.8)(3.2,1.6) \psline(4,.8)(4.8,1.6) \psline(3.6,1.2)(4.4,2)
 \end{pspicture}}
 \\ \hline
 \text{(b)} \quad & 
 \resizebox{!}{1.8cm}{ \begin{pspicture}(2.5,0)(5.5,2.2) 
\pscircle(4,1){.9} \psline[linewidth=3pt](4,0)(4,2)
 \psline(4,.6)(3.1, 1.5) \psline(4,1.2)(4.5,1.9) \rput(3.5,.8){$_{e_0}$}
 \psline(4,1.2)(4.7,1.7) \psline(3.98,.6)(3.7,2) 
 \psline(3.4,1.2)(3.4,1.8) \psline(3.4,1.2)(3.2,1.7)
 \end{pspicture}}
  & 
 \resizebox{!}{1.8cm}{ \begin{pspicture}(2.5,0)(5.5,2.2) 
\pscircle(4,1){.9} \psline[linewidth=3pt](4,0)(4,2)
 \psline(4,.6)(3.1, 1.5) \psline(4,1.2)(4.5,1.9) \rput(3.6,.8){$_{e_0}$}
 \psline(4,.9)(4.7,1.8) \psline(3.8,.8)(3.8,2) 
 \psline(3.5,1.1)(3.5,1.9) \psline(3.35,1.25)(3.35,1.85) 
 \end{pspicture}}
  \\ \hline
 \text{(c)} \quad & 
 \resizebox{!}{1.8cm}{ \begin{pspicture}(2.5,0)(5.5,2.2) 
\pscircle(4,1){.9} \psline[linewidth=3pt](4,0)(4,2)
 \psline(4,.6)(3.1, 1.5) \psline(4,1.2)(4.5,1.9) \rput(3.83,1.1){$_{e_0}$}
 \psline(4,1.2)(4.6,1.8) \psline(3.98,.6)(3.2,1.7) 
 \psline(4.02,.6)(4.8,1.7) \psline(4,.6)(4.9,1.5)
 \end{pspicture}}
  & 
 \resizebox{!}{1.8cm}{ \begin{pspicture}(2.5,0)(5.5,2.2) 
\pscircle(4,1){.9} \psline[linewidth=3pt](4,0)(4,2) \rput(3.83,1.3){$_{e_0}$}
 \psline(4,.6)(3.1, 1.5) \psline(3.5,1.1)(3.5,1.9) 
 \psline(4,.9)(4.8,1.7)  \psline(4,1.1)(4.7,1.8)  
 \psline(4,1.4)(4.5,1.9)  \psline(4,1.6)(4.4,2)  
 \end{pspicture}}
 \\ \hline
 \text{(d)} \quad & 
 \resizebox{!}{1.8cm}{ \begin{pspicture}(2.5,0)(5.5,2.2) 
\pscircle(4,1){.9} \psline[linewidth=3pt](4,0)(4,2)
 \psline(4,.6)(3.1, 1.5) \psline(4,1.2)(4.5,1.9) \rput(4.23,1){$_{e_0}$}
 \psline(4,1.2)(4.6,1.8) \psline(3.98,.6)(3.2,1.7) 
 \end{pspicture}}
  & 
 \resizebox{!}{1.8cm}{ \begin{pspicture}(2.5,0)(5.5,2.2) 
\pscircle(4,1){.9} \psline[linewidth=3pt](4,0)(4,2) \rput(4.23,1.1){$_{e_0}$}
 \psline(4,.6)(3.1, 1.5) \psline(3.5,1.1)(3.5,1.9) 
 \psline(4,1.4)(4.5,1.9)  \psline(4,1.6)(4.4,2)  
 \end{pspicture}}
  \\ \hline
 \text{(e)} \quad & 
 \resizebox{!}{1.8cm}{ \begin{pspicture}(2.5,0)(5.5,2.2) 
\pscircle(4,1){.9} \pscircle*(4.5,1){.15} \rput(4,1.2){$_{e_0}$}
 \psline[linewidth=3pt](3,1)(5,1)
 \psline(3.5,1)(3.2,1.8) \psline(3.5,1)(3.6,1.9)
 \psline(4.5,1)(4.8,1.8) \psline(4.5,1)(4.4,1.9)
 \psline(4.5,1)(4.8,.2) \psline(4.5,1)(4.4,.1)
 \end{pspicture}}
  & 
   \resizebox{!}{1.8cm}{ \begin{pspicture}(2.5,0)(5.5,2.2) 
\pscircle(4,1){.9} \pscircle*(4.3,1){.15} \rput(3.6,1.2){$_{e_0}$}
 \psline[linewidth=3pt](3,1)(5,1)
 \psline(3.2,1)(3.1,2) \psline(3.4,1)(3.3,2)
 \psline(4,1)(3.9,2) \psline(3.8,1)(3.7,2)
 \psline(4.6,1)(4.7,0) \psline(4.8,1)(4.9,0)
 \end{pspicture}}
 \\ \hline
 \text{(f)} \quad & 
 \resizebox{!}{1.8cm}{ \begin{pspicture}(2.5,0)(5.5,2.2) 
\pscircle(4,1){.9} \pscircle*(3.5,1){.15} \rput(4.1,1.2){$_{e_0}$}
 \psline[linewidth=3pt](3,1)(5,1)
 \psline(3.5,1)(3.2,1.8) \psline(3.5,1)(3.6,1.9)
 \psline(3.5,1)(3.2,.2) \psline(3.5,1)(3.6,.1)
 \psline(4.5,1)(4.8,.2) \psline(4.5,1)(4.4,.1)
 \end{pspicture}}
& 
   \resizebox{!}{1.8cm}{ \begin{pspicture}(2.5,0)(5.5,2.2) 
\pscircle(4,1){.9} \pscircle*(3.7,1){.15} \rput(4.4,1.2){$_{e_0}$}
 \psline[linewidth=3pt](3,1)(5,1)
 \psline(3.2,1)(3.1,2) \psline(3.4,1)(3.3,2)
 \psline(4,1)(4.1,0) \psline(4.2,1)(4.3,0)
 \psline(4.6,1)(4.7,0) \psline(4.8,1)(4.9,0)
 \end{pspicture}}
\end{array}$
\caption{Here are the six ways for an edge to be inserted as a negative edge when considering $D_{\min}$. }\label{5cases}
\end{figure*}

Recall from Definition \ref{operad-QA} that $\partial_Q(D,\fcan,m,\omega_D)\in\Qa{k-1}$ is given by a sum of diagrams $\sum_{i=0}^{k-1} D/e_i+ \sum_{i=0}^{k-1} D_{e_i}$, where $D/e_i$ is given by collapsing $e_i$, and $D_{e_i}$ is given by relabeling $e_i$ non-metric. Note that for $i>0$, $T_{e_i}$ is a composition of two diagrams along $e_i$. Consequently one of these two diagrams contains $e_0$, giving a negative edge in its minimum, so that $p(D_{e_i},\dots)=0$ by Lemma \ref{neq-k}. Similarly, for those edges $e_i$ such that $D/e_i$ has less than $k-1$ positive edges, we have $p(D/e_i,\dots)=0$ by Lemma \ref{neq-k}. However, this only happens in restricted cases. First, $(D/e_0)_{\min}=D_{\min}$ so that $|(D/e_0)_{\min}|_\Pos=k-1$, and additionally, in certain cases discussed below, it is possible that there is one more edge, denoted by $e'$, such that $e_0$ is converted to a positive edge when considering $(D/e')_{\min}$. Figure \ref{Ex:pdT=0} illustrates this situation.
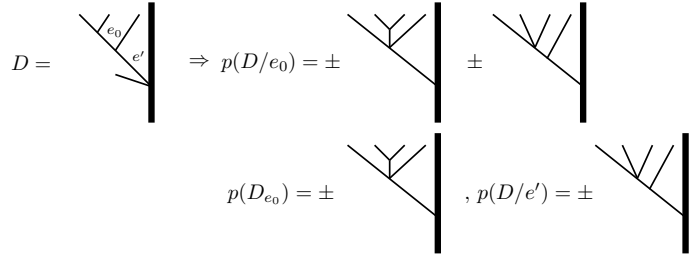
\begin{figure*}
\begin{eqnarray*}
\resizebox{!}{1.6cm}{ \begin{pspicture}(1.5,0)(5.2,2) 
 \psline[linewidth=3pt](4,0)(4,2) \rput(3.4,1.5){$_{e_0}$} \rput(3.8,1.1){$_{e'}$}
 \psline(4,.6)(2.8, 1.8) \psline(4,.6)(3.4,.8) \rput(2,1){$D=$} \rput(4.8,1){$\Rightarrow$}
 \psline(3.4,1.2)(3.8,1.8) \psline(3.1,1.5)(3.3,1.8)
 \end{pspicture}}  
 &
\resizebox{!}{1.6cm}{ \begin{pspicture}(1,0)(4.3,2) 
 \psline[linewidth=3pt](4,0)(4,2) 
 \psline(4,.6)(2.5, 1.8)  \rput(1.4,1){$p(D/e_0)=\pm$}
 \psline(3.2,1.25)(3.8,1.8) \psline(3.2,1.25)(3.2,1.55)
 \psline(3.2,1.55)(3.45,1.8)  \psline(3.2,1.55)(2.95,1.8)
 \end{pspicture}}  
 \resizebox{!}{1.6cm}{ \begin{pspicture}(2,0)(5.9,2) 
 \psline[linewidth=3pt](4,0)(4,2) 
 \psline(4,.6)(2.5, 1.8)  \rput(2.2,1){$\pm$}
 \psline(3.4,1.066)(3.8,1.8)
 \psline(3.2,1.25)(3.45,1.8)  \psline(3.2,1.25)(2.95,1.8)
 \end{pspicture}}  
  \\
 & 
\resizebox{!}{1.6cm}{ \begin{pspicture}(1,0)(4.3,2) 
 \psline[linewidth=3pt](4,0)(4,2) 
 \psline(4,.6)(2.5, 1.8)  \rput(1.4,1){$p(D_{e_0})=\pm$}
 \psline(3.2,1.25)(3.8,1.8) \psline(3.2,1.25)(3.2,1.55)
 \psline(3.2,1.55)(3.45,1.8)  \psline(3.2,1.55)(2.95,1.8)
 \end{pspicture}}  
\resizebox{!}{1.6cm}{ \begin{pspicture}(.3,0)(4.2,2) 
 \psline[linewidth=3pt](4,0)(4,2) 
 \psline(4,.6)(2.5, 1.8)  \rput(1.4,1){, $p(D/e')=\pm$}
 \psline(3.4,1.066)(3.8,1.8)
 \psline(3.2,1.25)(3.45,1.8)  \psline(3.2,1.25)(2.95,1.8)
 \end{pspicture}}  
 \end{eqnarray*}
\caption{In the diagram $D$ with edges $e_0$ and $e'$ as above, we have $p(D/e_0,\dots)=\pm p(D_{e_0},\dots)\pm p(D/e',\dots)$ and $p(D_{e'},\dots)=0$.}\label{Ex:pdT=0}
\end{figure*}
To summarize, $p(D/e_i,\dots)=0$ for all $e_i\notin \{e_0,e'\}$, and the check for $p(\partial_Q(D,\fcan,m,\omega_D))=0$ reduces to two cases:
$$ 
\bigg\{
\begin{array}{lll}
\text{Case }1:& p(D/e_0,\dots)=p(D_{e_0},\dots) & \text{, when $D$ has no edge }e',\\
\text{Case }2:& p(D/e_0,\dots)=p(D_{e_0},\dots)\pm p(D/e',\dots) & \text{, when $D$ has an  edge }e'.
\end{array}
$$

\smallskip
{\bf Case 1:} 
The diagrams in Figure \ref{case1:no-e'} depict the situations in which this case can occur.
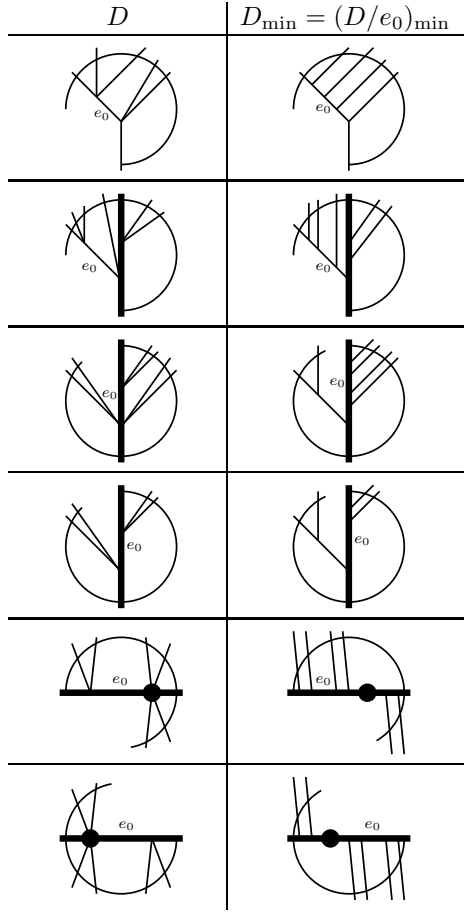
\begin{figure*}
$\begin{array}{c|c}
  D & D_{\min}=(D/e_0)_{\min} \\ \hline
 \resizebox{!}{1.8cm}{ \begin{pspicture}(2.5,0)(5.5,2.2) 
\psarc(4,1){.9}{-90}{180} \psline(4,.8)(4.6,1.8) \psline(3.6,1.2)(3.6,2) \rput(3.7,.9){$_{e_0}$}
 \psline(4,0)(4,.8) \psline(4,.8)(3.2,1.6) \psline(4,.8)(4.8,1.6) \psline(3.6,1.2)(4.4,2)
 \end{pspicture}}
 & 
 \resizebox{!}{1.8cm}{ \begin{pspicture}(2.5,0)(5.5,2.2) 
\psarc(4,1){.9}{-90}{180} \psline(3.8,1)(4.6,1.8) \psline(3.4,1.4)(4,2) \rput(3.6,1){$_{e_0}$}
 \psline(4,0)(4,.8) \psline(4,.8)(3.2,1.6) \psline(4,.8)(4.8,1.6) \psline(3.6,1.2)(4.4,2)
 \end{pspicture}}
 \\ \hline
 \resizebox{!}{1.8cm}{ \begin{pspicture}(2.5,0)(5.5,2.2) 
\psarc(4,1){.9}{-90}{180} \psline[linewidth=3pt](4,0)(4,2)
 \psline(4,.6)(3.1, 1.5) \psline(4,1.2)(4.5,1.9) \rput(3.5,.8){$_{e_0}$}
 \psline(4,1.2)(4.7,1.7) \psline(3.98,.6)(3.7,2) 
 \psline(3.4,1.2)(3.4,1.8) \psline(3.4,1.2)(3.2,1.7)
 \end{pspicture}}
  & 
 \resizebox{!}{1.8cm}{ \begin{pspicture}(2.5,0)(5.5,2.2) 
\psarc(4,1){.9}{-90}{180} \psline[linewidth=3pt](4,0)(4,2)
 \psline(4,.6)(3.1, 1.5) \psline(4,1.2)(4.5,1.9) \rput(3.6,.8){$_{e_0}$}
 \psline(4,.9)(4.7,1.8) \psline(3.8,.8)(3.8,2) 
 \psline(3.5,1.1)(3.5,1.9) \psline(3.35,1.25)(3.35,1.85) 
 \end{pspicture}}
  \\ \hline
 \resizebox{!}{1.8cm}{ \begin{pspicture}(2.5,0)(5.5,2.2) 
\psarc(4,1){.9}{-225}{90} \psline[linewidth=3pt](4,0)(4,2)
 \psline(4,.6)(3.1, 1.5) \psline(4,1.2)(4.5,1.9) \rput(3.83,1.1){$_{e_0}$}
 \psline(4,1.2)(4.6,1.8) \psline(3.98,.6)(3.2,1.7) 
 \psline(4.02,.6)(4.8,1.7) \psline(4,.6)(4.9,1.5)
 \end{pspicture}}
  & 
 \resizebox{!}{1.8cm}{ \begin{pspicture}(2.5,0)(5.5,2.2) 
\psarc(4,1){.9}{-245}{90} \psline[linewidth=3pt](4,0)(4,2) \rput(3.83,1.3){$_{e_0}$}
 \psline(4,.6)(3.1, 1.5) \psline(3.5,1.1)(3.5,1.9) 
 \psline(4,.9)(4.8,1.7)  \psline(4,1.1)(4.7,1.8)  
 \psline(4,1.4)(4.5,1.9)  \psline(4,1.6)(4.4,2)  
 \end{pspicture}}
 \\ \hline
 \resizebox{!}{1.8cm}{ \begin{pspicture}(2.5,0)(5.5,2.2) 
\psarc(4,1){.9}{-225}{90} \psline[linewidth=3pt](4,0)(4,2)
 \psline(4,.6)(3.1, 1.5) \psline(4,1.2)(4.5,1.9) \rput(4.23,1){$_{e_0}$}
 \psline(4,1.2)(4.6,1.8) \psline(3.98,.6)(3.2,1.7) 
 \end{pspicture}}
  & 
 \resizebox{!}{1.8cm}{ \begin{pspicture}(2.5,0)(5.5,2.2) 
\psarc(4,1){.9}{-245}{90} \psline[linewidth=3pt](4,0)(4,2) \rput(4.23,1.1){$_{e_0}$}
 \psline(4,.6)(3.1, 1.5) \psline(3.5,1.1)(3.5,1.9) 
 \psline(4,1.4)(4.5,1.9)  \psline(4,1.6)(4.4,2)  
 \end{pspicture}}
  \\ \hline
 \resizebox{!}{1.8cm}{ \begin{pspicture}(2.5,0)(5.5,2.2) 
\psarc(4,1){.9}{-80}{180} \pscircle*(4.5,1){.15} \rput(4,1.2){$_{e_0}$}
 \psline[linewidth=3pt](3,1)(5,1)
 \psline(3.5,1)(3.2,1.8) \psline(3.5,1)(3.6,1.9)
 \psline(4.5,1)(4.8,1.8) \psline(4.5,1)(4.4,1.9)
 \psline(4.5,1)(4.8,.2) \psline(4.5,1)(4.4,.1)
 \end{pspicture}}
  & 
   \resizebox{!}{1.8cm}{ \begin{pspicture}(2.5,0)(5.5,2.2) 
\psarc(4,1){.9}{-60}{180} \pscircle*(4.3,1){.15} \rput(3.6,1.2){$_{e_0}$}
 \psline[linewidth=3pt](3,1)(5,1)
 \psline(3.2,1)(3.1,2) \psline(3.4,1)(3.3,2)
 \psline(4,1)(3.9,2) \psline(3.8,1)(3.7,2)
 \psline(4.6,1)(4.7,0) \psline(4.8,1)(4.9,0)
 \end{pspicture}}
 \\ \hline
 \resizebox{!}{1.8cm}{ \begin{pspicture}(2.5,0)(5.5,2.2) 
\psarc(4,1){.9}{100}{360} \pscircle*(3.5,1){.15} \rput(4.1,1.2){$_{e_0}$}
 \psline[linewidth=3pt](3,1)(5,1)
 \psline(3.5,1)(3.2,1.8) \psline(3.5,1)(3.6,1.9)
 \psline(3.5,1)(3.2,.2) \psline(3.5,1)(3.6,.1)
 \psline(4.5,1)(4.8,.2) \psline(4.5,1)(4.4,.1)
 \end{pspicture}}
& 
   \resizebox{!}{1.8cm}{ \begin{pspicture}(2.5,0)(5.5,2.2) 
\psarc(4,1){.9}{120}{360} \pscircle*(3.7,1){.15} \rput(4.4,1.2){$_{e_0}$}
 \psline[linewidth=3pt](3,1)(5,1)
 \psline(3.2,1)(3.1,2) \psline(3.4,1)(3.3,2)
 \psline(4,1)(4.1,0) \psline(4.2,1)(4.3,0)
 \psline(4.6,1)(4.7,0) \psline(4.8,1)(4.9,0)
 \end{pspicture}}
\end{array}$
\caption{Case $1$: $p$ is non-vanishing only for $D/e_0$ and $D_{e_0}$. The open ``half-circles'' for $D$ indicate that each diagram $D$ may be completed in an arbitrary way outside the half-circle. Thus, a half-circle that connects to an edge denotes that more edges may be added on the side of the edge outside the half-circle.}\label{case1:no-e'}
\end{figure*}
The proof that $p(D/e_0,\dots)=p(D_{e_0},\dots)$ now follows. Since $e_0$ in $D_{e_0}$ is non-metric, $D_{e_0}$ decomposes as $$(D_{e_0},\fcan,m_{e_0},\omega'\wedge \omega'')=\sigma\cdot((D',\fcan,m,\omega')\circ_{e_0}(D'',\fcan,m,\omega'')),$$ so that 
\begin{eqnarray*}
p(D_{e_0},\fcan,m_{e_0},\omega'\wedge \omega'')&=&\sigma\cdot(p(D',\fcan,m,\omega')\circ_{e_0}p(D'',\fcan,m,\omega''))\\
&=&\sum_{S',S''} \sigma\cdot((S',\fcan,\omega(S',D'))\circ_{e_0}(S'',\fcan,\omega(S'',D''))),
\end{eqnarray*}
where we sum over all $S'$ and $S''$ with $S'_{\max}\leq D'_{\min}$ and $S''_{\max}\leq D''_{\min}$. On the other hand, $$p(D/e_0,\fcan,m,\omega'\wedge \omega'')=\sum(S,\fcan,\omega(S,D)),$$ where we sum over all $S$ with $S_{\max}\leq (D/e_0)_{\min}=D_{\min}$. Consequently, the equality $p(D/e_0,\dots)=p(D_{e_0},\dots)$ follows (up to sign) by noting that each $S_{\max}$ with $S_{\max}\leq D_{\min}$ is given by a composition along $e_0$, \emph {i.e.}, $S_{\max}=S'_{\max}\circ_{e_0}S''_{\max}$ (since, by the order relations, a change of $e_0$ when reducing $D_{\min}$ also changes the number of positive edges).
The proof of this case will be complete, once we compare the signs for $S'\circ_{e_0}S''$ and $S$. To calculate these signs, we assume that $D'$ has $r$ leaves and degree $p$  with $\omega'=e'_1\wedge\dots\wedge e'_p$, and $D''$ has $s$ leaves and degree $q$ with $\omega''=e''_1\wedge\dots\wedge e''_q$. Let $\xi_{D_{\min}}$, $\xi_{D'_{\min}}$ and $\xi_{D''_{\min}}$ denote the orientations from Step 1 in Appendix \ref{APP:binary-standard-orientation}. We first assume that $S_{\max}=D_{\min}$; then $S'_{\max}=D'_{\min}$ and $S''_{\max}=D''_{\min}$, and $\omega(S,D)=(\omega'\wedge \omega'')\rfloor \xi_{D_{\min}}$, $\omega(S',D')=\omega'\rfloor \xi_{D'_{\min}}$, and $\omega(S'',D'')=\omega''\rfloor \xi_{D''_{\min}}$. We need to consider the cases $\sigma=id$ or $\sigma\neq id$; the latter occurs when $S'$ is an inner product diagram and $e_0$ is a thick module edge on the left of $S'$.
\begin{itemize}
\item Let $\sigma=id$; then either $S'$ is not an inner product diagram or $\circ_{e_0}$ is not a composition at the first position. Assume $e_0$ composes at the $i\T$ position. Then,
\begin{multline*}
\quad\quad\quad(S',\fcan, \omega(S',D'))\circ_{i}(S'',\fcan, \omega(S'',D''))\\
=(-1)^{i(s+1)+r\cdot q}(S'\circ_{e_0}S'',\fcan,(\omega'\rfloor \xi_{D'_{\min}})\wedge(\omega''\rfloor \xi_{D''_{\min}})\wedge e_0).
\end{multline*}
On the other hand, since $D_{\min}=D'_{\min}\circ_{e_0}D''_{\min}$, Remark \ref{REM:xi-by-composition} shows that $\xi_{D_{\min}}$ is given by the canonically induced orientation coming from composing $D'_{\min}$ and $D''_{\min}$ with orientations $\xi_{D'_{\min}}$ and $\xi_{D''_{\min}}$, respectively. According to Equation \eqref{EQ:circ-i} the induced orientation of this composition with sign given by $\epsilon=i(s+1)+r\cdot 0$ is:
$$ \quad\quad\quad \xi_{D_{\min}}
=(-1)^{i(s+1)+r\cdot 0}\cdot \xi_{D'_{\min}}\wedge \xi_{D''_{\min}}\wedge e_0. $$ 
Therefore,
\begin{eqnarray*}
\quad\quad\quad \omega(S,D)&=&(\omega'\wedge \omega'')\rfloor \xi_{D_{\min}}=(-1)^{i(s+1)}(\omega'\wedge \omega'')\rfloor(\xi_{D'_{\min}}\wedge \xi_{D''_{\min}}\wedge e_0)\\
&=&(-1)^{i(s+1)+r\cdot q}(\omega'\rfloor \xi_{D'_{\min}})\wedge(\omega''\rfloor \xi_{D''_{\min}})\wedge e_0.
\end{eqnarray*}
\item Let $\sigma\neq id$; so that $S'$ is an inner product diagram and the composition at $e_0$ is at the first position of $S'$. Assume that $S''$ is a module diagram with $s'$ leaves to the left of the module branch, $s''$ leaves to the right of the module branch, and $s=s'+s''+1$; see Figure \ref{FIG:mod-inner}.
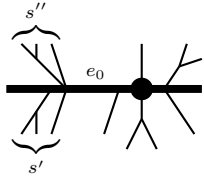
\begin{figure}
\begin{pspicture}(.2,1)(2.8,3) 
 \psline(1.5,1.4)(1.7,2)  \psline(2.7,1.4)(2.3,2)  \psline(2,1.6)(2,2.6)
 \psline(2,1.6)(1.8,1.2) \psline(2,1.6)(2.2,1.2)
 \psline(2.3,2)(2.5,2.3) \psline(2.5,2.3)(2.8,2.4) \psline(2.5,2.3)(2.6,2.6)
 \psline(1,2)(.8,2.6) \psline(.6,2.4)(.6,2.6) 
 \psline(1,2)(.4,2.6) \rput(1.4,2.2){$_{e_0}$}
 \rput(.6,2.9){$\overbrace{\quad}^{s''}$}  \rput(.6,1.1){$\underbrace{\quad}_{s'}$}
 \psline(1,2)(.8,1.4) \psline(.6,1.7)(.6,1.4) 
 \psline(.8,2)(.4,1.4)
 \psline[linewidth=3pt](0.2,2)(2.8,2)
 \pscircle*(2,2){.15}
\end{pspicture} 
\caption{The composition $S'\circ_{e_0}S''$.}\label{FIG:mod-inner}
\end{figure}
Then as before, we obtain a sign of $(-1)^{1\cdot (s+1)+r\cdot q}$ for the composition $S'\circ_{e_0}S''$ with orientation $(\omega'\rfloor \xi_{D'_{\min}})\wedge(\omega''\rfloor \xi_{D''_{\min}})\wedge e_0$, but the cyclic permutation introduces an additional sign of $(-1)^{s'(s''+r)}$. On the other hand, we may use Remark \ref{REM:xi-by-composition} to calculate $\xi_{D_{\min}}$, which is again given by the orientations $\xi_{D'_{\min}}$ and $\xi_{D''_{\min}}$ and the fact that $D_{\min}$ is given by the composition $D_{\min}=D'_{\min}\circ_{e_0}D''_{\min}$, now with an additional cyclic permutation $\sigma$. The sign for the composition $D_{\min}=D'_{\min}\circ_{e_0}D''_{\min}$ from Equation \eqref{EQ:circ-i} is given by $\epsilon=1\cdot (s+1)+r\cdot 0$, and the sign coming from the permutation $\sigma$ is $(-1)^{s' \cdot (s''+r)}$. This gives a total orientation $\xi_{D_{\min}}$ of
\[ \xi_{D_{\min}}=(-1)^{s' \cdot (s''+r)}\cdot (-1)^{1\cdot (s+1)+r\cdot 0} \cdot \xi_{D'_{\min}}\wedge \xi_{D''_{\min}}\wedge e_0. \]
%
%
Thus we obtain again the same sign as before:
\begin{eqnarray*}
\quad\quad\quad \omega(S,D)&=&(\omega'\wedge \omega'')\rfloor \xi_{D_{\min}} \\
&=&(-1)^{s+1+s'(s''+r)}(\omega'\wedge \omega'')\rfloor(\xi_{D'_{\min}}\wedge \xi_{D''_{\min}}\wedge e_0)\\
&=& (-1)^{1+s+r\cdot q+s'(s''+r)}(\omega'\rfloor\xi_{D'_{\min}})\wedge(\omega''\rfloor \xi_{D''_{\min}})\wedge e_0.
\end{eqnarray*}
\end{itemize}
For the general case $S_{\max}<D_{\min}$, note that any steps given by a local move that changes $\omega(S,D)$ correspond exactly to a local move that change $\omega(S',D')$ or $\omega(S'',D'')$ by the same sign.

\begin{figure*}[p]
$\begin{array}{c|c|c|c}
 D & (D/e_0)_{\min} & D'_{\min}\circ_{e_0} D''_{\min} & (D/e')_{\min} 
 \\ \hline \quad 
 \resizebox{!}{2.13cm}{\begin{pspicture}(.9,0)(4.2,4.2) 
 \psline(2.5,.2)(2.5,3.8)\psline(2.5,2.4)(3.6,3.8)
\psline(2.5,1)(2,1.6) \psline(2.5,1)(3,1.6) \psline(2.5,1)(1.4,1.6) \psline(2.5,1)(3.6,1.6)
\rput(1.7,2){$\overbrace{}^{>0}$} \rput(2.7,1.8){$e'$} \rput(2.2,3){$e_0$}
\psline(2.5,2.4)(3,3.8) \psline(2.5,2.4)(1.6,3.2) \psline(1.6,3.2)(1.6,3.8)
\psline(1.6,3.2)(1.1,3.8) \psline(1.6,3.2)(2.1,3.8)
 \end{pspicture}}
 \quad &\quad
 \resizebox{!}{2.13cm}{\begin{pspicture}(0.4,0.2)(4,3.8) 
 \psline(2.2,.4)(2.2,1.2) \psline(2.2,1.2)(.4,3) \psline(2,1.4)(3.6,3)
\psline(1.8,1.6)(3.2,3) \psline(2.2,2.4)(2.8,3) \psline(.6,2.8)(.8,3)
\rput(2.02,2.08){$e'$} \psline(2.4,2.2)(1.2,3.4) \rput(1.95,2.95){$e_0$} 
\psline(2,2.6)(2.8,3.4) \psline(1.6,3)(2,3.4) \psline(1.4,3.2)(1.6,3.4)
\psline(2.2,1.2)(4,3)
 \end{pspicture}} 
 \quad &\quad
 \resizebox{!}{2.13cm}{\begin{pspicture}(0.4,0.6)(4,4.4)
 \psline(2.2,.8)(2.2,1.2) \psline(2.2,1.2)(.4,3) \psline(2,1.4)(3.6,3)
\psline(1.8,1.6)(3.2,3)  \psline(.6,2.8)(.8,3) \psline(2.2,2.4)(2.8,3)
\rput(2.02,2.08){$e'$} \psline(2.4,2.2)(1.6,3) \rput(1.9,3.2){$e_0$} 
\psline(2,2.6)(2.4,3)  \psline[linestyle=dotted](1.6,3)(1.6,3.4) 
\psline(1.6,3.4)(1.6,3.6) \psline(1.6,3.6)(1.2,4) \psline(1.6,3.6)(2,4)
\psline(1.4,3.8)(1.6,4)\psline(2.2,1.2)(4,3)
 \end{pspicture}}
 \quad &
 \resizebox{!}{2.13cm}{\begin{pspicture}(0.2,0)(4,3.4) 
 \psline(2.2,.2)(2.2,1.2) \psline(2.2,1.2)(.4,3) \psline(2,1.4)(3.6,3)
\psline(1.8,1.6)(3.2,3) \psline(1.6,1.8)(2.8,3) \psline(.6,2.8)(.8,3)
\psline(1.4,2)(2.4,3) \psline(1.2,2.2)(2,3) \rput(1.27,2.53){$e_0$} 
\psline(1.6,2.6)(1.2,3) \psline(1.4,2.8)(1.6,3)\psline(2.2,1.2)(4,3)
 \end{pspicture}}
 \quad  \\ \hline 
 \resizebox{!}{2.13cm}{\begin{pspicture}(.5,0)(4,4) 
 \psline[linewidth=3pt](2.5,.2)(2.5,3.6)
\psline(2.5,1)(2.05,1.6)   \psline(2.5,1)(3,1.6) \psline(2.5,1)(1.3,1.6) 
\psline(2.5,1)(3.6,1.6) \psline(2.5,1)(2.25,1.6) \psline(2.5,1)(.9,1.6)
\rput(1.1,1.83){$\overbrace{}^{>0}$} \rput(1.8,2){$e'$} \rput(1,2.6){$e_0$}
\psline(2.5,1)(1.2,2.2) \psline(1.2,2.2)(1.2,3.6)
\psline(1.2,3)(.9,3.6) \psline(1.2,3)(1.5,3.6) 
\psline(1.2,2.2)(1.9,3.6) \psline(1.2,2.2)(2.3,3.6)
 \end{pspicture}}
 \quad &\quad
 \resizebox{!}{2.13cm}{\begin{pspicture}(0,0)(3.8,4.2)
\psline[linewidth=3pt](3,.4)(3,3.8) 
\rput(1.7,2.95){$e_0$} \rput(1.65,2.45){$e'$}
\psline(3,.8)(.4,3.4) \psline(3,1.2)(3.4,1.6) \psline(3,1.6)(3.4,2)
\psline(1.6,2.2)(2.8,3.4) \psline(1.8,2)(2.8,3) \psline(2,1.8)(2.8,2.6)
\psline(.6,3.2)(.8,3.4) \psline(1.9,2.5)(1,3.4) \psline(1.2,3.2)(1.4,3.4)
\psline(1.4,3)(1.8,3.4) \psline(1.7,2.7)(2.4,3.4) \end{pspicture}}
 \quad &\quad
 \resizebox{!}{2.13cm}{\begin{pspicture}(1,0)(3.8,4)
\psline[linewidth=3pt](3,.4)(3,2.6) 
\rput(2.2,2.7){$e_0$} \rput(2.15,1.95){$e'$}
\psline(3,.8)(1.4,2.4) \psline(3,1.2)(3.4,1.6) \psline(3,1.6)(3.4,2)
\psline(1.6,2.2)(1.8,2.4) \psline(2,2.4)(2.4,2) \psline(2.2,2.2)(2.4,2.4)
\psline(2.8,2.4)(2.1,1.7)  \psline(2.8,2)(2.3,1.5) \psline(2.8,1.6)(2.5,1.3)
\psline[linestyle=dotted](2,2.4)(2,3) \psline(2,3)(2,3.2)
\psline(2,3.2)(2.4,3.6) \psline(2,3.2)(1.6,3.6) \psline(1.8,3.4)(2,3.6) \end{pspicture}}
 \quad &
 \resizebox{!}{2.13cm}{\begin{pspicture}(0,0)(3.8,4.2)
\psline[linewidth=3pt](3,.4)(3,3.8) 
\rput(1.25,2.95){$e_0$} 
\psline(3,.8)(.4,3.4) \psline(3,1.2)(3.4,1.6) \psline(3,1.6)(3.4,2)
\psline(1.6,2.2)(2.8,3.4) \psline(1.8,2)(2.8,3) \psline(2,1.8)(2.8,2.6)
\psline(.6,3.2)(.8,3.4) \psline(1.6,3)(1.2,3.4) \psline(1.4,3.2)(1.6,3.4)
\psline(1.2,2.6)(2,3.4) \psline(1.4,2.4)(2.4,3.4)  \end{pspicture}}
 \quad  \\ \hline 
 \resizebox{!}{2.13cm}{\begin{pspicture}(0.5,0)(3.7,4)
 \psline[linewidth=3pt](1.5,.2)(1.5,3.6) \psline(1.5,1)(2.8,2.2) 
\psline(1.5,1)(.8,1.6) \psline(1.5,1)(1.2,1.6) \psline(1.5,1)(1,1.6) 
\psline(1.5,1)(1.8,1.6) \psline(1.5,1)(2,1.6) \psline(1.5,1)(2.4,1.6) 
\psline(1.5,1)(2.6,1.6) \psline(2.8,3)(2.5,3.6) \psline(2.8,3)(3.1,3.6)
\psline(2.8,2.2)(2.8,3.6) \psline(2.8,2.2)(3.1,2.8) \psline(2.8,2.2)(3.3,2.8)
\rput(2.3,2){$e'$} \rput(2.6,2.6){$e_0$}
  \end{pspicture}}
 \quad &\quad
 \resizebox{!}{2.13cm}{\begin{pspicture}(0,0)(3.6,4.2)
\psline[linewidth=3pt](1,.4)(1,3.8)
 \rput(2.2,3.05){$e_0$} \rput(1.7,2.3){$e'$}
 \psline(1,.8)(.2,1.6) \psline(.4,1.4)(.6,1.6) \psline(.5,1.3)(.8,1.6)
 \psline(1,1)(3.4,3.4) \psline(1,1.2)(3.2,3.4) \psline(1,1.4)(3,3.4)
\psline(1,3.2)(1.2,3.4) \psline(1,3)(1.4,3.4)  \psline(2.3,2.7)(1.6,3.4) 
 \psline(1.8,3.2)(2,3.4) \psline(1.9,3.1)(2.2,3.4) \psline(2.2,2.8)(2.8,3.4) \end{pspicture}}
 \quad &\quad
 \resizebox{!}{2.13cm}{\begin{pspicture}(0,0)(2.8,3.9)
\psline[linewidth=3pt](1,.4)(1,3)
 \rput(1.8,2.8){$e_0$} \rput(1.4,2){$e'$} \psline[linestyle=dotted](1.6,2.6)(1.6,3)
 \psline(1,.8)(.2,1.6) \psline(.4,1.4)(.6,1.6) \psline(.5,1.3)(.8,1.6)
 \psline(1,1)(2.6,2.6) \psline(1,1.2)(2.4,2.6) \psline(1,1.4)(2.2,2.6)
\psline(1,2.4)(1.2,2.6) \psline(1,2.2)(1.4,2.6)
 \psline(1.6,2.6)(1.9,2.3) \psline(1.8,2.4)(2,2.6) \psline(1.5,3.3)(1.7,3.5)
 \psline(1.6,3)(1.6,3.2) \psline(1.6,3.2)(1.9,3.5) \psline(1.6,3.2)(1.3,3.5)  \end{pspicture}}
 \quad &
 \resizebox{!}{2.13cm}{\begin{pspicture}(-.5,0)(3.6,4.2)
\psline[linewidth=3pt](1,.4)(1,3.8)
 \rput(1.5,2.6){$e_0$} 
 \psline(1,.8)(.2,1.6) \psline(.4,1.4)(.6,1.6) \psline(.5,1.3)(.8,1.6)
 \psline(1,1)(3.4,3.4) \psline(1,1.2)(3.2,3.4) \psline(1,1.4)(3,3.4)
 \psline(1,1.6)(2.8,3.4) \psline(1,1.8)(2.6,3.4)
\psline(1,3.2)(1.2,3.4) \psline(1,3)(1.4,3.4)  \psline(2.1,2.9)(1.6,3.4) 
 \psline(1.9,3.1)(2.2,3.4) 
 \end{pspicture}}
 \quad  \\ \hline 
 \resizebox{!}{2.13cm}{\begin{pspicture}(0,0.2)(4,4) 
 \psline[linewidth=3pt](.4,1.4)(3.6,1.4) \rput(2.2,2.4){$e'$} \rput(1.6,2.7){$e_0$}
\pscircle*(2,1.4){.15} \psline(2,2.6)(2,.6) \psline(2.3,2.2)(1.7,.6) 
\psline(1.7,2.2)(2.3,.6) \psline(2,1.4)(1.5,2.2) \psline(2,1.4)(2.5,2.2) 
\psline(2,2.6)(2.1,3.2) \psline(2,2.6)(2.4,3.2) \psline(2,2.6)(1.5,3.1)
\psline(1.5,3.1)(1.5,3.6) \psline(1.5,3.1)(1.2,3.6) \psline(1.5,3.1)(1.8,3.6)
 \end{pspicture}}
 \quad &\quad
 \resizebox{!}{2.13cm}{\begin{pspicture}(-.6,0)(2.4,3.8)
\psline[linewidth=3pt](-.6,1.4)(2.4,1.4) \pscircle*(1.3,1.4){.15} 
 \rput(.35,1.87){$e'$} \rput(-.03,2.7){$e_0$}
 \psline(.1,2)(.7,3.2) \psline(0,2.2)(.5,3.2)
\psline(.4,1.4)(-.5,3.2) \psline(.8,1.4)(.4,2.2) \psline(1,1.4)(.6,2.2)
\psline(1.6,1.4)(2,.6) \psline(1.8,1.4)(2.2,.6) \psline(2,1.4)(2.4,.6)
\psline(0,1.4)(-.4,2.2) \psline(-.2,1.4)(-.6,2.2) 
\psline(-.3,2.8)(-.1,3.2) \psline(-.4,3)(-.3,3.2) \end{pspicture}}
 \quad &\quad
 \resizebox{!}{2.13cm}{\begin{pspicture}(-.6,0.2)(2.4,4)
\psline[linewidth=3pt](-.6,1.4)(2.4,1.4) \pscircle*(1.3,1.4){.15} 
 \rput(.35,1.87){$e'$} \rput(-.4,2.8){$e_0$}
 \psline(.1,2)(.4,2.6) \psline(0,2.2)(.2,2.6)
\psline(.4,1.4)(-.2,2.6) \psline(.8,1.4)(.4,2.2) \psline(1,1.4)(.6,2.2)
\psline(1.6,1.4)(2,.6) \psline(1.8,1.4)(2.2,.6) \psline(2,1.4)(2.4,.6)
\psline(0,1.4)(-.4,2.2) \psline(-.2,1.4)(-.6,2.2) 
\psline[linestyle=dotted](-.2,2.6)(-.2,3) \psline(-.2,3.2)(-.2,3)
\psline(-.2,3.2)(0,3.6) \psline(-.3,3.4)(-.2,3.6) \psline(-.2,3.2)(-.4,3.6) \end{pspicture}}
 \quad &
 \resizebox{!}{2.13cm}{\begin{pspicture}(-.8,0)(2.8,3.6)
\psline[linewidth=3pt](-.6,1.4)(2.8,1.4) \pscircle*(1.7,1.4){.15} 
 \rput(.4,1.83){$e_0$} \psline(.1,2)(.4,2.6) \psline(0,2.2)(.2,2.6)
\psline(.4,1.4)(-.2,2.6) \psline(.8,1.4)(.4,2.2) \psline(1,1.4)(.6,2.2)
\psline(1.2,1.4)(.8,2.2) \psline(2.2,1.4)(2.6,.6) \psline(2,1.4)(2.4,.6)
\psline(1.4,1.4)(1,2.2) \psline(2.4,1.4)(2.8,.6) 
\psline(0,1.4)(-.4,2.2) \psline(-.2,1.4)(-.6,2.2)   \end{pspicture}}
 \quad  \\ \hline 
 \resizebox{!}{2.13cm}{\begin{pspicture}(1,0)(4,4.2)
 \psline[linewidth=3pt](2.5,.2)(2.5,3.8)
\psline(2.5,1)(1.9,1.6) \psline(2.5,1)(3,1.6) \psline(2.5,1)(1.5,1.6) \psline(2.5,1)(3.6,1.6)
\rput(1.7,2){$\overbrace{}^{>0}$} \rput(2.7,1.8){$e'$} \rput(2,3.05){$e_0$}
\psline(2.5,2.4)(1.6,3.2) \psline(2.5,2.4)(2.2,3.6) \psline(2.5,2.4)(2,3.6)
 \psline(2.5,2.4)(2.8,3.6) \psline(2.5,2.4)(3.2,3.6) \psline(2.5,2.4)(3,3.6)
\psline(1.6,3.2)(1.6,3.8) \psline(1.6,3.2)(1.4,3.8) \psline(1.6,3.2)(1.8,3.8)
  \end{pspicture}}
 \quad &\quad
 \resizebox{!}{2.13cm}{\begin{pspicture}(1,0)(3.8,4.4)
\psline[linewidth=3pt](3,.4)(3,4) 
\rput(2.1,3.19){$e_0$}  \rput(2.8,1.85){$e'$}
\psline(3,.8)(2.4,1.4) \psline(3,1.2)(3.4,1.6) \psline(3,1.6)(3.4,2)
\psline(2.6,1.2)(2.8,1.4) \psline(3,2.4)(3.4,2.8) \psline(3,2.8)(3.4,3.2)
\psline(3,3.2)(3.4,3.6) \psline(2.1,2.9)(2.8,3.6) \psline(2.3,2.7)(2.8,3.2)
\psline(3,2)(1.4,3.6) \psline(1.6,3.4)(1.8,3.6) \psline(1.8,3.2)(2.2,3.6)  \end{pspicture}}
 \quad &\quad
 \resizebox{!}{2.13cm}{\begin{pspicture}(1.2,0)(3.8,4.6) 
\psline[linewidth=3pt](3,.4)(3,3.8) 
\rput(1.8,3.3){$e_0$}  \rput(2.8,1.85){$e'$}
\psline(2.2,2.8)(2.4,3) \psline(2.4,2.6)(2.8,3) \psline(3,2)(2,3) 
\psline(3,.8)(2.4,1.4) \psline(3,1.2)(3.4,1.6) \psline(3,1.6)(3.4,2)
\psline(2.6,1.2)(2.8,1.4) \psline(3,2.4)(3.4,2.8) \psline(3,2.8)(3.4,3.2)
\psline(3,3.2)(3.4,3.6)  \psline[linestyle=dotted](2,3)(2,3.6)
\psline(2,3.6)(2,3.8) \psline(2,3.8)(2.4,4.2) \psline(2,3.8)(1.6,4.2)
\psline(1.8,4)(2,4.2) \end{pspicture}}
 \quad &
 \resizebox{!}{2.13cm}{\begin{pspicture}(0,0)(3.8,4.2)
\psline[linewidth=3pt](3,.4)(3,3.8) \rput(1.25,2.95){$e_0$} 
\psline(3,2.2)(3.4,2.6) \psline(3,2.6)(3.4,3) \psline(3,3)(3.4,3.4)
\psline(3,.8)(.4,3.4) \psline(3,1.2)(3.4,1.6) \psline(3,1.6)(3.4,2)
\psline(.6,3.2)(.8,3.4) \psline(1.6,3)(1.2,3.4) \psline(1.4,3.2)(1.6,3.4)
\psline(1.6,2.2)(2.8,3.4) \psline(1.2,2.6)(2,3.4) \psline(1.4,2.4)(2.4,3.4)   \end{pspicture}}
 \quad  \\ \hline 
 \resizebox{!}{2.13cm}{\begin{pspicture}(0.8,0)(4.8,3.6)
\psline[linewidth=3pt](1.2,1.4)(4.4,1.4) 
\rput(2.7,1.7){$e'$} \rput(3.2,2.3){$e_0$} \pscircle*(2,1.4){.15} 
\psline(2,2.2)(2,.6) \psline(2.3,2.2)(1.7,.6)  \psline(1.7,2.2)(2.3,.6) 
\psline(3.2,1.4)(3.4,.6) \psline(3.2,1.4)(3.7,.6)
\psline(3.2,1.4)(3.6,2.2) \psline(3.2,1.4)(3.8,2.2)
\psline(3.2,1.4)(3.6,3) \psline(3.5,2.6)(3.5,3) \psline(3.5,2.6)(3.7,3)
  \end{pspicture}}
 \quad &\quad
 \resizebox{!}{2.13cm}{\begin{pspicture}(0.2,0)(3.4,3.6)
\psline[linewidth=3pt](.2,1.4)(3.4,1.4) \pscircle*(1.3,1.4){.15} 
\rput(2.35,1.67){$e'$} \rput(2.55,2.35){$e_0$}
\psline(.6,1.4)(.2,2.2) \psline(.8,1.4)(.4,2.2) \psline(1,1.4)(.6,2.2)
\psline(1.6,1.4)(2,.6) \psline(1.8,1.4)(2.2,.6) \psline(2,1.4)(2.4,.6)
\psline(3,1.4)(3.4,.6) \psline(2.8,1.4)(3.2,.6) \psline(2.6,1.4)(3.4,3)
 \psline(2.9,2)(2.4,3) \psline(2.8,2.2)(3.2,3) \psline(2.6,2.6)(2.8,3) \psline(2.5,2.8)(2.6,3)  \end{pspicture}}
 \quad &\quad
 \resizebox{!}{2.13cm}{\begin{pspicture}(0.2,0.1)(3.4,3.7)
\psline[linewidth=3pt](.2,1.4)(3.4,1.4) \pscircle*(1.3,1.4){.15} 
\rput(2.35,1.67){$e'$} \rput(2.4,2.35){$e_0$}  \psline(2.8,2.2)(2.7,2)
\psline[linestyle=dotted](2.6,2.2)(2.6,2.6) \psline(2.6,2.6)(2.6,2.8)
\psline(.6,1.4)(.2,2.2) \psline(.8,1.4)(.4,2.2) \psline(1,1.4)(.6,2.2)
\psline(1.6,1.4)(2,.6) \psline(1.8,1.4)(2.2,.6) \psline(2,1.4)(2.4,.6)
\psline(3,1.4)(3.4,.6) \psline(2.8,1.4)(3.2,.6) \psline(2.6,1.4)(3,2.2)
\psline(2.6,2.8)(2.4,3.2) \psline(2.8,1.8)(2.6,2.2) \psline(2.6,2.8)(2.8,3.2)
\psline(2.5,3)(2.6,3.2)  \end{pspicture}}
 \quad &
 \resizebox{!}{2.13cm}{\begin{pspicture}(-.8,0)(2.6,3.6)
\psline[linewidth=3pt](-.8,1.4)(3,1.4) \pscircle*(1.3,1.4){.15} 
 \rput(.4,1.83){$e_0$} \psline(.1,2)(.4,2.6) \psline(0,2.2)(.2,2.6)
\psline(.4,1.4)(-.2,2.6) \psline(.8,1.4)(.4,2.2) \psline(1,1.4)(.6,2.2)
\psline(1.6,1.4)(2,.6) \psline(1.8,1.4)(2.2,.6) \psline(2,1.4)(2.4,.6)
\psline(2.6,1.4)(3,.6) \psline(2.4,1.4)(2.8,.6) 
\psline(0,1.4)(-.4,2.2) \psline(-.2,1.4)(-.6,2.2) \psline(-.4,1.4)(-.8,2.2) \end{pspicture}}
 \quad  \\ \hline 
 \resizebox{!}{2.13cm}{\begin{pspicture}(0.9,0)(3.1,4.2) 
 \psline[linewidth=3pt](2,.2)(2,3.6) \rput(1.8,2.5){$e'$} \rput(1.8,1.8){$e_0$}
\psline(2,1)(1.7,1.6) \psline(2,1)(1.5,1.6) \psline(2,1)(2.3,1.6) 
\psline(2,1)(2.5,1.6) \psline(2,2)(2.3,2.6) \psline(2,2)(2.5,2.6)
\psline(2,2.8)(1.7,3.4) \psline(2,2.8)(1.5,3.4) \psline(2,2.8)(2.3,3.4)
\psline(2,2.8)(2.5,3.4) \rput(1.6,3.7){$\overbrace{}^{>0}$} 
 \end{pspicture}}
 \quad &\quad
 \resizebox{!}{2.13cm}{\begin{pspicture}(0.8,0.2)(3.2,4) 
 \psline[linewidth=3pt](2,.4)(2,3.6) \psline(2,.8)(1.4,1.4) \psline(1.6,1.2)(1.8,1.4)
\psline(2,1)(2.4,1.4) \psline(2,1.2)(2.4,1.6) 
\psline(2,1.95)(2.4,2.35) \psline(2,2.15)(2.4,2.55)
\psline(2,2.6)(1.4,3.2) \psline(2,2.8)(2.4,3.2)  \psline(2,3)(2.4,3.4)
\psline(1.6,3)(1.8,3.2) \rput(2.26,1.8){$e_0$} \rput(1.8,2.4){$e'$}
 \end{pspicture}}
 \quad &\quad
 \resizebox{!}{2.13cm}{\begin{pspicture}(0.8,0.2)(3.2,4) 
 \psline[linewidth=3pt](2,.4)(2,1.6) \psline[linewidth=3pt](2,2)(2,3.6)
\psline[linestyle=dotted](2,1.6)(2,2) \rput(1.8,2.5){$e'$}
\psline(2,.8)(1.4,1.4) \psline(1.6,1.2)(1.8,1.4)
\psline(2,1)(2.4,1.4) \psline(2,1.2)(2.4,1.6) 
\psline(2,2.4)(2.4,2.8) \psline(2,2.2)(2.4,2.6) \psline(2,2.6)(1.4,3.2) 
\psline(2,2.8)(2.4,3.2)  \psline(2,3)(2.4,3.4)
\psline(1.6,3)(1.8,3.2) \rput(1.74,1.8){$e_0$} 
  \end{pspicture}}
 \quad &
 \resizebox{!}{2.13cm}{\begin{pspicture}(0.8,0.2)(3.2,4) 
 \psline[linewidth=3pt](2,.4)(2,3.6) \psline(2,.8)(1.4,1.4) \psline(1.6,1.2)(1.8,1.4)
\psline(2,1)(2.4,1.4) \psline(2,1.2)(2.4,1.6) 
\psline(2,2.4)(2.4,2.8) \psline(2,2.2)(2.4,2.6) \psline(2,2)(1.4,2.6) 
\psline(2,2.8)(2.4,3.2)  \psline(2,3)(2.4,3.4)
\psline(1.6,2.4)(1.8,2.6) \rput(1.74,1.8){$e_0$} 
 \end{pspicture}}
 \quad  \\ \hline 
 \resizebox{!}{2.13cm}{\begin{pspicture}(0.8,0)(5,3.6) 
\psline[linewidth=3pt](1.2,1.4)(4.6,1.4)  \rput(4.3,2.5){$\overbrace{}^{>0}$} 
\rput(3.6,1.67){$e'$} \rput(2.6,1.6){$e_0$} \pscircle*(2,1.4){.15} 
\psline(2,2.2)(2,.6) \psline(2.3,2.2)(1.7,.6)  \psline(1.7,2.2)(2.3,.6) 
\psline(3,1.4)(3.2,.6) \psline(3,1.4)(3.5,.6) \psline(4,1.4)(4.2,.6)
\psline(4,1.4)(4.5,.6) \psline(4,1.4)(4.2,2.2) \psline(4,1.4)(4.5,2.2)  \end{pspicture}}
 \quad &\quad
 \resizebox{!}{2.13cm}{\begin{pspicture}(0.3,0)(3.9,3.6)
\psline[linewidth=3pt](.2,1.4)(4,1.4) \pscircle*(1.3,1.4){.15} 
\rput(3.1,1.67){$e'$} \rput(2.4,1.2){$e_0$} 
\psline(.6,1.4)(.2,2.2) \psline(.8,1.4)(.4,2.2) \psline(1,1.4)(.6,2.2)
\psline(1.6,1.4)(2,.6) \psline(1.8,1.4)(2.2,.6) \psline(2,1.4)(2.4,.6)
\psline(2.6,1.4)(3,.6) \psline(2.8,1.4)(3.2,.6) \psline(3.2,1.4)(3.8,2.6)
\psline(3.4,1.4)(3.8,.6) \psline(3.6,1.4)(4,.6) \psline(3.5,2)(3.2,2.6)  \end{pspicture}}
 \quad &\quad
 \resizebox{!}{2.13cm}{\begin{pspicture}(0.3,0)(3.9,3.6)
 \psline[linewidth=3pt](0,1.4)(2.2,1.4) \psline[linewidth=3pt](2.6,1.4)(4.2,1.4)
\pscircle*(1.1,1.4){.15} \psline[linestyle=dotted](2.2,1.4)(2.6,1.4)
\rput(3.3,1.67){$e'$} \rput(2.4,1.2){$e_0$} 
\psline(.4,1.4)(0,2.2) \psline(.6,1.4)(.2,2.2) \psline(.8,1.4)(.4,2.2)
\psline(1.4,1.4)(1.8,.6) \psline(1.6,1.4)(2,.6) \psline(1.8,1.4)(2.2,.6)
\psline(2.8,1.4)(3.2,.6) \psline(3,1.4)(3.4,.6) \psline(3.4,1.4)(4,2.6)
\psline(3.6,1.4)(4,.6) \psline(3.8,1.4)(4.2,.6) \psline(3.7,2)(3.4,2.6)  \end{pspicture}}
 \quad &
 \resizebox{!}{2.13cm}{\begin{pspicture}(0,0)(3.8,3.6)
\psline[linewidth=3pt](.2,1.4)(3.8,1.4) \pscircle*(1.3,1.4){.15}  \rput(2.35,1.6){$e_0$} 
\psline(.6,1.4)(.2,2.2) \psline(.8,1.4)(.4,2.2) \psline(1,1.4)(.6,2.2)
\psline(1.6,1.4)(2,.6) \psline(1.8,1.4)(2.2,.6) \psline(2,1.4)(2.4,.6)
\psline(3,1.4)(3.4,.6) \psline(2.8,1.4)(3.2,.6) \psline(2.6,1.4)(3.2,2.6)
\psline(3.4,1.4)(3.8,.6) \psline(3.2,1.4)(3.6,.6) \psline(2.9,2)(2.6,2.6)   \end{pspicture}}
 \quad 
\end{array} $
\caption{Case $2$: $p$ is non-vanishing for $D/e_0$, $D_{e_0}$, and $D/e'$. Each of these diagrams may be completed in an arbitrary way. (To improve readability, the circle that were drawn in the previous diagrams have been omitted here.)}\label{Case2:e'}
\end{figure*}
{\bf Case 2:} Considering those diagrams depicted in Figure \ref{5cases} such that $e_0$ is the only negative edge, Figure \ref{Case2:e'} shows the situations in which Case $2$ occurs. (Here, the cases that are symmetric with respect to a $180^\circ$ symmetry, such as cases (e) and (f) in Figure \ref{5cases}, have been excluded.) %
As in Case $1$ above, $p(D/e_0,\dots)=\sum_{S_{\max}\leq D_{\min}} (S,\dots)$. This sum contains all the terms that appear in 
\begin{multline*}
p(D_{e_0},\dots)=p(D'\circ_{e_0} D'',\dots)=\sigma\cdot p(D',\dots,)\circ_{e_0}p(D'',\dots)\\
=\sum_{S'_{\max}\leq D'_{\min}, S''_{\max}\leq D''_{\min}} (S',\dots)\circ_{e_0} (S'',\dots),
\end{multline*}
where the signs can be checked as in Case 1. However, $p(D/e_0,\dots)$ can now have more terms. In fact, a direct inspection shows that the binary diagrams $S_{\max}\leq(D/e_0)_{\min}$ from Figure \ref{Case2:e'} are of two types: (1) those with terms using no local moves from Definition \ref{DEF:local-move} for $e_0$, and (2) those with terms that use local moves. The first type is given by $p(D_{e_0},\dots)$; the second type is given by $p(D/e',\dots)$. 

It only remains to check that the induced signs for $p(D/e_0,\dots)$ and $p(D/e',\dots)$ cancel. We first assume that $S_{\max}=(D/e')_{\min}$. Denote the orientation for $D$ by $\omega_D=\omega_R\wedge e_0\wedge e'$, where $\omega_R$ stands for the orientation of the remaining edges in $D$. With this, the two terms $D/e_0$ and $D/e'$ appear with opposite signs,
\[
\partial_Q(D,\fcan,m,\omega_D)=\pm\Big((D/e_0,\fcan,m,\omega_{D/e_0})-(D/e',\fcan,m,\omega_{D/e'})\Big)+\dots,
\]
where $\omega_{D/e_0}=\omega_R\wedge e'$ and $\omega_{D/e'}=\omega_R\wedge e_0$.
We claim that $p(D/e_0,\fcan,m,\omega_{D/e_0})=p(D/e',\fcan,m,\omega_{D/e'})$. By assumption $S_{\max}=(D/e')_{\min}$; let us calculate the induced orientations $\omega(S,D/e_0)$ and $\omega(S,D/e')$ and show they coincide.

\noindent
Note, that in each of the cases from Figure \ref{Case2:e'}, there is a unique path of local moves from $(D/e_0)_{\min}=D_{\min}$ to $(D/e')_{\min}$ that preserves the number of positive edges. This path is determined by a sequence of edges $e_1,\dots,e_j$ to which we apply the local moves. For example, in the first row in Figure \ref{Case2:e'}, this sequence is depicted as $e_1,e_2$ on the left of Figure \ref{FIG:orient-move-C2}.
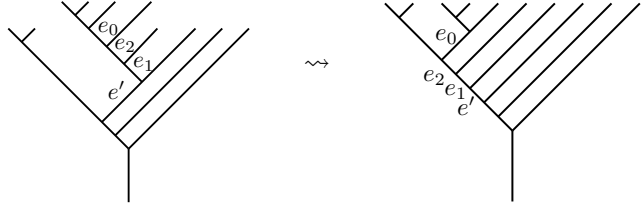
\begin{figure}[ht]
\[
 \resizebox{!}{3.2cm}{\begin{pspicture}(0.4,0.2)(4,3.8)
 \psline(2.2,.4)(2.2,1.2) \psline(2.2,1.2)(.4,3) \psline(2,1.4)(3.6,3)
\psline(1.8,1.6)(3.2,3) \psline(2.13,2.47)(2.63,3) \psline(.6,2.8)(.8,3)
\rput(2.02,2.08){$e'$} \psline(2.4,2.2)(1.2,3.4) \rput(1.9,3){$e_0$} 
\psline(1.87,2.73)(2.54,3.4) \psline(1.6,3)(2,3.4) \psline(1.4,3.2)(1.6,3.4)
\psline(2.2,1.2)(4,3) \rput(2.43,2.47){$e_1$} \rput(2.15,2.75){$e_2$}
 \end{pspicture}} 
 \quad \quad\twig{-2} \quad \quad
  \resizebox{!}{3.2cm}{\begin{pspicture}(0.4,0)(4,3.4) 
 \psline(2.2,.2)(2.2,1.2) \psline(2.2,1.2)(.4,3) \psline(2,1.4)(3.6,3)
\psline(1.8,1.6)(3.2,3) \psline(1.6,1.8)(2.8,3) \psline(.6,2.8)(.8,3)
\psline(1.4,2)(2.4,3) \psline(1.2,2.2)(2,3) \rput(1.27,2.53){$e_0$} 
\psline(1.6,2.6)(1.2,3) \psline(1.4,2.8)(1.6,3)\psline(2.2,1.2)(4,3)
\rput(1.1,1.95){$e_2$} \rput(1.4,1.75){$e_1$} \rput(1.55,1.55){$e'$}
 \end{pspicture}}
\]
\caption{Local moves change $(D/e_0)_{\min}$ to $(D/e')_{\min}$.}\label{FIG:orient-move-C2}
\end{figure}
Now, write $\xi_{D_{\min}}$ in the form $\xi_{D_{\min}}=e'\wedge e_1\wedge\dots\wedge e_j\wedge e_0\wedge \xi_R$ where $\xi_R$ is the orientation on the remaining edges. In order to calculate $\omega(S,D/e_0)$ we need to perform a sequence of moves on $e',e_1,\dots,e_j$ as depicted in Figure \ref{FIG:orient-move-C2}. At the $i\T$ local move, the $i\T$ orientation $\xi_i$ changes by a factor of $-1$, so that $\xi_i=(-1)^i \cdot e'\wedge e_1\wedge\dots\wedge e_j\wedge e_0\wedge \xi_R$. Keeping track of the positive edge, while making local moves, tells us that $e_i$ is positive after the $i\T$ move. After $j+1$ steps, we obtain the orientation $\xi_{j+1}=(-1)^{j+1} \cdot e'\wedge e_1\wedge\dots\wedge e_j\wedge e_0\wedge \xi_R$ which applies to the diagram $S_{\max}=(D/e')_{\min}$, as depicted on the right of Figure \ref{FIG:orient-move-C2} with positive edge $e_0$ so that the orientation induced from $\omega_{D/e_0}$ is now $\omega_R\wedge e_0$. With this, we calculate $\omega(S,D/e_0)=(\omega_R\wedge e_0)\rfloor\xi_{j+1}=\omega_R\rfloor (e'\wedge e_1\wedge\dots\wedge e_j\wedge \xi_R)$.

\noindent
On the other hand, we see from Remark \ref{REM:xi-by-composition}, that $\xi_{(D/e')_{\min}}=(-1)^{j+1}\cdot e'\wedge e_1\wedge\dots\wedge e_j\wedge e_0\wedge \xi_R$. Since $\omega_{D/e'}=\omega_R\wedge e_0$, we obtain $\omega(S,D/e')=(\omega_R\wedge e_0)\rfloor \xi_{(D/e')_{\min}}=\omega_R\rfloor (e'\wedge e_1\wedge\dots\wedge e_j\wedge \xi_R)$, which coincides with $\omega(S,D/e_0)$.

The cases where $S_{\max}<(D/e')_{\min}$ apply the same local moves and change the signs in $p(D/e_0,\dots)$ and $p(D/e',\dots)$ the same way, establishing equality of $\omega(S,D/e_0)=\omega(S,D/e')$. 

This completes the sign check for Case 2, and with this the proof of the lemma.
\end{proof}

\begin{lemma}\label{B_max-lemma}
Let $(B,\fcan,m,\ost_{B})\in \Qa{n}$ be a maximal binary diagram, \emph{i.e.} one of the following three diagrams,
\[
 \begin{pspicture}(0.5,.5)(4,2)
 \rput(1,1.25){$(T_n)_{\max}=$} \rput(3.6,1.2){,}
 \psline(3,.5)(3,1) \psline(3,1)(2,2) \psline(3,1)(4,2)
 \psline(3.8,1.8)(3.6,2)  \psline(3.6,1.6)(3.2,2)  \psline(3.4,1.4)(2.8,2)  \psline(3.2,1.2)(2.4,2)
 \end{pspicture} 
 \quad
 \begin{pspicture}(0,.5)(4,2)
 \rput(1,1.25){$(M_{k,l})_{\max}=$} \rput(3.6,1.2){,}
 \psline[linewidth=3pt](3,.5)(3,2) \psline(3,1)(4,2)
 \psline(3,1.2)(2.2,2) \psline(3,1.4)(2.4,2) \psline(3,1.6)(2.6,2) \psline(3,1.8)(2.8,2)
 \psline(3.2,1.2)(3.2,2) \psline(3.4,1.4)(3.4,2) \psline(3.6,1.6)(3.6,2) \psline(3.8,1.8)(3.8,2)
 \end{pspicture} 
 \quad
 \begin{pspicture}(0,.5)(3.8,2)
 \rput(1,1.25){$(I_{k,l})_{\max}=$}
 \psline[linewidth=3pt](2,1.25)(4,1.25) \pscircle*(3,1.25){.15}
 \psline(3.2,1.25)(3.4,2) \psline(3.4,1.25)(3.6,2) \psline(3.6,1.25)(3.8,2) \psline(3.8,1.25)(4,2)
 \psline(2.8,1.25)(2.6,.5) \psline(2.6,1.25)(2.4,.5) \psline(2.4,1.25)(2.2,.5) \psline(2.2,1.25)(2,.5) \rput(4.2,1.2){.}
 \end{pspicture}
\]
Then, $\partial_C \circ p(B,\fcan,m,\ost_{B})=p\circ \partial_Q (B,\fcan,m,\ost_{B})$.
\end{lemma}
\begin{proof}
We show $\partial_C \circ p=p\circ \partial_Q$ when applied to $B=(T_n)_{\max}$, $B=(M_{k,l})_{\max}$, and $B=(I_{k,l})_{\max}$ by direct computation. Note that for any corolla $c$, Lemma \ref{p-values}\emph{(ii)} gives
\[
\partial_C(p(c_{\max},\fcan,m,\ost_B))=\partial_C(c,\fcan,+1)= \sum_{D/e=c} (D,\fcan,+e).
\]
We need to evaluate $p(\partial_Q(c_{\max},\fcan,m,\ost_B))$. To simplify notation, we will often ignore the label $\fcan$ and express the quadruple as a diagram together with its orientation.

The first case $B=(T_n)_{\max}$ appears in \cite{MS} (see \cite[Figures 6 and 7]{MS}),
\[\xymatrix{
\resizebox{!}{1.3cm}{
 \begin{pspicture}(2,.5)(4.4,2)
 \rput(3.3,1){$_{e_1}$} \rput(3.6,1.4){$_{\dots}$} \rput(4.1,1.6){$_{e_{n-2}}$}
 \psline(3,.5)(3,1) \psline(3,1)(2,2) \psline(3,1)(4,2)
 \psline(3.8,1.8)(3.6,2)  \psline(3.6,1.6)(3.2,2)  \psline(3.4,1.4)(2.8,2)  \psline(3.2,1.2)(2.4,2)
 \end{pspicture} }
\ar[r]^p \ar[d]_{\partial_Q} & 
\resizebox{!}{1.3cm}{
\begin{pspicture}(1,1)(3,3)
 \psline(2,1)(2,2)
 \psline(2,2)(1.25,3)
 \psline(2,2)(1.55,3)
 \psline(2,2)(1.85,3)
 \psline(2,2)(2.15,3)
 \psline(2,2)(2.45,3)
 \psline(2,2)(2.75,3)
\end{pspicture}}
 \ar[d]^{\partial_C} \\
\partial_Q((\A_n)_{\max}) \ar[r]^{p\quad} & 
\sum_{r\geq 2}\sum_{s\geq 0} \A_n^{r,s},
}\]
where $\ost_{(T_n)_{\max}}=e_1\wedge\dots\wedge e_{n-2}$ and
\begin{eqnarray*}
\begin{pspicture}(.3,0)(1,1) 
\rput(.65,1){$\A_n^{r,s}$} \end{pspicture} &\begin{pspicture}(0,0)(.3,1) \rput(.15,1){$=$} \end{pspicture}& 
\resizebox{!}{2.3cm}{
\begin{pspicture}(1.2,0)(6.4,2.2)
 \psline(2,0)(2,1.7) \psline(2,1.2)(1.8,1.7) \psline(2,1.2)(2.2,1.7)
 \psline(2,.5)(1.2,1.7)  \psline(2,.5)(1.5,1.7)
 \psline(2,.5)(2.5,1.7)  \psline(2,.5)(2.8,1.7) \psline(2,.5)(3.1,1.7)
\rput(2.8,2){$\overbrace{\quad\quad }^{s}$}
\rput(2,2){$\overbrace{\quad}^{r}$}
\rput(5,1){with orientation $e$; and}
\end{pspicture} } 
\\
\partial_Q((\A_n)_{\max}) &=& \sum_{\text{edge }e_j} (-1)^j \left(
 \resizebox{!}{.8cm}{
 \begin{pspicture}(2,1.1)(4,2)
 \psline(3,.5)(3,1) \psline(3,1)(2,2) \psline(3,1)(4,2)
 \psline(3.8,1.8)(3.6,2)  \psline(3.5,1.5)(3.2,2)  \psline(3.5,1.5)(2.8,2)  \psline(3.2,1.2)(2.4,2)
 \end{pspicture} } \omega_j -
\resizebox{!}{.8cm}{
 \begin{pspicture}(2,1.1)(4,2)
 \psline(3,.5)(3,1) \psline(3,1)(2,2)  \psline(3,1)(3.4,1.4) \psline(3.6,1.6)(4,2)
 \psline[linewidth=3pt](3.4,1.4)(3.6,1.6) \rput(3.7,1.3){$n$}
 \psline(3.8,1.8)(3.6,2)  \psline(3.6,1.6)(3.2,2)  \psline(3.4,1.4)(2.8,2)  \psline(3.2,1.2)(2.4,2)
 \end{pspicture} } \omega_j
\right),
\end{eqnarray*}
where the edge $e_j$ is collapsed, so that $\omega_j=e_1\wedge\dots\wedge\widehat{e_j}\wedge\dots\wedge e_{n-2}$.
Here, the thick edge labeled ``$\,n$'' denotes a non-metric edge. Note that
 \begin{equation*}
 p\left(\sum_j (-1)^j
 \resizebox{!}{.8cm}{
 \begin{pspicture}(2,1.1)(4,2)
 \psline(3,.5)(3,1) \psline(3,1)(2,2) \psline(3,1)(4,2)
 \psline(3.8,1.8)(3.6,2)  \psline(3.5,1.5)(3.2,2)  \psline(3.5,1.5)(2.8,2)  \psline(3.2,1.2)(2.4,2)
 \end{pspicture} } \omega_j
 \right) = \sum_{r}\sum_{s>0} \A_n^{r,s}.
 \end{equation*}
The sign calculation for $B/e_j$ is the same for all three cases $B=(T_n)_{\max},(M_{k,l})_{\max}$ and $(I_{k,l})_{\max}$. Since we make one edge move from $B$ to $(B/e_j)_{\min}$, we have:
\begin{eqnarray*}
\xi_{(B/e_j)_{\min}}&=&(-1)^1\cdot (-1)^{1+2+\dots+(n-3)}e_1\wedge\dots\wedge e_j\wedge\dots\wedge e_{n-2}\\
&=&(-1)^{(1+\dots+(n-3))+(n-2)-j+1} e_1\wedge\dots\wedge\widehat{e_j}\wedge\dots\wedge e_{n-2}\wedge e_j\\
&=&(-1)^{(1+\dots+(n-4))-j} e_1\wedge\dots\wedge\widehat{e_j}\wedge\dots\wedge e_{n-2}\wedge e_j.\\
\text{Thus } \omega(T_{n}^{r,s},B/e_j)&=& (-1)^j \omega_j\rfloor \xi_{(B/e_j)_{\min}} = +e_j.
\end{eqnarray*}
On the other hand,
\begin{equation*}
p\left(\sum_j (-1)^{j+1}
\resizebox{!}{.8cm}{
 \begin{pspicture}(2,1.1)(4,2)
 \psline(3,.5)(3,1) \psline(3,1)(2,2)  \psline(3,1)(3.4,1.4) \psline(3.6,1.6)(4,2)
 \psline[linewidth=3pt](3.4,1.4)(3.6,1.6) \rput(3.7,1.3){$n$}
 \psline(3.8,1.8)(3.6,2)  \psline(3.6,1.6)(3.2,2)  \psline(3.4,1.4)(2.8,2)  \psline(3.2,1.2)(2.4,2)
 \end{pspicture} }\omega_j
 \right) = \sum_r \A_n^{r,0},
 \end{equation*}
with the sign calculated as follows: Denote by $B'$ and $B''$ the subtrees of $B_{e_j}$ such that $B_{e_j}=B'\circ_{e_j} B''$. Then $\ost_{B'}=e_1\wedge\dots\wedge e_{j-1}$ and $\ost_{B''}=e_{j+1}\wedge\dots\wedge e_{n-2}$ so that
\[ (B_{e_j},\fcan,m_{e_j},\omega_j)=(B',\fcan,m,\ost_{B'})\circ_{j+1} (B'',\fcan,m,\ost_{B''}) \]
Thus
\begin{eqnarray*}
(-1)^{j+1} p(B_{e_j},\fcan,m_{e_j},\omega_j)&=&(-1)^{j+1} p(B',\fcan,m,\ost_{B'})\circ_{j+1} p(B'',\fcan,m,\ost_{B''})\\
&=& (-1)^{j+1} (c',\fcan,+1)\circ_{j+1} (c'',\fcan, +1)\\
&=& (-1)^{j+1+(j+1)(n-j+1)+(j+1)(n-2-j)} \\
&& \quad \cdot (c'\circ_{e_j}c'', \fcan, +e_j)\\
&=& (c'\circ_{e_j}c'', \fcan, +e_j).
\end{eqnarray*}

\smallskip  
Next, we consider the case $B=(M_{k,l})_{\max}$. We will show that the following diagram commutes:
\[\xymatrix{
\resizebox{!}{1.5cm}{
 \begin{pspicture}(2,.5)(4.6,2)
 \rput(2.8,1.15){$_{e_1}$} \rput(3.45,1){$_{e_{k+1}}$} \rput(3.7,1.4){$_{\dots}$}
 \rput(3.2,1.7){$_{e_{k}}$} \rput(3.1,1.5){$_{.}$}\rput(3.1,1.4){$_{.}$}\rput(3.1,1.3){$_{.}$}
 \rput(4.2,1.6){$_{e_{k+l-1}}$}
 \psline[linewidth=3pt](3,.5)(3,2) \psline(3,1)(4,2)
 \psline(3,1.2)(2.2,2) \psline(3,1.4)(2.4,2) \psline(3,1.6)(2.6,2) \psline(3,1.8)(2.8,2)
 \psline(3.2,1.2)(3.2,2) \psline(3.4,1.4)(3.4,2) \psline(3.6,1.6)(3.6,2) \psline(3.8,1.8)(3.8,2)
 \end{pspicture}  }
\ar[r]^p \ar[d]_{\partial_Q} & 
\resizebox{!}{1.5cm}{
\begin{pspicture}(1,1)(3,3)
 \psline(2,2)(1.0,3)
 \psline(2,2)(1.2,3)
 \psline(2,2)(1.4,3)
 \psline(2,2)(1.6,3)
 \psline(2,2)(2.4,3)
 \psline(2,2)(2.6,3)
 \psline(2,2)(2.8,3)
 \psline(2,2)(3.0,3)
 \psline(2,2)(3.2,3)
 \psline[linewidth=3pt](2,1)(2,3)
\end{pspicture}
}
 \ar[d]^{\partial_C} \\
\partial_Q((M_{k,l})_{\max}) \ar[r]^{p\quad\quad \quad\quad} & 
\sum_{r,s} M_{k,l}^{r,s} +\sum_{r,s} N_{k,l}^{r,s} + \sum_{r,s} O_{k,l}^{r,s},
}\]
where $\ost_{(M_{k,l})_{\max}}=(-1)^{k}\cdot e_1\wedge\dots\wedge e_{k+l-1}$; and the module trees
\[
\begin{pspicture}(.3,0)(1,1) \rput(.5,1){$M_{k,l}^{r,s}$} \end{pspicture} \begin{pspicture}(0,0)(.3,1) \rput(.15,1){$=$} \end{pspicture}
\resizebox{!}{2.1cm}{
\begin{pspicture}(.8,0)(3.2,2)
 \psline[linewidth=3pt](2,0)(2,2) \psline(2,1.2)(1.8,1.8) \psline(2,1.2)(2.2,1.8)
 \psline(2,1.2)(1.6,1.8)  \psline(2,1.2)(2.4,1.8)
 \psline(2,.5)(1.2,1.2)  \psline(2,.5)(1.5,1.2) \psline(2,.5)(.9,1.2)
 \psline(2,.5)(2.5,1.2)  \psline(2,.5)(2.8,1.2) \psline(2,.5)(3.1,1.2)
\rput(2.77,1.5){$\overbrace{\quad\quad }^{s}$}
\rput(1.23,1.5){$\overbrace{\quad}^{r}$}
\end{pspicture} }
\]
\[
\begin{pspicture}(.3,0)(1,1) \rput(.5,1){$N_{k,l}^{r,s}$} \end{pspicture} \begin{pspicture}(0,0)(.3,1) \rput(.15,1){$=$} \end{pspicture}
\resizebox{!}{2.1cm}{
\begin{pspicture}(.4,0)(3.2,2)
 \psline[linewidth=3pt](2,0)(2,2)  \psline(2,.5)(1.8,1.2) \psline(2,.5)(.6,1.2)
 \psline(2,.5)(1.2,1.2)  \psline(2,.5)(1.6,1.1) \psline(2,.5)(.9,1.2)
 \psline(2,.5)(2.5,1.2)  \psline(2,.5)(2.8,1.2) \psline(2,.5)(3.1,1.2)
 \psline(1.6,1.1)(1.5,1.6)  \psline(1.6,1.1)(1.7,1.6)  \psline(1.6,1.1)(1.3,1.6)
\rput(1.5,1.9){$\overbrace{}^{s}$}
\rput(.93,1.5){$\overbrace{\quad\quad}^{r}$}
\end{pspicture} }
\quad\quad\quad\quad
\begin{pspicture}(.3,0)(1,1) \rput(.5,1){$O_{k,l}^{r,s}$} \end{pspicture} \begin{pspicture}(0,0)(.3,1) \rput(.15,1){$=$} \end{pspicture}
\resizebox{!}{2.1cm}{
\begin{pspicture}(.7,0)(3.6,2)
 \psline[linewidth=3pt](2,0)(2,2)  \psline(2,.5)(2.2,1.2) \psline(2,.5)(3.4,1.2)
 \psline(2,.5)(2.8,1.2)  \psline(2,.5)(2.4,1.1) \psline(2,.5)(3.1,1.2)
 \psline(2,.5)(1.5,1.2)  \psline(2,.5)(1.2,1.2) \psline(2,.5)(.9,1.2)
 \psline(2.4,1.1)(2.5,1.6)  \psline(2.4,1.1)(2.3,1.6)  \psline(2.4,1.1)(2.7,1.6)
\rput(2.5,1.9){$\overbrace{}^{s}$}
\rput(3.07,1.5){$\overbrace{\quad\quad}^{r}$}
\end{pspicture} }
\]
have orientation $+e$, where $e$ is the unique edge in the diagram. Now:
\begin{multline*}
 \begin{pspicture}(-2.4,.5)(4,2) 
  \rput(0,1.2){$\partial_Q((M_{k,l})_{\max}) =\sum_{e_j} (-1)^j\left(\begin{matrix}\\ \\ \\ \quad \end{matrix}\right.$}
 \psline[linewidth=3pt](3,.5)(3,2) \psline(3,1)(4,2) 
 \psline(3,1.2)(2.2,2) \psline(3,1.5)(2.4,2) \psline(3,1.5)(2.6,2) \psline(3,1.8)(2.8,2)
 \psline(3.2,1.2)(3.2,2) \psline(3.4,1.4)(3.4,2) \psline(3.6,1.6)(3.6,2) \psline(3.8,1.8)(3.8,2)
 \end{pspicture}
   \begin{pspicture}(1.4,.5)(5,2)
 \psline[linewidth=3pt](3,.5)(3,2) \psline(3,1)(4,2)\rput(1.8,1.2){$-$}
 \psline(3,1.2)(2.2,2) \psline(3,1.4)(2.4,2) \psline(3,1.6)(2.6,2) \psline(3,1.8)(2.8,2)
 \psline(3.2,1.2)(3.2,2) \psline(3.4,1.4)(3.4,2) \psline(3.6,1.6)(3.6,2) \psline(3.8,1.8)(3.8,2)
 \psline[linewidth=5pt](3,1.45)(3,1.7) \rput(3.2,1.7){$n$}
\rput(4.1,1.2){$ \left.\begin{matrix}\\ \\ \\ \quad \end{matrix}\right)$} 
 \end{pspicture}
  \\
   \begin{pspicture}(1,.5)(4,2)
 \psline[linewidth=3pt](3,.5)(3,2) \psline(3,1)(4,2) \rput(1.8,1.2){$+(-1)^1$}
 \psline(3,1)(2.2,2) \psline(3,1.4)(2.4,2) \psline(3,1.6)(2.6,2) \psline(3,1.8)(2.8,2)
 \psline(3.2,1.2)(3.2,2) \psline(3.4,1.4)(3.4,2) \psline(3.6,1.6)(3.6,2) \psline(3.8,1.8)(3.8,2)
\end{pspicture}
   \begin{pspicture}(1.6,.5)(4,2)
 \psline[linewidth=3pt](3,.5)(3,2) \psline(3,1)(4,2) \rput(1.8,1.2){$-(-1)^1$}
 \psline(3,1.2)(2.2,2) \psline(3,1.4)(2.4,2) \psline(3,1.6)(2.6,2) \psline(3,1.8)(2.8,2)
 \psline(3.2,1.2)(3.2,2) \psline(3.4,1.4)(3.4,2) \psline(3.6,1.6)(3.6,2) \psline(3.8,1.8)(3.8,2)
 \psline[linewidth=5pt](3,1.05)(3,1.3) \rput(2.7,1){$n$}
\end{pspicture}
 \begin{pspicture}(-.2,.5)(4,2) \rput(1.1,1.2){$+\sum_{e_j}(-1)^{j} \left(\begin{matrix}\\ \\ \\ \quad \end{matrix}\right.$}
 \psline[linewidth=3pt](3,.5)(3,2) \psline(3,1)(4,2) 
 \psline(3,1.2)(2.2,2) \psline(3,1.4)(2.4,2) \psline(3,1.6)(2.6,2) \psline(3,1.8)(2.8,2)
 \psline(3.2,1.2)(3.2,2) \psline(3.5,1.5)(3.4,2) \psline(3.5,1.5)(3.6,2) \psline(3.8,1.8)(3.8,2)
 \end{pspicture}
   \begin{pspicture}(1.8,.5)(4.6,2)
 \psline[linewidth=3pt](3,.5)(3,2) \psline(3,1)(4,2)\rput(2,1.2){$-$}
 \psline(3,1.2)(2.2,2) \psline(3,1.4)(2.4,2) \psline(3,1.6)(2.6,2) \psline(3,1.8)(2.8,2)
 \psline(3.2,1.2)(3.2,2) \psline(3.4,1.4)(3.4,2) \psline(3.6,1.6)(3.6,2) \psline(3.8,1.8)(3.8,2)
 \psline[linewidth=3pt](3.4,1.4)(3.6,1.6) \rput(3.7,1.3){$n$}
\rput(4.1,1.2){$ \left.\begin{matrix}\\ \\ \\ \quad\,\, \end{matrix}\right).$} 
 \end{pspicture}
 \end{multline*}
All of these terms have orientation $\omega_j=(-1)^k\cdot e_1\wedge\dots\wedge\widehat{e_j} \wedge\dots\wedge e_{k+l-1}$.
To complete this case, it remains to evaluate $p$:
\[
p\left( 
  \begin{pspicture}(1.1,1.2)(4.2,2)
 \psline[linewidth=3pt](3,.5)(3,2) \psline(3,1)(4,2) \rput(1.8,1.2){$\sum_j (-1)^j$}
 \psline(3,1.2)(2.2,2) \psline(3,1.5)(2.4,2) \psline(3,1.5)(2.6,2) \psline(3,1.8)(2.8,2)
 \psline(3.2,1.2)(3.2,2) \psline(3.4,1.4)(3.4,2) \psline(3.6,1.6)(3.6,2) \psline(3.8,1.8)(3.8,2)
 \end{pspicture}
\right) = \sum_r \sum_s N_{k,l}^{r,s}
\]
and
\[
p\left( 
\begin{pspicture}(.9,1.2)(4.2,2)
\rput(1.7,1.2){$\sum_j (-1)^{j+1}$}
 \psline[linewidth=3pt](3,.5)(3,2) \psline(3,1)(4,2)
 \psline(3,1.2)(2.2,2) \psline(3,1.4)(2.4,2) \psline(3,1.6)(2.6,2) \psline(3,1.8)(2.8,2)
 \psline(3.2,1.2)(3.2,2) \psline(3.4,1.4)(3.4,2) \psline(3.6,1.6)(3.6,2) \psline(3.8,1.8)(3.8,2)
 \psline[linewidth=5pt](3,1.45)(3,1.7) \rput(3.2,1.7){$n$}
 \end{pspicture}
\right) = \sum_{r>0} M_{k,l}^{r,l}.
\]
The sign for $B/e_j$ in the first equation follows as in $B=(T_n)_{\max}$ above. On the other hand $B_{e_j}=B'\circ_{e_j}B''$ with $\ost_{B'}=(-1)^{j-1}e_1\wedge \dots\wedge e_{j-1}\wedge e_{k+1}\wedge\dots\wedge e_{k+l-1}$ and $\ost_{B''}=e_{j+1}\wedge \dots\wedge e_{k}$, and we get:
\[ (B_{e_j},\fcan,m_{e_j},\omega_j)=(-1)^{k+j-1+(k-j)(l-1)}(B',\fcan,m,\ost_{B'})\circ_j (B'',\fcan,m,\ost_{B''}). \]
Thus
\begin{eqnarray*}
 (-1)^{j+1} p(B_{e_j},\fcan,m_{e_j},\omega_j)&=&(-1)^{kl+jl+j} p(B',\fcan,m,\ost_{B'})\circ_j p(B'',\fcan,m,\ost_{B''})\\
&=& (-1)^{kl+jl+j} (c',\fcan,+1)\circ_j (c'',\fcan, +1)\\
&=& (-1)^{kl+jl+j+j(k-j+2+1)+(l+j)(k-j)} \\
&&\quad \cdot (c'\circ_{e_j}c'', \fcan, +e_j)\\
&=& (c'\circ_{e_j}c'', \fcan, +e_j).
\end{eqnarray*}
Next,
\[
p\left( (-1)
  \begin{pspicture}(2,1.2)(4.2,2)
 \psline[linewidth=3pt](3,.5)(3,2) \psline(3,1)(4,2) 
 \psline(3,1)(2.2,2) \psline(3,1.4)(2.4,2) \psline(3,1.6)(2.6,2) \psline(3,1.8)(2.8,2)
 \psline(3.2,1.2)(3.2,2) \psline(3.4,1.4)(3.4,2) \psline(3.6,1.6)(3.6,2) \psline(3.8,1.8)(3.8,2)
\end{pspicture}
\right) = \sum_{r} M_{k,l}^{r,0} 
\]
and
\[
p\left( 
   \begin{pspicture}(2,1.2)(4.2,2)
 \psline[linewidth=3pt](3,.5)(3,2) \psline(3,1)(4,2) 
 \psline(3,1.2)(2.2,2) \psline(3,1.4)(2.4,2) \psline(3,1.6)(2.6,2) \psline(3,1.8)(2.8,2)
 \psline(3.2,1.2)(3.2,2) \psline(3.4,1.4)(3.4,2) \psline(3.6,1.6)(3.6,2) \psline(3.8,1.8)(3.8,2)
 \psline[linewidth=5pt](3,1.05)(3,1.3) \rput(2.7,1){$n$}
\end{pspicture}
\right) = M_{k,l}^{0,l}.
\]
This time $B_{e_1}=B'\circ_{e_1}B''$ with $\ost_{B'}= e_{k+1}\wedge\dots\wedge e_{k+l-1}$ and $\ost_{B''}=e_{2}\wedge \dots\wedge e_{k}$, and we get:
\[ (B_{e_1},\fcan,m_{e_1},\omega_1)=(-1)^{k+(k-1)(l-1)}(B',\fcan,m,\ost_{B'})\circ_1 (B'',\fcan,m,\ost_{B''}).\]
Thus
\begin{eqnarray*}
 p(B_{e_1},\fcan,m_{e_1},\omega_1)&=&(-1)^{kl+l+1} p(B',\fcan,m,\ost_{B'})\circ_1 p(B'',\fcan,m,\ost_{B''})\\
&=& (-1)^{kl+l+1} (c',\fcan,+1)\circ_1 (c'',\fcan, +1)\\
&=& (-1)^{kl+l+1+1(k+1+1)+(l+1)(k-1)} (c'\circ_{e_1}c'', \fcan, +e_1)\\
&=& (c'\circ_{e_1}c'', \fcan, +e_1).
\end{eqnarray*}
Furthermore,
\[
p\left( 
   \begin{pspicture}(1.1,1.2)(4.2,2)
 \psline[linewidth=3pt](3,.5)(3,2) \psline(3,1)(4,2) \rput(1.8,1.2){$\sum_j (-1)^j$}
 \psline(3,1.2)(2.2,2) \psline(3,1.4)(2.4,2) \psline(3,1.6)(2.6,2) \psline(3,1.8)(2.8,2)
 \psline(3.2,1.2)(3.2,2) \psline(3.5,1.5)(3.4,2) \psline(3.5,1.5)(3.6,2) \psline(3.8,1.8)(3.8,2)
 \end{pspicture}
 \right) = \sum_{r> 0} \sum_s O_{k,l}^{r,s} + \sum_{r}\sum_{0<s<l} M_{k,l}^{r,s} 
 \]
and
\[
p\left( 
 \begin{pspicture}(.9,1.2)(4.2,2) \rput(1.8,1.2){$\sum_j (-1)^{j+1}$}
 \psline[linewidth=3pt](3,.5)(3,2) \psline(3,1)(4,2)
 \psline(3,1.2)(2.2,2) \psline(3,1.4)(2.4,2) \psline(3,1.6)(2.6,2) \psline(3,1.8)(2.8,2)
 \psline(3.2,1.2)(3.2,2) \psline(3.4,1.4)(3.4,2) \psline(3.6,1.6)(3.6,2) \psline(3.8,1.8)(3.8,2)
 \psline[linewidth=3pt](3.4,1.4)(3.6,1.6) \rput(3.7,1.3){$n$}
 \end{pspicture}
\right) = \sum_s O_{k,l}^{0,s}.
\]
Write $B_{e_j}=B'\circ_{e_j} B''$. Then $\ost_{B'}=(-1)^k \cdot e_1\wedge\dots\wedge e_{j-1}$, and $\ost_{B''}=e_{j+1}\wedge\dots\wedge e_{k+l-1}$, so that
\[(B_{e_j},\fcan,m_{e_j},\omega_j)=(B',\fcan,m,\ost_{B'})\circ_{j+1} (B'',\fcan,m,\ost_{B''}).\]
Thus
\begin{eqnarray*}
 (-1)^{j+1} p(B_{e_j},\fcan,m_{e_j},\omega_j)&=&(-1)^{j+1} p(B',\fcan,m,\omega_{B'})\circ_{j+1} p(B'',\fcan,m,\omega_{B''})\\
&=& (-1)^{j+1} (c',\fcan,+1)\circ_{j+1} (c'',\fcan, +1)\\
&=& (-1)^{j+1+(j+1)(k+l+1-j+1)+(j+1)(k+l-1-j)} \\
&& \quad\cdot (c'\circ_{e_j}c'', \fcan, +e_j)\\
&=& (c'\circ_{e_j}c'', \fcan, +e_j).
\end{eqnarray*}

\smallskip
Finally, we consider $B=(I_{k,l})_{\max}$. We will show the following diagram commutes:
\[\xymatrix{
\resizebox{!}{1.5cm}{  \begin{pspicture}(1.8,.5)(4,2)  
 \psline[linewidth=2pt](2,1.25)(4,1.25) \pscircle*(3,1.25){.15} 
 \rput(3.5,1){$_{e_1 \dots e_{k}}$} \rput(2.5,1.5){$_{e_{k+l} \dots e_{k+1}}$}
\psline(3.2,1.25)(3.4,2)\psline(3.4,1.25)(3.6,2)\psline(3.6,1.25)(3.8,2)\psline(3.8,1.25)(4,2)
\psline(2.8,1.25)(2.6,.5)\psline(2.6,1.25)(2.4,.5)\psline(2.4,1.25)(2.2,.5)\psline(2.2,1.25)(2,.5)  \end{pspicture}}
\ar[r]^p \ar[d]_{\partial_Q} & 
\resizebox{!}{1.5cm}{
\begin{pspicture}(.7,1)(3.3,3)
 \psline(2,2)(1.4,3)
 \psline(2,2)(1.8,3)
 \psline(2,2)(2.2,3)
 \psline(2,2)(2.6,3)
 \psline(2,2)(1.4,1)
 \psline(2,2)(1.8,1)
 \psline(2,2)(2.2,1)
 \psline(2,2)(2.6,1)
 \psline[linewidth=3pt](.7,2)(3.3,2)
 \pscircle*(2,2){.15}
\end{pspicture}
}
 \ar[d]^{\partial_C} \\
\partial_Q((I_{k,l})_{\max}) \ar[r]^{p\quad\quad \quad\quad\quad \quad\quad} & 
\sum_{r,s} I_{k,l}^{r,s} +\sum_{r,s} J_{k,l}^{r,s} + \sum_{r,s} K_{k,l}^{r,s} + \sum_{r,s} L_{k,l}^{r,s},
}\]
where $\ost_{(I_{k,l})_{\max}}=(-1)^l\cdot e_1\wedge\dots\wedge e_{k+l}$ was calculated in Example \ref{EXA:I_kl-orientation}, and
\[
\begin{pspicture}(.3,0)(1,1) \rput(.5,1){$I_{k,l}^{r,s}$} \end{pspicture} \begin{pspicture}(0,0)(.3,1) \rput(.15,1){$=$} \end{pspicture}
\begin{pspicture}(-.2,1)(2.8,3) 
 \psline(1.6,1.4)(2.4,2.6)  \psline(2.4,1.4)(1.6,2.6) \psline(2,1.4)(2,2.6)
 \psline(1,2)(.8,2.6) \psline(1,2)(.6,2.6) \psline(1,2)(.4,2.6) 
 \rput(.6,2.8){$\overbrace{\quad}^{r}$}
 \psline(1,2)(.8,1.4) \psline(1,2)(.6,1.4) \psline(1,2)(.4,1.4)
 \rput(.6,1.2){$\underbrace{\quad}_{s}$}
 \psline[linewidth=3pt](0,2)(2.8,2)
 \pscircle*(2,2){.15}
\end{pspicture} 
\quad\quad\quad\quad
\begin{pspicture}(.3,0)(1,1) \rput(.5,1){$J_{k,l}^{r,s}$} \end{pspicture} \begin{pspicture}(0,0)(.3,1) \rput(.15,1){$=$} \end{pspicture}
\begin{pspicture}(.6,1)(4,3)  
 \psline(1.6,1.4)(2.4,2.6)  \psline(2.4,1.4)(1.6,2.6) \psline(2,1.4)(2,2.6)
 \psline(3,2)(3.2,2.6) \psline(3,2)(3.4,2.6) \psline(3,2)(3.6,2.6) 
 \rput(3.4,2.8){$\overbrace{\quad}^{r}$}
 \psline(3,2)(3.2,1.4) \psline(3,2)(3.4,1.4) \psline(3,2)(3.6,1.4)
 \rput(3.4,1.2){$\underbrace{\quad}_{s}$}
 \psline[linewidth=3pt](0.8,2)(4,2)
 \pscircle*(2,2){.15}
\end{pspicture} 
\]\[
\begin{pspicture}(.3,0)(1,1) \rput(.5,1){$K_{k,l}^{r,s}$} \end{pspicture} \begin{pspicture}(0,0)(.3,1) \rput(.15,1){$=$} \end{pspicture}
\begin{pspicture}(.6,1.4)(3,3.8)  
\psline[linewidth=3pt](1,2)(3,2) \pscircle*(2,2){.15}
\psline(2,2)(1.7,1.4)\psline(2,2)(1.9,1.4)\psline(2,2)(2.1,1.4)\psline(2,2)(2.3,1.4)
\psline(2,2)(1.2,2.6)\psline(2,2)(1.5,2.6)\psline(2,2)(1.7,2.6)
\psline(2,2)(2.3,2.6)\psline(2,2)(2.5,2.6)\psline(2,2)(2.7,2.6)
\psline(2,2)(2,2.5)
\psline(2,2.5)(1.7,3.2)\psline(2,2.5)(1.9,3.2)\psline(2,2.5)(2.1,3.2)\psline(2,2.5)(2.3,3.2)
 \rput(1.45,2.95){$\overbrace{\quad}^{r}$}
 \rput(2,3.55){$\overbrace{\quad}^{s}$}
\end{pspicture}
\quad\quad\quad\quad
\begin{pspicture}(.3,0)(1,1) \rput(.5,1){$L_{k,l}^{r,s}$} \end{pspicture} \begin{pspicture}(0,0)(.3,1) \rput(.15,1){$=$} \end{pspicture}
\begin{pspicture}(.6,0.2)(3,2.6)  
\psline[linewidth=3pt](1,2)(3,2) \pscircle*(2,2){.15}
\psline(2,2)(1.7,2.6)\psline(2,2)(1.9,2.6)\psline(2,2)(2.1,2.6)\psline(2,2)(2.3,2.6)
\psline(2,2)(1.7,1.4)\psline(2,2)(1.5,1.4)\psline(2,2)(1.3,1.4)
\psline(2,2)(2.5,1.4)\psline(2,2)(2.3,1.4)\psline(2,2)(2.7,1.4)
\psline(2,2)(2,1.5)
\psline(2,1.5)(1.7,.8)\psline(2,1.5)(1.9,.8)\psline(2,1.5)(2.1,.8)\psline(2,1.5)(2.3,.8)
 \rput(1.45,1.05){$\underbrace{\quad}_{r}$}
 \rput(2,.45){$\underbrace{\quad}_{s}$}
\end{pspicture}
\]
have orientation $+e$, with $e$ being the unique edge in the diagram. Now, $\partial_Q((I_{k,l})_{\max})$ equals
\begin{multline*}
 \scalebox{.85}[.85] {
   \begin{pspicture}(1.3,.5)(4,2) 
  \psline[linewidth=2pt](2,1.25)(4,1.25) \pscircle*(3,1.25){.13} \rput(1.4,1.2){$= (-1)$}
\psline(3,1.25)(3,2)\psline(3.4,1.25)(3.6,2)\psline(3.6,1.25)(3.8,2)\psline(3.8,1.25)(4,2)
\psline(2.8,1.25)(2.6,.5)\psline(2.6,1.25)(2.4,.5)\psline(2.4,1.25)(2.2,.5)\psline(2.2,1.25)(2,.5)
\end{pspicture}
   \begin{pspicture}(.8,.5)(4,2)
  \psline[linewidth=2pt](2,1.25)(4,1.25) \pscircle*(3,1.25){.13} \rput(1.4,1.2){$-(-1)$}
    \rput(3.2,1){$n$} \psline[linewidth=5pt](3,1.25)(3.25,1.25)
\psline(3.2,1.25)(3.4,2)\psline(3.4,1.25)(3.6,2)\psline(3.6,1.25)(3.8,2)\psline(3.8,1.25)(4,2)
\psline(2.8,1.25)(2.6,.5)\psline(2.6,1.25)(2.4,.5)\psline(2.4,1.25)(2.2,.5)\psline(2.2,1.25)(2,.5)
\end{pspicture}
 \begin{pspicture}(-.4,.5)(4,2) 
  \rput(1,1.2){$+\sum_{e_j} (-1)^{j} \left(\begin{matrix}\\ \\ \\ \quad \end{matrix}\right.$}
  \psline[linewidth=2pt](2,1.25)(4,1.25) \pscircle*(3,1.25){.13} 
\psline(3.2,1.25)(3.4,2)\psline(3.5,1.25)(3.6,2)\psline(3.5,1.25)(3.8,2)\psline(3.8,1.25)(4,2)
\psline(2.8,1.25)(2.6,.5)\psline(2.6,1.25)(2.4,.5)\psline(2.4,1.25)(2.2,.5)\psline(2.2,1.25)(2,.5)
 \end{pspicture}
   \begin{pspicture}(1.4,.5)(5,2) 
  \psline[linewidth=2pt](2,1.25)(4,1.25) \pscircle*(3,1.25){.13} \rput(1.7,1.2){$-$}
      \rput(3.5,1){$n$} \psline[linewidth=5pt](3.4,1.25)(3.65,1.25)
\psline(3.2,1.25)(3.4,2)\psline(3.4,1.25)(3.6,2)\psline(3.6,1.25)(3.8,2)\psline(3.8,1.25)(4,2)
\psline(2.8,1.25)(2.6,.5)\psline(2.6,1.25)(2.4,.5)\psline(2.4,1.25)(2.2,.5)\psline(2.2,1.25)(2,.5)
\rput(4.1,1.2){$ \left.\begin{matrix}\\ \\ \\ \quad \end{matrix}\right)$} 
 \end{pspicture}
 } \\
  \scalebox{.85}[.85] {
  \begin{pspicture}(.5,.5)(4,2)  
    \psline[linewidth=2pt](2,1.25)(4,1.25) \pscircle*(3,1.25){.13} \rput(1.2,1.2){$+(-1)^{k+1}$}
\psline(3.2,1.25)(3.4,2)\psline(3.4,1.25)(3.6,2)\psline(3.6,1.25)(3.8,2)\psline(3.8,1.25)(4,2)
\psline(3,1.25)(3,.5)\psline(2.6,1.25)(2.4,.5)\psline(2.4,1.25)(2.2,.5)\psline(2.2,1.25)(2,.5)
 \end{pspicture}
   \begin{pspicture}(.3,.5)(4,2)  
  \rput(2.8,1.5){$n$} \psline[linewidth=5pt](3,1.25)(2.75,1.25)
 \psline[linewidth=2pt](2,1.25)(4,1.25) \pscircle*(3,1.25){.13} \rput(1.2,1.2){$-(-1)^{k+1}$}
\psline(3.2,1.25)(3.4,2)\psline(3.4,1.25)(3.6,2)\psline(3.6,1.25)(3.8,2)\psline(3.8,1.25)(4,2)
\psline(2.8,1.25)(2.6,.5)\psline(2.6,1.25)(2.4,.5)\psline(2.4,1.25)(2.2,.5)\psline(2.2,1.25)(2,.5) 
\end{pspicture}
 \begin{pspicture}(-.4,.5)(4,2)  
  \rput(1,1.2){$+\sum_{e_j} (-1)^{j} \left(\begin{matrix}\\ \\ \\ \quad \end{matrix}\right.$}
   \psline[linewidth=2pt](2,1.25)(4,1.25) \pscircle*(3,1.25){.13} 
\psline(3.2,1.25)(3.4,2)\psline(3.4,1.25)(3.6,2)\psline(3.6,1.25)(3.8,2)\psline(3.8,1.25)(4,2)
\psline(2.8,1.25)(2.6,.5)\psline(2.5,1.25)(2.4,.5)\psline(2.5,1.25)(2.2,.5)\psline(2.2,1.25)(2,.5)
 \end{pspicture}
   \begin{pspicture}(1.4,.5)(5,2) 
  \psline[linewidth=2pt](2,1.25)(4,1.25) \pscircle*(3,1.25){.13} \rput(1.7,1.2){$-$}
    \rput(2.5,1.5){$n$} \psline[linewidth=5pt](2.35,1.25)(2.6,1.25)
\psline(3.2,1.25)(3.4,2)\psline(3.4,1.25)(3.6,2)\psline(3.6,1.25)(3.8,2)\psline(3.8,1.25)(4,2)
\psline(2.8,1.25)(2.6,.5)\psline(2.6,1.25)(2.4,.5)\psline(2.4,1.25)(2.2,.5)\psline(2.2,1.25)(2,.5)
\rput(4.1,1.2){$ \left.\begin{matrix}\\ \\ \\ \quad\,\, \end{matrix}\right).$} 
 \end{pspicture}
 }
 \end{multline*}
Again, all of these terms have orientation $\omega_j=(-1)^l\cdot e_1\wedge\dots\wedge\widehat{e_j} \wedge\dots\wedge e_{k+l-1}$. We evaluate $p$ as follows. First,
\[
p\left( (-1) 
 \begin{pspicture}(1.9,1.17)(4.1,2) 
  \psline[linewidth=2pt](2,1.25)(4,1.25) \pscircle*(3,1.25){.13} 
\psline(3,1.25)(3,2)\psline(3.4,1.25)(3.6,2)\psline(3.6,1.25)(3.8,2)\psline(3.8,1.25)(4,2)
\psline(2.8,1.25)(2.6,.5)\psline(2.6,1.25)(2.4,.5)\psline(2.4,1.25)(2.2,.5)\psline(2.2,1.25)(2,.5)
\end{pspicture}
\right) = \sum_{s} I_{k,l}^{1,s} 
\]
and
\[
p\left( 
 \begin{pspicture}(1.9,1.17)(4.1,2) 
  \psline[linewidth=2pt](2,1.25)(4,1.25) \pscircle*(3,1.25){.13} 
    \rput(3.2,1){$n$} \psline[linewidth=5pt](3,1.25)(3.25,1.25)
\psline(3.2,1.25)(3.4,2)\psline(3.4,1.25)(3.6,2)\psline(3.6,1.25)(3.8,2)\psline(3.8,1.25)(4,2)
\psline(2.8,1.25)(2.6,.5)\psline(2.6,1.25)(2.4,.5)\psline(2.4,1.25)(2.2,.5)\psline(2.2,1.25)(2,.5)
\end{pspicture}
\right) = J_{k,l}^{k,0}.
\]
Writing $B_{e_1}=B'\circ_{e_1} B''$ gives $\ost_{B'}=(-1)^l\cdot e_{k+1}\wedge\dots\wedge e_{k+l}$ and $\ost_{B''}=e_{2}\wedge\dots\wedge e_{k}$. This gives
\begin{eqnarray*}
(B_{e_1},\fcan,m_{e_1},\omega_1)&=&(-1)^{l(k-1)}(B',\fcan,m,\ost_{B'})\circ_{2} (B'',\fcan,m,\ost_{B''}).\\
\text{Thus }  p(B_{e_1},\fcan,m_{e_1},\omega_1)&=&(-1)^{l(k-1)} p(B',\fcan,m,\ost_{B'})\circ_{2} p(B'',\fcan,m,\ost_{B''})\\
&=& (-1)^{l(k-1)} (c',\fcan,+1)\circ_{2} (c'',\fcan, +1)\\
&=& (-1)^{l(k-1)+2(k+1+1)+l(k+1)} (c'\circ_{e_1}c'', \fcan, +e_1)\\
&=& (c'\circ_{e_1}c'', \fcan, +e_1).
\end{eqnarray*}
Second,
\[
p\left(  \sum_{j} (-1)^j 
 \begin{pspicture}(1.9,1.17)(4.1,2) 
  \psline[linewidth=2pt](2,1.25)(4,1.25) \pscircle*(3,1.25){.13} 
\psline(3.2,1.25)(3.4,2)\psline(3.5,1.25)(3.6,2)\psline(3.5,1.25)(3.8,2)\psline(3.8,1.25)(4,2)
\psline(2.8,1.25)(2.6,.5)\psline(2.6,1.25)(2.4,.5)\psline(2.4,1.25)(2.2,.5)\psline(2.2,1.25)(2,.5)
 \end{pspicture}
\right) =  \sum_{r\geq 2}\sum_{s} I_{k,l}^{r,s} +\sum_{r} \sum_{s} K_{k,l}^{r,s}
\]
and
\[
p\left( \sum_{j} (-1)^{j+1} 
 \begin{pspicture}(1.9,1.17)(4.1,2) 
  \psline[linewidth=2pt](2,1.25)(4,1.25) \pscircle*(3,1.25){.13} 
      \rput(3.5,1){$n$} \psline[linewidth=5pt](3.4,1.25)(3.65,1.25)  
\psline(3.2,1.25)(3.4,2)\psline(3.4,1.25)(3.6,2)\psline(3.6,1.25)(3.8,2)\psline(3.8,1.25)(4,2)
\psline(2.8,1.25)(2.6,.5)\psline(2.6,1.25)(2.4,.5)\psline(2.4,1.25)(2.2,.5)\psline(2.2,1.25)(2,.5)
 \end{pspicture}
\right) = \sum_{r<k} J_{k,l}^{r,0}. 
\]
Writing $B_{e_j}=B'\circ_{e_j} B''$ gives $\ost_{B'}=(-1)^l \cdot e_1\wedge\dots\wedge e_{j-1}\wedge e_{k+1}\wedge\dots\wedge e_{k+l}$, and $\ost_{B''}=e_{j+1}\wedge\dots\wedge e_{k}$. This gives,
\[(B_{e_j},\fcan,m_{e_j},\omega_j)=(-1)^{l(k-j)}(B',\fcan,m,\ost_{B'})\circ_{j+1} (B'',\fcan,m,\ost_{B''})\]
Thus
\begin{eqnarray*}
(-1)^{j+1} p(B_{e_j},\fcan,m_{e_j},\omega_j)&=&(-1)^{j+1+l(k+j)} \\
&&\quad \cdot p(B',\fcan,m,\ost_{B'})\circ_{j+1} p(B'',\fcan,m,\ost_{B''})\\
&=& (-1)^{j+1+l(k+j)} (c',\fcan,+1)\circ_{j+1} (c'',\fcan, +1)\\
&=& (-1)^{j+1+l(k+j)+(j+1)(k-j+2+1)+(l+j+1)(k-j)}\\
&&\quad \cdot (c'\circ_{e_j}c'', \fcan, +e_j)\\
&=& (c'\circ_{e_j}c'', \fcan, +e_j).
\end{eqnarray*}
Third,
\[
p\left( (-1)^{k+1} 
 \begin{pspicture}(1.9,1.17)(4.1,2) 
  \psline[linewidth=2pt](2,1.25)(4,1.25) \pscircle*(3,1.25){.13} 
\psline(3.2,1.25)(3.4,2)\psline(3.4,1.25)(3.6,2)\psline(3.6,1.25)(3.8,2)\psline(3.8,1.25)(4,2)
\psline(3,1.25)(3,.5)\psline(2.6,1.25)(2.4,.5)\psline(2.4,1.25)(2.2,.5)\psline(2.2,1.25)(2,.5)
\end{pspicture}
\right) = \sum_{r} J_{k,l}^{r,1} 
\]
and
\[
p\left( (-1)^k 
 \begin{pspicture}(1.9,1.17)(4.1,2) 
  \rput(2.8,1.5){$n$} \psline[linewidth=5pt](3,1.25)(2.75,1.25)
 \psline[linewidth=2pt](2,1.25)(4,1.25) \pscircle*(3,1.25){.13} 
\psline(3.2,1.25)(3.4,2)\psline(3.4,1.25)(3.6,2)\psline(3.6,1.25)(3.8,2)\psline(3.8,1.25)(4,2)
\psline(2.8,1.25)(2.6,.5)\psline(2.6,1.25)(2.4,.5)\psline(2.4,1.25)(2.2,.5)\psline(2.2,1.25)(2,.5)  \end{pspicture}
\right) = I_{k,l}^{0,l}.
\]
Writing $B_{e_{k+1}}=\sigma\cdot (B'\circ_{e_{k+1}} B'')$ gives $\ost_{B'}=e_1\wedge\dots\wedge e_{k}$, and $\ost_{B''}=e_{k+2}\wedge\dots\wedge e_{k+l}$, where $\sigma\in S_{k+l+2}$ is the cyclic permutation ``$+(l+1)\text{ (mod }k+l+2$)''. This gives,
\[ (B_{e_{k+1}},\fcan,m_{e_{k+1}},\omega_j)=(-1)^l\cdot \sigma\cdot ((B',\fcan,m,\ost_{B'})\circ_{1} (B'',\fcan,m,\ost_{B''})).\]
Thus
\begin{eqnarray*}
 (-1)^{k} p(B_{e_{k+1}},\fcan,m_{e_{k+1}},\omega_j)&=&(-1)^{k+l} \\
 && \quad \cdot \sigma\cdot (p(B',\fcan,m,\ost_{B'})\circ_{1} p(B'',\fcan,m,\ost_{B''}))\\
&=& (-1)^{k+l} \sigma\cdot ((c',\fcan,+1)\circ_{1} (c'',\fcan, +1))\\
&=& (-1)^{k+l+1(l+1+1)+(k+2)(l-1)} \\
&& \quad \cdot \sigma\cdot (c'\circ_{e_{k+1}}c'', \fcan\circ \sigma^{-1}, +e_{k+1})\\
&=& (c'\circ_{e_{k+1}}c'', \fcan, +e_{k+1}).
\end{eqnarray*}
Fourth,
\[
p\left(\sum_j (-1)^j 
 \begin{pspicture}(1.9,1.17)(4.1,2) 
   \psline[linewidth=2pt](2,1.25)(4,1.25) \pscircle*(3,1.25){.13}
\psline(3.2,1.25)(3.4,2)\psline(3.4,1.25)(3.6,2)\psline(3.6,1.25)(3.8,2)\psline(3.8,1.25)(4,2)
\psline(2.8,1.25)(2.6,.5)\psline(2.5,1.25)(2.4,.5)\psline(2.5,1.25)(2.2,.5)\psline(2.2,1.25)(2,.5)
 \end{pspicture}
\right) = \sum_{r}\sum_{s\geq 2} J_{k,l}^{r,s} +\sum_{r}\sum_{s} L_{k,l}^{r,s} 
\]
and
\[
p\left( \sum_j (-1)^{j+1} 
 \begin{pspicture}(1.9,1.17)(4.1,2) 
\psline[linewidth=2pt](2,1.25)(4,1.25) \pscircle*(3,1.25){.13} 
    \rput(2.5,1.5){$n$} \psline[linewidth=5pt](2.35,1.25)(2.6,1.25)
\psline(3.2,1.25)(3.4,2)\psline(3.4,1.25)(3.6,2)\psline(3.6,1.25)(3.8,2)\psline(3.8,1.25)(4,2)
\psline(2.8,1.25)(2.6,.5)\psline(2.6,1.25)(2.4,.5)\psline(2.4,1.25)(2.2,.5)\psline(2.2,1.25)(2,.5)
 \end{pspicture}
\right) = \sum_{s<l} I_{k,l}^{0,s}.
\]
Writing $B_{e_j}=\sigma\cdot (B'\circ_{e_j} B'')$ gives $\ost_{B'}=(-1)^{j-k-1} \cdot e_1\wedge\dots\wedge e_{j-1}$, and $\ost_{B''}=e_{j+1}\wedge\dots\wedge e_{k+l}$, where $\sigma\in S_{k+l+2}$ is the cyclic permutation ``$+(k+l+1-j)\text{ (mod }k+l+2$)''. This gives,
\[ (B_{e_j},\fcan,m_{e_j},\omega_j)=(-1)^{l+j-k-1}\sigma\cdot ((B',\fcan,m,\ost_{B'})\circ_{1} (B'',\fcan,m,\ost_{B''})).\]
Thus
\begin{eqnarray*}
 (-1)^{j+1} p(B_{e_j},\fcan,m_{e_j},\omega_j)&=&(-1)^{k+l}\\
 && \quad\cdot \sigma\cdot ( p(B',\fcan,m,\ost_{B'})\circ_{1} p(B'',\fcan,m,\ost_{B''}))\\
&=& (-1)^{k+l}\sigma\cdot ((c',\fcan,+1)\circ_{1} (c'',\fcan, +1))\\
&=& (-1)^{k+l+1\cdot (k+l-j+2+1)+(j+1)(k+l-j)}\\
&&\quad \cdot\sigma\cdot (c'\circ_{e_j}c'', \fcan.\sigma^{-1}, +e_j)\\
&=& (-1)^{(j+1)(k+l-j+1)} \sigma\cdot (c'\circ_{e_j}c'', \fcan.\sigma^{-1}, +e_j)\\
&=& (c'\circ_{e_j}c'', \fcan, +e_j).
\end{eqnarray*}

This completes the proof the lemma.
\end{proof}

\bibliographystyle{amsplain}

\end{document}